\documentclass{amsart}
\usepackage{amsmath,amscd,amsthm,amsfonts,a4wide,amssymb}
\usepackage[all,cmtip]{xy}
\usepackage{fancyhdr}
\theoremstyle{plain}
\newtheorem{theorem}                 {Theorem}      [section]
\newtheorem{proposition}  [theorem]  {Proposition}
\newtheorem{corollary}    [theorem]  {Corollary}
\newtheorem{lemma}        [theorem]  {Lemma}

\theoremstyle{definition}
\newtheorem{example}      [theorem]  {Example}
\newtheorem{remark}       [theorem]  {Remark}
\newtheorem{definition}   [theorem]  {Definition}

\newcommand{\be}[1]{\begin{equation}\label{#1}}
   \newcommand{\ee}{\end{equation}}

\numberwithin{equation}{section}

\DeclareMathOperator{\Ima}{Im}
\DeclareMathOperator{\spa}{span}

\DeclareMathOperator{\Hom}{Hom}
\DeclareMathOperator{\rank}{rank}
\DeclareMathOperator{\pr}{\pi}
\DeclareMathOperator{\Tr}{trace}

\newcommand{\wt}{\widetilde}
\newcommand{\wh}{\widehat}
\newcommand{\pa}{\partial}

\def \nn{\mathbb N}
\def \zn{\mathbb Z}
\def \rn{\mathbb R}
\def \cn{\mathbb C}
\def \hn{\mathbb H}
\def \H{\mathcal H}
\def \Z{\mathcal Z}

\def \HH{\underline{\mathcal H}}

\def \Gr{Gr}

\def \g{\mathfrak{g}}

\def \O#1{{\rm O}(#1)}
\def \SO#1{{\rm SO}(#1)}

\def \U#1{{\rm U}(#1)}

\def \Sp#1{{\rm Sp}(#1)}

\def\ip#1#2{\langle#1,#2\rangle}

\def \CP#1{\mathbb{C}P^{#1}}
\def \RP#1{\mathbb{R}P^{#1}}
\def \HP#1{\mathbb{H}P^{#1}}

\def \d{\mathrm{d}}

\newcommand{\zbar}{\bar{z}}
\newcommand{\CC}{\underline{\mathbb C}}
\newcommand{\ov}{\overline}
\newcommand{\ul}{\underline}
\newcommand{\alg}{{\rm alg}}

\newcommand{\Hh}{\mathcal{H}} %horizontal bundle
\newcommand{\Vv}{\mathcal{V}} %verticaal bundle

\newcommand{\eu}{\mathrm{e}} %Euler's number
\newcommand{\ii}{\mathrm{i}} %sqrt{-1}

\begin{document}
%\baselineskip 18pt \larger[1]
%\begin{footnotesize}
%\begin{flushright}
%CP3-ORIGINS-2011-19 \& DIAS-2011-05
%\end{flushright}
%\end{footnotesize}
\bigskip

\title{New constructions of twistor lifts for harmonic maps}

\author{Martin Svensson} 
\author{John C. Wood}



\thanks{The first author was supported by the Danish Council for Independent Research under the project \emph{Symmetry Techniques in Differential Geometry}.
The second author thanks the Department of Mathematics and Computer Science of the University of Southern Denmark, Odense, for support and hospitality during part of the preparation of this work.}

\address
{Department of Mathematics \& Computer Science, University of
Southern Denmark, and CP3--Origins, Centre of Excellence for Particle Physics and Phenomenology, Campusvej 55, DK-5230 Odense M, Denmark}
\email{svensson@imada.sdu.dk}

\address
{Department of Pure Mathematics, University of Leeds, Leeds LS2 9JT, Great Britain}
\email{j.c.wood@leeds.ac.uk}

\begin{abstract}  
We show that given a harmonic map $\varphi$ from a Riemann surface to a
classical compact simply connected inner symmetric space, there is a $J_2$-holomorphic twistor lift of 
$\varphi$ (or its negative) \emph{if and only if}\/ it is nilconformal.
In the case of harmonic maps of finite uniton number, we give algebraic formulae in terms of holomorphic data which describes their extended solutions.  In particular, this gives explicit formulae for the twistor lifts of all harmonic maps of finite uniton number from a surface to the above symmetric spaces.
\end{abstract}

\maketitle

\section{Introduction}

\emph{Harmonic maps} are smooth maps between Riemannian manifolds which extremize the `Dirichlet' energy integral (see, for example, \cite{eells-lemaire}).
Harmonic maps from surfaces to symmetric spaces are of particular interest to both geometers, as they include minimal surfaces, and to theoretical physicists, as they constitute the non-linear $\sigma$-model of particle physics. 
Twistor methods for finding such harmonic maps have been around for a long time; a general theory was given by F.~E.\ Burstall and J.~H.\ Rawnsley  \cite{burstall-rawnsley}, see also \cite{davidov-sergeev}.  The idea is to find a \emph{twistor fibration (for harmonic maps)} --- this is a fibration $Z \to N$ from an almost complex manifold $Z$, called a \emph{twistor space}, to a Riemannian manifold $N$ with the property that holomorphic maps from (Riemann) surfaces to $Z$ project to harmonic maps to $N$.
For a symmetric space $N$, twistor spaces exist if $N$ is \emph{inner} \cite{burstall-rawnsley}, i.e., has inner Cartan involution; then they are generalized flag manifolds equipped with a certain non-integrable complex structure $J_2$.
All harmonic maps from the $2$-sphere arise this way, i.e., have a \emph{twistor lift} to a suitable flag manifold, see \cite{burstall-rawnsley}.

Burstall \cite{burstall-grass} showed that, given a harmonic map $\varphi$ from a surface to a complex Grassmannian, there is a twistor lift of $\varphi$ or its orthogonal complement $\varphi^{\perp}$ if and only if $\varphi$ is \emph{nilconformal} in the sense that its derivative is nilpotent. 
We extend this result to other classical symmetric spaces as follows. 

\begin{theorem} \label{th:main}
Let $\varphi$ be a harmonic map from a Riemann surface to a classical compact simply connected inner symmetric space.  Then there is a twistor lift of\/ $\varphi$ or $-\varphi$ if and only if\/ $\varphi$ is nilconformal.
\end{theorem}

Any classical compact simply connected inner symmetric space is the product of 
irreducible ones: it suffices to prove our result for these.  According to
\cite[p.\ 38]{burstall-rawnsley} they are
(i) the complex Grasmannians $G_k(\cn^n)$,
(ii) the oriented real Grassmannians $\wt G_k(\rn^n)$ with
$k(n-k)$ even, 
(iii) the space $\SO{2m} \big/ \U{m}$ of (positive) orthogonal complex structures on $\rn^{2m}$,
 (iv)  quaternionic Grassmannians $G_k(\hn^n)$, and
(v) the space $\Sp{m} \big/ \U{m}$ of `quaternionic' complex structures on $\cn^{2m}$.  In the Grassmannian cases (i), (ii) and (iv), $-\varphi$ means the map $\varphi^{\perp}: p \mapsto \varphi(p)^{\perp}$.

To establish our result, we introduce the idea of \emph{$A^{\varphi}_z$-filtrations}, first of all for harmonic maps to complex Grassmannians, and show how these are related to twistor lifts, see
Proposition \ref{pr:J2-from-splits}.
(In fact, we can find such filtrations incorporating any given uniton, giving us the existence of twistor lifts associated to that uniton, see Theorem \ref{th:lift-with-uniton}.)   Then we adapt our technique to the `real' cases, showing in a constructive way how to build twistor lifts of harmonic maps from a surfaces to a real Grassmannian $\wt G_k(\rn^n)$ with
$k(n-k)$ even, and to the space $\O{2m} \big/ \U{m}$, see Propositions \ref{th:lift-with-uniton-real} and \ref{th:lift-OCS}.  Similar results hold for maps to quaternionic projective space and to the space $\Sp{m} \big/ \U{m}$, see \S \ref{sec:sympl}.  Putting these results together gives our theorem.

Note that harmonic maps to oriented real Grassmannians with $k(n-k)$ odd can be dealt with by embedding them in higher-dimensional Grassmannians (see Remark
\ref{re:odd-even}), and that harmonic maps to Grassmannians $G_k(\rn^n)$ of \emph{unoriented} real subspaces are covered by those to oriented ones if a certain Steifel--Whitney class vanishes (Remark \ref{re:unoriented}). 
	
Nilconformal harmonic maps include all harmonic maps of finite uniton number (Example \ref{ex:can-filt}).  They also include \emph{strongly conformal} harmonic maps, in particular the superconformal harmonic maps from the plane or a torus studied in \cite{bolton-woodward,bolton-pedit-woodward}, see
Example \ref{ex:superconf}; such maps from a torus are of finite type but not of finite uniton number \cite{pacheco-tori}.  

In the case that $\varphi$ has an \emph{extended solution} $\Phi$ (always true locally), we show how $A^{\varphi}_z$-filtrations are equivalent to certain other filtrations, called \emph{$F$-filtrations}, of G.\ Segal's Grassmannian model \cite{segal} of $\Phi$. We can then compute the twistor lift from the  $F$-filtration and $\Phi$.    We identify the $F$-filtration which gives Burstall's twistor lift.
When $\varphi$ has finite uniton number, we may choose $\Phi$ to be polynomial; in that case,  we have a natural $F$-filtration which leads to a new twistor lift, which we call the \emph{canonical twistor lift}, see Theorem \ref{th:can-lift}.
Again, we can adapt these techniques to find twistor lifts of harmonic maps to
the other classical compact simply connected inner symmetric spaces, see Corollaries \ref{co:real-twistor-lifts} and \ref{co:twistor-lifts-OCS}, and \S \ref{sec:sympl}.

In the case that $\varphi$ has \emph{finite uniton number} we can do these constructions \emph{explicitly}, as follows, see \S \ref{sec:explicit}.
In \cite{ferreira-simoes-wood}, simple formulae for the unitons of the factorization due to G.\ Segal were found,  thus giving explicit algebraic formulae for all harmonic maps of finite uniton number from a Riemann surface to the unitary group and complex Grassmannians, not involving any integration.  Such formulae for K.\ Uhlenbeck's factorization \cite{uhlenbeck} --- which is dual to that of Segal --- appeared in \cite{dai-terng}.
In \cite{unitons}, it was shown how these formulae are extreme cases of a general method of finding explicit formulae for uniton factorizations, and the method was adapted to construct harmonic maps to the orthogonal and symplectic groups, the real and quaternionic Grassmannians, and the spaces $\SO{2m} \big/ \U{m}$ and $\Sp{m} \big/ \U{m}$, thus finding all harmonic maps to classical Lie groups and their inner symmetric spaces explicitly in terms of algebraic data.

In the present paper, we use those formulae to obtain explicit algebraic formulae for the $J_2$-holomorphic twistor lifts of arbitrary harmonic maps of finite uniton number from a Riemann surface to a complex Grassmannian in terms of the freely chosen holomorphic data
which give the unitons of the harmonic map.    We then find the algebraic conditions on the holomorphic data which give the twistor lifts of harmonic maps to
real and quaternionic Grassmannians, and to the spaces $\SO{2m} \big/ \U{m}$ and $\Sp{m} \big/ \U{m}$.   In particular, this gives explicit formulae for all harmonic maps from the two-sphere to the classical compact simply connected inner symmetric spaces and their twistor lifts.

We thank the referee for some very useful comments.

\section{preliminaries} \label{sec:prelims}

\subsection{Harmonic maps from surfaces to a Lie group} \label{subsec:Lie-groups}
Throughout the paper, all manifolds, bundles, and structures on them will be taken to be smooth, i.e., $C^{\infty}$.  By `surface' we shall mean `Riemann surface', i.e., `connected $1$-dimensional complex manifold'; we do not assume compactness.  Harmonic maps from surfaces exhibit conformal invariance (see, for example, \cite{wood60}) so that the concept of harmonic map from a Riemann surface is well defined.  In the case of maps from a Riemann surface $M$
to a Lie group $G$, we can formulate the harmonicity equations in the following way \cite{uhlenbeck,guest-book}.

For any smooth map $\varphi:M\to G$, set $A^{\varphi}=\frac{1}{2}\varphi^{-1}\d\varphi$;
thus $A^{\varphi}$ is a $1$-form with values in the Lie algebra $\g$ of $G$;
note that it is half the pull-back of the Maurer--Cartan form of $G$. 

Now, any compact Lie group can be embedded in the unitary group
$\U n$, so we first consider that group.  The group $\U n$ acts on $\cn^n$ in the standard way.  Let $\CC^n$ denote the trivial complex bundle $\CC^n = M \times \cn^n$, then
$D^{\varphi} = \d+A^{\varphi}$
defines a unitary connection on $\CC^n$.
We decompose $A^{\varphi}$ and $D^{\varphi}$ into types; for convenience we do this by taking a
local complex coordinate $z$ on an open set $U$ of $M$.  Explicitly, on writing
$\d\varphi = \varphi_z \d z + \varphi_{\zbar}\d\zbar$,
$A = A^{\varphi}_z \d z +  A^{\varphi}_{\zbar} \d\zbar$,
$D^{\varphi} = D^{\varphi}_z \d z + D^{\varphi}_{\zbar} \d\zbar$,
$\pa_z = \pa/\pa z$ and $\pa_{\zbar} = \pa/\pa\zbar$, we have
\be{type-decomp}
A^{\varphi}_z=\frac{1}{2}\varphi^{-1}\varphi_z\,,\quad A^{\varphi}_{\zbar}=\frac{1}{2}\varphi^{-1}\varphi_{\zbar}\,, \quad
D^{\varphi}_z = \pa_z + A^{\varphi}_z \,,\quad
D^{\varphi}_{\zbar} = \pa_{\zbar} + A^{\varphi}_{\zbar} \,.
\ee

By the \emph{(Koszul--Malgrange) holomorphic structure \cite{koszul-malgrange} induced by $\varphi$} we mean the unique holomorphic structure on $\CC^n$ with $\bar{\pa}$-operator given on each coordinate domain $(U,z)$ by $D^{\varphi}_{\bar z}$;  we denote the resulting holomorphic vector bundle by $(\CC^n, D^{\varphi}_{\bar z})$.   If $\varphi$ is constant, then $D^{\varphi}_{\bar z} = \pa_{\zbar}$ giving $\CC^n$ the
\emph{standard (product) holomorphic structure}. 
Uhlenbeck \cite{uhlenbeck} showed that \emph{a smooth map $\varphi:M\to G$ is harmonic
if and only if, on each coordinate domain, $A^{\varphi}_z$ is a holomorphic endomorphism of the holomorphic vector bundle $(\CC^n, D^{\varphi}_{\bar z})$}. 
For later use, note that, if $\varphi$ is replaced by $g\varphi$ for some $g \in \U{n}$, then
all the quantities in \eqref{type-decomp} are unchanged.

Let $\nn = \{0,1,2,\ldots\}$. For any $n \in \nn$ and $k \in \{0,1,\ldots,n\}$, let $G_k(\cn^n)$ denote the Grassmannian of
$k$-dimensional subspaces of $\cn^n$; it is convenient to write
 $G_*(\cn^n)$ for the disjoint union
$\cup_{k=0,1,\ldots,n}G_k(\cn^n)$.
We shall often identify, without comment, a smooth map
$\varphi:M \to G_k(\cn^n)$ with the rank $k$ subbundle 
of $\CC^n = M \times \cn^n$ whose fibre at $p \in M$ is $\varphi(p)$; we denote this subbundle
also by $\varphi$, not underlining
this as in, for example, \cite{burstall-wood, ferreira-simoes, ferreira-simoes-wood}.

For a subspace $V$ of $\cn^n$ we denote by $\pi_V$
(resp.\ $\pi_V^{\perp}$) orthogonal projection from $\cn^n$ to $V$
(resp.\ to its orthogonal complement $V^{\perp}$); we use the same notation for orthogonal projection from $\CC^n$ to a subbundle.
Recall the Cartan embedding: 
\begin{equation} \label{cartan}
\iota:G_*(\cn^n)\hookrightarrow \U n,\quad \iota(V)=\pi_V-\pi_{V}^{\perp}\,;
\end{equation}
this is totally geodesic, and isometric up to a constant factor.
We shall identify $V$ with its image $\iota(V)$; since $\iota(V^{\perp}) = -\iota(V)$,
this identifies $V^{\perp}$ with $-V$.
 
\subsection{Harmonic maps from surfaces to complex Grassmannians} \label{subsec:Grass}

Harmonic maps from surfaces to  Grassmannians were studied by Burstall and the second author in \cite{burstall-wood} where the following definitions were made.  Any subbundle $\varphi$ of $\CC^n$ inherits a metric by restriction, and a connection $\nabla_{\!\varphi}$ by orthogonal projection:
$$
(\nabla_{\!\varphi})_Z(v) = \pi_{\varphi}(\pa_Z v)
	\qquad (Z \in \Gamma(TM), \ v \in \Gamma(\varphi)\,);
$$
here $\Gamma(\cdot)$ denotes the space of (smooth) sections of a vector bundle.

Let $\varphi$ and $\psi$ be two mutually orthogonal subbundles of
$\CC^n$.  Then, by the
\emph{$\pa'$ and $\pa''$-second fundamental forms of\/ $\varphi$ in $\varphi \oplus \psi$} we mean the vector bundle morphisms
$A'_{\varphi,\psi}\,, A''_{\varphi,\psi}:\varphi \to \psi$
defined on each coordinate domain $(U,z)$ by
\be{2nd-f-f}
A'_{\varphi,\psi}(v) = \pi_{\psi}(\pa_z v) \quad \text{and} \quad
A''_{\varphi,\psi}(v) = \pi_{\psi}(\pa_{\zbar}v)
	\qquad (v \in \Gamma(\varphi)\,).
\ee

The second fundamental forms
$A'_{\varphi} = A'_{\varphi,\varphi^{\perp}}:\varphi \to \varphi^{\perp}$ and
$A''_{\varphi} = A''_{\varphi,\varphi^{\perp}}:\varphi \to \varphi^{\perp}$
are particularly important as, on identifying $\varphi: M \to G_*(\cn^n)$ with its composition
$\iota\circ\varphi:M \to \U n$ with the Cartan embedding, it is easily seen that
the fundamental endomorphism
$A^{\varphi}_z$ of \eqref{type-decomp} is minus the direct sum of $A'_{\varphi}$
and $A'_{\varphi^{\perp}}$, similarly the connection
$D^{\varphi}$ of the last section is the direct sum of
$\nabla_{\!\varphi}$ and $\nabla_{\!\varphi^{\perp}}$.
It follows that \emph{a smooth map $\varphi:M \to G_*(\cn^n)$ is
harmonic if and only if\/ $A'_{\varphi}$ is holomorphic, i.e.,
$A'_{\varphi} \circ \nabla''_{\varphi}
	= \nabla''_{\varphi^{\perp}} \circ A'_{\varphi}$}\,,
where we write $\nabla''_{\varphi} = (\nabla_{\!\varphi})_{\pa/\pa\zbar}$;
this can be shown without reference to $\U n$, see
\cite[Lemma 1.3]{burstall-wood}.

Now, for any holomorphic (or antiholomorphic) endomorphism $E$, at points where it
 does not have maximal rank, we shall `fill out zeros'  as in
\cite[Proposition 2.2]{burstall-wood} (cf.\ \cite[\S 3.1]{unitons})
to make its image and kernel into subbundles $\Ima E$ and $\ker E$ of $\CC^n$.
In particular, we obtain subbundles $G'(\varphi) = \Ima A' _{\varphi}$
and $G''(\varphi) = \Ima A''_{\varphi}$ called the
\emph{$\pa'$- and $\pa''$-Gauss transforms} or \emph{Gauss bundles of\/ $\varphi$}.  Note that,
\emph{if\/ $\varphi$ is harmonic, then so are its Gauss transforms.}  This can be seen
by using diagrams as in \cite[Proposition 2.3]{burstall-wood}, or by noting that it is a special case of adding a uniton, cf.\ \cite{wood60}.

By iterating these constructions we obtain the \emph{$i$th $\pa'$-Gauss transform}
 $G^{(i)}(\varphi)$ defined by $G^{(1)}(\varphi) = G'(\varphi)$,  $G^{(i)}(\varphi) = G'(G^{(i-1)}(\varphi))$, and the \emph{$i$th $\pa''$-Gauss transform} $G^{(-i)}(\varphi)$
defined by $G^{(-1)}(\varphi) = G''(\varphi)$,  $G^{(-i)}(\varphi) = G''(G^{(-i+1)}(\varphi))$; on setting $G^{(0)}(\varphi) = \varphi$,
we obtain a sequence $(G^{(i)}(\varphi))_{i \in \zn}$ (where $\zn$ denotes the set of integers)
of harmonic maps called \cite{wolfson} the
\emph{harmonic sequence of\/ $\varphi$}.

\section{Twistor spaces and lifts} \label{sec:twistor-spaces}

\subsection{Twistor spaces of complex Grassmannians} \label{subsec:twistor-Grass}
Let $N$ be a Riemannian manifold.  By a \emph{twistor fibration of $N$
(for harmonic maps)} is meant
\cite{burstall-rawnsley} an almost complex manifold (called a \emph{twistor space}) 
$(Z,J)$ and a fibration $\pi:Z \to N$ such that, for every holomorphic map from a Riemann surface
$\psi:M \to Z$, the composition $\varphi = \pi \circ \psi:M \to N$ is harmonic.
(To deal with higher-dimensional domains, the definition is unchanged if we replace `Riemann surface' by `cosymplectic manifold', i.e., `almost Hermitian manifold with co-closed
K\"ahler form' \cite{burstall-rawnsley,salamon}.)
Then $\varphi$ is called the  \emph{twistor projection of\/ $\psi$}, and $\psi$ is called
\emph{a twistor lift of\/ $\varphi$}.    For an even-dimensional Riemannian manifold $N$, the primary example is the bundle $Z \to N$ of almost Hermitian structures whose fibre at $q \in N$ consists of all \emph{orthogonal complex structures}, i.e.\ \emph{almost complex structures compatible with the metric}. When $Z$ is equipped with a suitable non-integrable almost complex structure $J_2$, the bundle $Z \to N$ becomes a twistor fibration for harmonic maps, see \cite{eells-salamon} and
\cite[Chapter 2]{burstall-rawnsley}.
If $N$ is orientable, we may consider the subbundle $Z^+ \to N$ of \emph{positive} almost Hermitian structures; however, $Z$ and $Z^+$ are usually too large to be useful and we look for subbundles of them.

For symmetric spaces, a general theory of such
twistor fibrations is given in \cite{burstall-rawnsley}.  We shall now describe those twistor spaces for a complex Grassmannian; for the real and symplectic cases, see \S \ref{sec:real}\textit{ff.}  For any complex vector spaces or vector bundles $E$, $F$,
$\Hom(E,F) =  \Hom_{\cn}(E,F)$ will denote the vector space or bundle of (complex-)linear maps from $E$ to $F$.

Let $n,t,d_0,d_1, \ldots, d_t$ be positive integers with $\sum_{i=0}^t d_i = n$. Let $F=F_{d_0,\ldots, d_t}$ be the (complex) flag manifold
$\U n \big/ \U{d_0}\times\dots\times\U{d_t}$. The elements of $F$ are $(t+1)$-tuples $\psi = (\psi_0,\psi_1, \ldots, \psi_t)$ of mutually orthogonal subspaces with $\psi_0\oplus\dots\oplus \psi_t=\cn^n$; we call these subspaces the \emph{legs (of\/ $\psi$)}. There is a canonical embedding of $F$ into the product $G_{d_0}(\cn^n)\times\dots\times G_{d_0+\dots+d_{t-1}}(\cn^n)$ given by sending 
$(\psi_0,\psi_1, \ldots,\psi_t)$ to its \emph{associated flag} $(T_0,\dots,T_{t-1})$ where $T_i=\sum_{j=0}^i\psi_j$\,;
the restriction to $F$ of the K\"ahler structure on this product is an (integrable) complex structure
which we denote by $J_1$.
Then the $(1,0)$-tangent space to $F$ at $(\psi_0,\psi_1, \ldots, \psi_t)$ with respect to $J_1$ is given by
\be{T10J1} 
T^{J_1}_{(1,0)}F = \sum_{0\leq i<j\leq t}\Hom(\psi_i,\psi_j).
\ee
Set $k = \sum_{j=0}^{[t/2]}d_{2j}$ and $N = G_k(\cn^n)$. We define
a mapping which gives the sum of the `even' legs:
\be{twistor-proj}
\pi_e: F_{d_0,\ldots, d_t} \to G_k(\cn^n),\qquad \psi = (\psi_0,\psi_1, \ldots,\psi_t)\mapsto\sum_{j=0}^{[t/2]}\psi_{2j}.
\ee

The projection $\pi_e$ is a Riemannian submersion with respect to the natural homogeneous metrics on $F$ and $G_k(\cn^n)$ so that each tangent space decomposes into the orthogonal direct sum of the \emph{vertical space}, made up of the tangents to the fibres and the
\emph{horizontal space}, its orthogonal complement.
The $(1,0)$-horizontal and vertical spaces with respect to $J_1$ are given by
$$
\Hh^{J_1}_{(1,0)}
= \sum_{\substack{i,j=0,\dots,t \\ i<j,\ j-i \text{ odd}}}\Hom(\psi_i,\psi_j)\,,
\qquad
\Vv^{J_1}_{(1,0)} =
\sum_{\substack{i,j=0,\dots,t\\ i<j,\ j-i \text{ even}}} \Hom(\psi_i,\psi_j).
$$ 

We define an almost complex structure $J_2$ by changing the sign of $J_1$ on the vertical space;
thus the $(1,0)$-horizontal space is unchanged, but the $(1,0)$ vertical space is different (unless $t=1$ when it is zero): 
$$
\Hh^{J_2}_{(1,0)}
= \sum_{\substack{i,j=0,\dots,t \\ i<j,\ j-i \text{ odd}}}\Hom(\psi_i,\psi_j)\,,
\qquad
\Vv^{J_2}_{(1,0)}=\sum_{\substack{i,j=0,\dots,t\\ j < i,\ j-i \text{ even}}} \Hom(\psi_i,\psi_j).
$$
This almost complex structure is never integrable except in the trivial case $t=1$,
see \cite{burstall-salamon}.   

Now let $M$ be a Riemann surface, which we shall always assume connected,
but not necessarily compact, and
let $\psi = (\psi_0,\psi_1, \ldots, \psi_t): M \to F$ be a smooth map; we shall call such a map a \emph{moving flag}.  
{}From the above description we immediately obtain the following \cite{burstall-grass}.

\begin{proposition} \label{pr:J1-J2}
Let $\psi = (\psi_0,\psi_1, \ldots, \psi_t):M \to F$ be a smooth map.
Then
\begin{enumerate}
\item[(i)] $\psi$ is $J_1$-holomorphic if and only if\/ $A'_{\psi_i, \psi_j} = 0$ whenever $i-j$ is positive$;$
\item[(ii)] $\psi$ is $J_2$-holomorphic if and only if
\begin{equation}\label{J2-cond}
A'_{\psi_i, \psi_j} = 0\text{ when $i-j$ is positive and odd, or $j-i$ is positive and even}.
\qed
\end{equation}
\end{enumerate}
\end{proposition}

\begin{remark}  \label{rem:aH-strs}
(i) Using Proposition \ref{pr:J1-J2}, it can be shown that  \emph{$\pi_e:(F,J_2) \to G_*(\cn^n)$ is a twistor fibration for harmonic maps}.
This is established in \cite[Corollary 5.10]{burstall-rawnsley}; 
see \cite{burstall-salamon} for a direct proof.   By noting that $F$ can be embedded in the bundle of all almost Hermitian structures preserving $J_1$ and $J_2$, we can also deduce it from  \cite[Theorem 3.5]{salamon} or \cite[Theorem 5.5]{rawnsley}.

(ii) Let $y \in G_k(\cn^n)$.  For each $\psi$ in the fibre $(\pi_e)^{-1}(y)$, the differential $(\d\pi_e)_{\psi}$ of the twistor projection at $\psi$  restricts to an isometry of the horizontal space at $\psi$ to $T_yG_k(\cn^n)$.  We can use this to transfer
the almost complex structure $J_1|_{\Hh} = J_2|_{\Hh}$ on the horizontal space to an almost Hermitian structure on $T_y G_k(\cn^n)$. This defines an inclusion map $i$ of $F$ in the bundle $Z \to N$ of almost Hermitian structures on $G_k(\cn^n)$, see \cite{rawnsley}, showing how we may regard $F$ as a subbundle of that bundle. 

If now, $\psi:M \to F$ is a $J_1$ or $J_2$-holomorphic map, for each $p \in M$, the differential of $\varphi$ at $p$ intertwines the almost complex structure of $M$ at $p$ with
the almost complex structure $i \circ \psi(p)$ on  $T_{\varphi(p)} G_k(\cn^n)$; thus
the map $\varphi$ is `rendered holomorphic'.  
\end{remark}

\subsection{$J_2$-holomorphic lifts of nilconformal maps from $A_z^{\varphi}$-filtrations} \label{subsec:J2-lifts}
We develop a general method for constructing $J_2$-holomorphic lifts from certain filtrations, which will culminate in Proposition \ref{pr:J2-from-splits}.
We first describe the filtrations we need; 
again, $M$ will denote an arbitrary (connected) Riemann surface.

\begin{definition}\label{def:Az-filt}
Let $\varphi:M \to \U{n}$ be a smooth map.
 Let $(Z_i)$ be a finite sequence of subbundles of $\CC^n$ which forms a \emph{filtration} of $\CC^n$:
\be{Az-filt}
\CC^n =Z_0  \supset  Z_1  \supset  \cdots  \supset  Z_t  \supset  Z_{t+1} = \ul{0}\,. 
\ee
We call $(Z_i)$ an \emph{$A_z^{\varphi}$-filtration (of length $t$)} if,
for each $i = 0,1,\ldots, t$,
\begin{enumerate}
 \item[(i)] $Z_i$ is a holomorphic subbundle, i.e.,
		$\Gamma(Z_i)$ is closed under $D_{\zbar}^{\varphi}$\,;
 \item[(ii)] $A_z^{\varphi}$ maps $Z_i$ into the smaller subbundle $Z_{i+1}$\,. 
\end{enumerate}
\end{definition}

Let $\varphi:M \to \U n$ be a smooth map. Say that $\varphi$ is \emph{nilconformal}
if $A_z^{\varphi}$ is nilpotent, i.e., $(A_z^{\varphi})^r=0$ for some non-negative
integer $r$.  Then  \emph{$A_z^{\varphi}$-filtrations exist if and only if\/ $\varphi$
is nilconformal}.
Note that $\varphi$ is nilconformal if and only if $g\varphi$ is for any $g \in \U n$.  Burstall \cite{burstall-grass} calls a smooth map $\varphi:M \to G_*(\cn^n)$ from a surface to a Grassmannian \emph{nilconformal} if
$(A_z^{\varphi})^r|_{\varphi}=0$ for some $r$; since this implies that $(A_z^{\varphi})^{r+1}|_{\varphi^{\perp}}=0$, nilconformality of $\varphi$ in Burstall's sense is equivalent to both nilconformality of $\varphi^{\perp}$ in his sense and nilconformality of $\varphi$ in our sense.

\emph{Any nilconformal map is weakly conformal}; indeed, by nilpotency of $A_z^{\varphi}$
we have $\Tr (A_z^{\varphi})^2 = 0$, which is easily seen to be the condition for weak conformality (cf.\ \cite{burstall-grass}).
Also, \emph{any harmonic map of finite uniton number is nilconformal}, see Example \ref{ex:can-filt} below.

For convenience, if $E$ and $F$ are subspaces of $\cn^N$, or subbundles of $\CC^N$ \ 
$(N \in \nn)$,  and $F \subset E$,
we write $E \ominus F$ to mean $F^{\perp} \cap E$.

 Given a filtration $(Z_i)$ of $\CC^n$ of length $t$, we define its \emph{legs} $\psi_i$ by 
\be{alt-Z}
\psi_i = Z_i \ominus Z_{i+1}\,, \quad \text{equivalently}, \quad
Z_i = \sum_{j \geq i} \psi_j    \quad (i = 0,1, \ldots, t+1) \,. 
\ee
Then the $(t+1)$-tuple $\psi = (\psi_0,\psi_1,\ldots, \psi_t)$ is an orthogonal decomposition of $\CC^n$ into subbundles.
If all subbundles are non-zero, $\psi$ defines a smooth map from $M$ to a flag manifold, i.e., $\psi$ is a moving flag as defined above; we continue to call it a moving flag even if some subbundles are zero.   We extend the map $\pi_e$ defined by \eqref{twistor-proj} to the space of such moving flags, so that we may continue to write $\pi_e \circ \psi = \sum_j \psi_{2j}$.

Now let $\varphi:M \to G_*(\cn^n)$ be a smooth map to a Grassmannian.
Say that a filtration \eqref{Az-filt} is \emph{alternating for $\varphi$} if
its legs $\psi_i = Z_i \ominus Z_{i+1}$ satisfy
\begin{equation}\label{alternate}
\psi_i \subset (-1)^i \varphi \quad 
\quad \text{for } i = 0,1,\ldots, t
\end{equation}
(where, as usual, $-\varphi$ means $\varphi^{\perp}$).
This is equivalent to $\varphi = \sum_j \psi_{2j}$, i.e., $\pi_e \circ \psi =  \varphi$.
We define \emph{alternating for $\varphi^{\perp}$} similarly.
By `alternating' we shall mean `alternating for $\varphi$ or $\varphi^{\perp}$'.

\begin{lemma} \label{le:J2-Az-equiv}
Let $\varphi:M \to G_*(\cn^n)$ be a smooth map. Then
equation \eqref{alt-Z} defines a one-to-one correspondence between
moving flags $\psi = (\psi_0,\psi_1,\ldots)$ with $\varphi = \pi_e \circ \psi$
(i.e. $\varphi = \sum_i \psi_{2i}$) and which satisfy the $J_2$-holomorphicity condition \eqref{J2-cond} 
and $A^{\varphi}_z$-filtrations $(Z_i)$ of\/ $\CC^n$ which are alternating for $\varphi$. 
\end{lemma}

\begin{proof} 
Set $U_i = \sum_{j \geq i} \psi_{2j} \subset \varphi$ and
$V_i = \sum_{j \geq i} \psi_{2j+1} \subset \varphi^{\perp}$,
so that $Z_{2i} = U_i \oplus V_i$ and $Z_{2i+1} = U_{i+1} \oplus V_i$. 
It is easy to see that \eqref{J2-cond} is equivalent to
\be{U-V-cond}
\left\{
\begin{array}{rl}
\text{(i)} & \text{$U_i$ and $V_i$ are holomorphic subbundles of $\varphi$
and $\varphi^{\perp}$, respectively, and}\\
\text{(ii)} & \text{$A'_{\varphi}(U_i) \subset V_i$ and
$A'_{\varphi^{\perp}}(V_i) \subset U_{i+1}$}\,;
\end{array}
\right.
\ee
here, condition (i) means that $\Gamma(U_i)$ is closed under $\nabla''_{\varphi}$ and $\Gamma(V_i)$ is closed under $\nabla''_{\varphi^{\perp}}$.
The result follows by noting that conditions \eqref{U-V-cond} (i) and (ii) are equivalent to conditions (i) and (ii) of Definition \ref{def:Az-filt}, respectively.
\end{proof}
 
Call a filtration \eqref{Az-filt} \emph{strict} if all the inclusions $Z_{i+1} \subset Z_i$ are strict; then we have the following result, illustrated by the diagrams \eqref{diag:alt}; the length of the filtration $(Z_i)$ is equal to $4$ in the left-hand diagram and $5$ in the right-hand one.
In the diagrams, the vertices in the left (resp.\ right) columns represent the subbundles
making up $\varphi$ (resp.\ $\varphi^{\perp}$).
As in \cite{burstall-wood},  the possible non-zero ($\pa'$-)second fundamental forms $A'_{\psi_i,\,\psi_j}$ are indicated by the arrows: more precisely, the absence of an arrow from $\psi_i$ to $\psi_j$ indicates that  $A'_{\psi_i,\,\psi_j} = 0$.
Note that the arrows indicating the only possible non-zero second fundamental forms $A'_{\psi_i,\,\psi_j}$ are (i) vertically upwards when $\psi_i$ and $\psi_j$ are both in $\varphi$, or both in $\varphi^{\perp}$;
(ii) downwards from left to right when $\psi_i \subset \varphi$, $\psi_j \subset \varphi^{\perp}$, in which case they are the
 components of $A'_{\varphi}$; (iii) downwards from right to left when $\psi_i \subset \varphi^{\perp}$, $\psi_j \subset \varphi$, in which case they are the components of $A'_{\varphi^{\perp}}$.
{}From Lemma \ref{le:J2-Az-equiv} we deduce the following.
  
\begin{proposition} \label{pr:J2-Az-equiv}
Let $\varphi:M \to G_*(\cn^n)$ be a smooth map. Then
formulae \eqref{alt-Z} define a one-to-one correspondence between
{\rm (i)} $J_2$-holomorphic lifts $(\psi_0,\psi_1,\ldots):M \to F$ of\/ $\varphi$ to a complex flag manifold (with all $\psi_i$ non-zero)
and {\rm (ii)} strict $A^{\varphi}_z$-filtrations $(Z_i)$ which are alternating for $\varphi$.

Further, $F = F_{d_0,\ldots, d_{t+1}}$ where $(Z_i)$ has length $t$, and
$d_i = \rank \psi_i$ \ $(i = 0,1,\ldots, t+1)$.
\qed \end{proposition}

\vspace{-4ex}

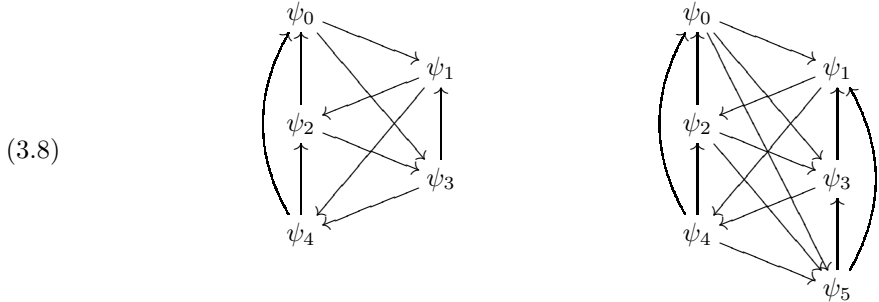
\begin{equation}
\begin{gathered}\label{diag:alt}
\xymatrixrowsep{0.5pc}\xymatrixcolsep{3pc}\xymatrix{ 
\psi_0 \ar[rd] \ar[rddd]  & \\
&  \psi_1 \ar[ld] \ar[lddd] \\
\psi_2 \ar[rd] \ar[uu] &\\
& \psi_3 \ar[ld] \ar[uu] \\
\psi_{4} \ar[uu] \ar@/^1.2pc/[uuuu]& }
\qquad \qquad \qquad\qquad
\xymatrixrowsep{0.5pc}\xymatrixcolsep{3pc}\xymatrix{ 
\psi_0 \ar[rd] \ar[rddd] \ar[rddddd]  & \\
&  \psi_1 \ar[ld] \ar[lddd] \\
\psi_2 \ar[rd] \ar[rddd] \ar[uu] &\\
& \psi_3 \ar[ld] \ar[uu] \\
\psi_{4} \ar[rd] \ar[uu] \ar@/^1.2pc/[uuuu]& \\
& \psi_5  \ar[uu] \ar@/_1.2pc/[uuuu] }
\end{gathered}
\end{equation}

\begin{remark} (i) The sets (i) and (ii) are empty unless $\varphi$ is harmonic and nilconformal.

(ii) Let $(Z_i)$ be a strict $A^{\varphi}_z$-sequence for a nilconformal harmonic map $M \to \U{n}$.  Then its legs $\psi_i = Z_i \ominus Z_{i+1}$ define a moving flag and $\wt\varphi = \sum_j \psi_{2j}:M \to G_*(\cn^n)$ defines a map from $M$ to a Grassmannian.
It would be interesting to study this map; however, in general, $(Z_i)$ is not an
$A^{\wt\varphi}_z$-sequence, so $(\psi_i)$ is not a $J_2$-holomorphic lift and $\wt\varphi$ is not harmonic.
\end{remark}

In practice, many $A^{\varphi}_z$-filtrations $(Z_i)$ are not alternating but satisfy the following weaker condition.   Let $\varphi:M \to G_*(\cn^n)$ be a smooth map.
Say that the subbundle $Z_i$ of $\CC^n$ \emph{splits (for $\varphi$)} if it is a direct sum:
\be{U-V}
Z_i = U_i \oplus V_i \quad \text{where}
\quad U_i \subset \varphi \quad \text{and} \quad
V_i \subset \varphi^{\perp}.
\ee
Say that \emph{the filtration $(Z_i)$ splits} if each subbundle $Z_i$ splits, 
equivalently, each $\psi_i = Z_i \ominus Z_{i+1}$ splits:
$\psi_i = \psi_i \cap {\varphi} \oplus \psi_i \cap \varphi^{\perp}$.
It is convenient to write
\be{Bi}
\wh\psi_{2i} = U_i \ominus U_{i+1} = \psi_i \cap \varphi \quad \text{ and } \quad
\wh\psi_{2i+1} = V_i \ominus V_{i+1} = \psi_i \cap \varphi^{\perp};
\ee
the resulting subbundles $\wh{\psi_i}$ and possible second fundamental forms
 $A'_{\wh\psi_i,\,\wh\psi_j}$ are shown for filtrations of lengths $1$ and $2$ in the diagrams \eqref{diag:splits}.  Note that some of
the $\wh{\psi_i}$ may be zero.
\begin{equation}
\begin{gathered}\label{diag:splits}
\xymatrixrowsep{0.6pc}\xymatrix{ \\
\wh\psi_0  \ar[rdd] &\wh\psi_1 \ar[ldd] \\ \\
\wh\psi_2   \ar[uu] & \wh\psi_3 \ar[uu]}
\qquad \qquad
\xymatrixrowsep{0.6pc}\xymatrixcolsep{3pc}\xymatrix{
& \wh\psi_0  \ar[rdddd] \ar[rdd] & \wh\psi_1 \ar[ldd] \ar[ldddd] \\ \\
& \wh\psi_2 \ar[uu] \ar[rdd] & \wh\psi_3 \ar[uu] \ar[ldd] \\ \\
& \wh\psi_4 \ar[uu] \ar@/^1.2pc/[uuuu] & \wh\psi_5 \ar[uu] \ar@/_1.2pc/[uuuu]}
\end{gathered}
\end{equation}

Note that, since $A_z^{\varphi}$ maps $Z_i$ into $Z_{i+1}$, it has no components between $\wh\psi_{2i}$ and $\wh\psi_{2i+1}$, accounting
for the lack of horizontal arrows in these diagrams. 

We now give four ways to convert an $A_z^{\varphi}$-filtration which splits into one which is alternating. 

\begin{lemma} \label{le:combine}
Let $\varphi:M \to G_*(\cn^n)$ be a smooth map.
Let $(Z_i)$ be an $A_z^{\varphi}$-filtration which splits; denote its length by $t$.
Set $\psi_i = Z_i \ominus Z_{i+1}$ and
define $U_i$, $V_i$ and $\wh\psi_i$ by \eqref{U-V} and \eqref{Bi}.

{\rm (i)} Set $\wt Z_{2j} = U_j \oplus V_j$ and
$\wt Z_{2j+1} = U_{j+1} \oplus V_j$. 
 Then $(\wt Z_i)$ is an alternating $A_z^{\varphi}$-filtration of length
$2t+1$ with legs
\be{split1}
(\wt\psi_0, \wt\psi_1, \ldots, \wt\psi_{2t}, \wt\psi_{2t+1}) = 
(\wh\psi_0, \wh\psi_1, \ldots, \wh\psi_{2t}, \wh\psi_{2t+1}).
\ee

{\rm (ii)}  Reversing the roles of\/ $\varphi$ and $\varphi^{\perp}$, set $\wt Z_{2j} = U_j \oplus V_j$ and
$\wt Z_{2j+1} = U_j \oplus V_{j+1}$. 
 Then $(\wt Z_i)$ is an alternating $A_z^{\varphi}$-filtration of length
$2t+1$ with legs
\be{split2}
(\wt\psi_0, \wt\psi_1, \ldots, \wt\psi_{2t}, \wt\psi_{2t+1})
= (\wh\psi_1, \wh\psi_0, \ldots, \wh\psi_{2t+1}, \wh\psi_{2t}).
\ee

{\rm (iii)} Set
$
\wt Z_{2j} = U_{2j-1} \oplus V_{2j}$  where  $U_{-1} = \varphi\,,$
 and  $\wt Z_{2j+1} = U_{2j+1} \oplus V_{2j}$\,.
Then $(\wt Z_i)$ is an alternating $A_z^{\varphi}$-filtration of length $t+1$.
Its legs are given by
\be{split3}
(\wt\psi_0, \wt\psi_1, \ldots, \wt\psi_{t+1}) = 
\begin{cases}
(\wh\psi_0, \wh\psi_1+\wh\psi_3, \wh\psi_2+\wh\psi_4, \wh\psi_5+\wh\psi_7, \ldots,
\wh\psi_{2t-2}+\wh\psi_{2t}, \wh\psi_{2t+1}) & \ ($t$ \ even), \\  
(\wh\psi_0, \wh\psi_1+\wh\psi_3, \wh\psi_2+\wh\psi_4, \wh\psi_5+\wh\psi_7, \ldots,
\wh\psi_{2t-1}+\wh\psi_{2t+1}, \wh\psi_{2t})  & \ ($t$ \ odd).
\end{cases}
\ee

{\rm (iv)} Reversing the roles of\/ $\varphi$ and $\varphi^{\perp}$, set
$$
\wt Z_{2j} = U_{2j} \oplus V_{2j-1}\,, \quad \text{and} \quad
\wt Z_{2j+1} = U_{2j} \oplus V_{2j+1} \text{ where } V_{-1} = \varphi^{\perp}.
$$
Then $(\wt Z_i)$ is an alternating $A_z^{\varphi}$-filtration of length $t+1$.
Its legs are given by
\be{split4}
(\wt\psi_0, \wt\psi_1, \ldots, \wt\psi_{t+1}) = 
\begin{cases} 
(\wh\psi_1, \wh\psi_0+ \wh\psi_2, \wh\psi_3+\wh\psi_5, \wh\psi_4+\wh\psi_6, \ldots
\wh\psi_{2t-1}+\wh\psi_{2t+1}, \wh\psi_{2t}) & \ ($t$ \ even), \\
(\wh\psi_1, \wh\psi_0+ \wh\psi_2, \wh\psi_3+\wh\psi_5, \wh\psi_4+\wh\psi_6, \ldots
\wh\psi_{2t-2}+\wh\psi_{2t}, \wh\psi_{2t+1}) & \ ($t$ \ odd).
\end{cases}
\ee

Further, each of the above four moving flags \eqref{split1}--\eqref{split4} satisfies the $J_2$-holomorphicity condition \eqref{J2-cond}.
\end{lemma}

\begin{proof}
This follows from Lemma \ref{le:J2-Az-equiv}: that the four moving flags satisfy the hypotheses of that lemma is easily checked.
\end{proof}

Parts (iii) and (iv) are illustrated for $t=4$ by the left- and right-hand diagrams of
\eqref{diag:split-alt}, respectively.  For clarity, the second fundamental forms $A'_{\wh\psi_i,\wh\psi_j}$
between subbundles $\wh\psi_i$, $\wh\psi_j$ which are both in $\varphi$ or both in $\varphi^{\perp}$ are not shown,
and only the arrows of least gradient between
subbundles $\wh\psi_i$, $\wh\psi_j$\,, with one in $\varphi$ and the other in $\varphi^{\perp}$, are shown.

\smallskip

\begin{equation}
\begin{gathered}\label{diag:split-alt}
\xymatrixrowsep{1.5pc}\xymatrixcolsep{3.2pc}\xymatrix{
\wt\psi_0=\wh\psi_0 \ar[rd]   & \wh\psi_1 \ar[ld] \\
\wh\psi_2 \ar[rd] & \wh\psi_3 \ar[ld] \ar@{-}[u]_{\wt\psi_1} \\
\wh\psi_4 \ar[rd] \ar@{-}[u]^{\wt\psi_2}  & \wh\psi_5 \ar[ld] \\
\wh\psi_6 \ar[rd]  & \wh\psi_7 \ar[ld] \ar@{-}[u]_{\wt\psi_3} \\
\wh\psi_8 \ar@{-}[u]^{\wt\psi_4} &  \wh\psi_9=\wt\psi_5  }
\qquad\qquad\qquad
\xymatrixrowsep{1.5pc}\xymatrixcolsep{3.2pc}\xymatrix{
\wh\psi_0 \ar[rd]   &  \wh\psi_1=\wt\psi_0 \ar[ld] \\
\wh\psi_2 \ar[rd] \ar@{-}[u]^{\wt\psi_1} & \wh\psi_3 \ar[ld]  \\
\wh\psi_4 \ar[rd]   & \wh\psi_5 \ar[ld] \ar@{-}[u]_{\wt\psi_2} \\
\wh\psi_6 \ar[rd] \ar@{-}[u]^{\wt\psi_3} & \wh\psi_7 \ar[ld]  \\
\wt\psi_5=\wh\psi_8  & \wh\psi_9 \ar@{-}[u]_{\wt\psi_4} }
\end{gathered}
\end{equation}

\medskip

By Proposition \ref{pr:J2-Az-equiv}, the four moving flags \eqref{split1}--\eqref{split4} define $J_2$-holomorphic lifts of $\varphi$ or $\varphi^{\perp}$ if all legs are non-zero. We now give a process for removing legs which are zero.  Recall that, for a moving flag
$\psi = (\psi_0,\psi_1, \ldots)$, its twistor projection is given by
$\pi_e\circ\psi = \sum_j \psi_{2j}$.   The following is easily checked.

\begin{lemma} \label{le:remove-zero}
Let $\psi = (\psi_0, \psi_1, \ldots, \psi_t)$ be a moving flag with some legs equal to zero.  Each of the following three operations
gives a moving flag
$\wt\psi = (\wt\psi_0, \wt\psi_1, \ldots, \wt\psi_s)$ with $s < t$, fewer zero legs,
and $\pi_e \circ \wt\psi = \pm\,\pi_e\circ\psi$;
  further, if\/ $\psi$ satisfies the $J_2$-holomorphicity condition \eqref{J2-cond}, then
so does $\wt\psi$.
  
\emph{Operation 1}.  If the first leg $\psi_0$ is zero, \emph{remove it and renumber}$:$
$\wt\psi_j = \psi_{j+1}$, thus reducing the number of legs by one.

\emph{Operation 2}. If the last leg $\psi_t$ is zero, \emph{remove it}, thus reducing the number of legs by one.

\emph{Operation 3}. If any other leg $\psi_i$ is zero, \emph{remove it and combine the legs
on each side}, giving a new lift with two fewer legs$:$
$$
\wt\psi_j = \psi_j \ (j < i-1)\,, \quad 
\wt\psi_j = \psi_{i-1} + \psi_{i+1} \ (j=i-1) \,, \quad
\wt\psi_j = \psi_{j+2}\ (j > i-1)\,.
$$
Note that, after Operation 1,  $\pi_e\circ \wt\psi = -\pi_e \circ \psi$;
after Operations 2 or 3, $\pi_e\circ \wt\psi = \pi_e \circ \psi$.
\qed \end{lemma}

By iterating these operations, we obtain the following result.

\begin{proposition}  \label{pr:remove-zero-legs}
Let $\psi = (\psi_0,\psi_1,\ldots, \psi_t)$ be a moving flag which satisfies \eqref{J2-cond}.
Set $\varphi = \pi_e\circ\psi$.
Then we can remove and combine legs by the operations in Lemma \ref{le:remove-zero} to obtain a moving flag 
$\wt\psi = (\wt\psi_0, \wt\psi_1,\ldots, \wt\psi_s)$
with $s \leq t$ and no $\wt\psi_i$ equal to zero, which satisfies \eqref{J2-cond}, and has
$\pi_e\circ \wt\psi = \pm\varphi;$ thus $\wt\psi: M \to F$ is 
a $J_2$-holomorphic lift of\/ $\pm\varphi$.
\qed \end{proposition}

Even when the legs are non-zero, we can sometimes obtain lifts with fewer legs by a fourth operation as follows; for this, recall the definition \eqref{2nd-f-f} of the
$\pa'$-second fundamental forms $A'_{\psi_i,\psi_j}$.     Again, the proof is by direct checking.
   
\begin{lemma} \label{le:reduce-legs}
Let $\psi = (\psi_0,\psi_1,\ldots, \psi_t)$ be a moving flag which satisfies  \eqref{J2-cond},
and set $\varphi =  \pi_e\circ\psi$.

\emph{Operation 4.}
If $A'_{\psi_i,\psi_{i+1}}$ is zero, \emph{replace}
$\ldots, \psi_{i-1}, \psi_i, \psi_{i+1}, \psi_{i+2}, \ldots$ by
$\ldots,\psi_{i-1}+\psi_{i+1}, \psi_i+\psi_{i+2}, \ldots$.

This gives a new moving flag $\wt\psi = (\wt\psi_0,\wt\psi_1,\ldots, \wt\psi_s)$
with $s \leq t$ and $\pi_e \circ \wt\psi = \varphi$ and
satisfying \eqref{J2-cond}. 
$($Here, we set $\psi_i$ equal to zero if\/ $i$ is outside the range $0 \leq i \leq t$.$)$ 
\qed \end{lemma}

Note that Operation 4
reduces the number of legs by two unless $i = 0$ or $i+1 = t$, in which case it reduces the number of legs by one.
By iterating this process, we can find a $J_2$-holomorphic map 
$(\psi_0,\psi_1,\ldots, \psi_t): M \to F$ satisfying
\be{norm-legs}
 A'_{\psi_i,\psi_{i+1}} \neq 0 \qquad (i = 0,1,\ldots, t-1). \hfill 
\ee

On putting the above results together, we obtain the main result of this section.

\begin{proposition} \label{pr:J2-from-splits}
Let $\varphi:M \to G_*(\cn^n)$ be a smooth map.
Let $(Z_i)$ be an $A^{\varphi}_z$-filtration which splits for $\varphi$. 
Set $\psi_i = Z_i \ominus Z_{i+1}$ and write
$\psi_i = \wh\psi_{2i} \oplus \wh\psi_{2i+1}$ where $\wh\psi_{2i} \subset \varphi$
and $\wh\psi_{2i+1} \subset \varphi^{\perp}$.
Then there is a $J_2$-holomorphic twistor lift
$\psi = (\psi_0,\psi_1,\ldots,\psi_t):M \to F$ of $\varphi$ or $\varphi^{\perp}$ to a flag manifold $F=F_{d_0,\ldots, d_t}$ satisfying \eqref{norm-legs} with every leg
$\psi_i$ the sum of some of the $\wh\psi_j$.
\end{proposition}

\begin{proof}
As in Lemma \ref{le:combine}(i) the moving flag $(\wh\psi_i)$ satisfies the
$J_2$-holomorphicity condition.  By carrying out Operations 1--4 as above, this can
be modified to give a $J_2$-holomorphic lift with the stated properties.
\end{proof}

Now, for any nilconformal harmonic map $\varphi:M \to G_*(\cn^n)$, we can find $A^{\varphi}_z$-filtrations which split
 for $\varphi$, see the next examples; thus we obtain the following result, which also
follows from the work of Burstall \cite[Section 3]{burstall-grass}.

\begin{corollary} \label{co:converse}
A smooth map $\varphi:M \to G_*(\cn^n)$ from a surface to a complex Grassmannian has a
$J_2$-holomorphic twistor lift\/ $\psi:M \to F$ to a flag manifold
if and only if it is harmonic and nilconformal.
\end{corollary}

Using our methods, we can give a more detailed result, see Theorem \ref{th:lift-with-uniton}.  

We now explain how Burstall's construction fits into our theory.
Recall that, when $\varphi:M \to \U n$ is a
 harmonic map, $A_z^{\varphi}$ is a holomorphic endomorphism of
 $(\CC^n, D^{\varphi}_{\zbar})$.

\begin{example} \label{ex:burstall-J2} 
Let $\varphi:M \to \U n$ be a nilconformal harmonic map so that
$(A^{\varphi}_z)^{t+1} = 0$ for some $t \in \nn$. Set $Z_0 = \CC^n$ and $Z_i = \Ima(A_z^{\varphi}|_{Z_{i-1}})$.
 Then we obtain an $A_z^{\varphi}$-filtration:
$
Z_i = \Ima(A_z^{\varphi})^i,
$
which we call the \emph{filtration by $A^{\varphi}_z$-images};
note that $Z_{t+1} = 0$ which implies that all inclusions $Z_{i+1} \subset Z_i$ are strict.
If\/ $\varphi$ is Grassmannian, this clearly splits, so
we may apply part (iii) of Lemma \ref{le:combine} to obtain an alternating
$A_z^{\varphi}$-filtration $(\wt Z_i)$   with legs
\eqref{split3}. Alternatively, we may apply part (iv) to obtain an alternating
$A_z^{\varphi}$-filtration $(\wt Z_i)$ with legs \eqref{split4}.

By strictness of the filtration, all the legs in \eqref{split3} and \eqref{split4} are non-zero, with the possible exception of the first ones.
However, since $Z_1 \neq Z_0$, either (a) $A'_{\varphi^{\perp}}$ is not surjective and $\wh\psi_0$ is non-zero, in which case
\eqref{split3} gives a $J_2$-holomorphic lift of $\varphi$, or
(b) $A'_{\varphi}$ is not surjective and $\wh\psi_1$ is non-zero, in which case \eqref{split4} gives a $J_2$-holomorphic lift of
$\varphi^{\perp}$; for some $\varphi$, both (a) and (b) hold and we get both lifts. 
In case (a) we get  $\wt Z_i = \wt U_i + \wt V_i$ with
$\wt U_i = \Ima(A_z^{\varphi})^{2i-1}|_{\varphi^{\perp}}
= \Ima \bigl((A'_{\varphi^{\perp}} \circ A'_{\varphi})^{i-1} \circ A'_{\varphi^{\perp}} \bigr)$ and
$\wt V_i = \Ima(A_z^{\varphi})^{2i}|_{\varphi^{\perp}}
= \Ima (A'_{\varphi} \circ A'_{\varphi^{\perp}})^i$ for $i = 1,2,\ldots$.
The formulae for case (b) are obtained by interchanging
$\varphi$ and $\varphi^{\perp}$.  This interprets a construction of Burstall \cite[Section 3]{burstall-grass};
see also Example \ref{ex:nilconf-Y}.
\end{example}
  
\begin{example} \label{ex:burstall-dual}
Dually (cf.\ Example \ref{ex:dual}), set
$\wh Z_i = \ker (A^{\varphi}_z)^{t+1-i}$ so that $\wh Z_{t+1}=0$; we call this the
\emph{filtration by $A^{\varphi}_z$-kernels}.
{}From $(A^{\varphi}_z)^{t+1} = 0$, we see that the filtration $(Z_i)$ by $A^{\varphi}_z$-images in the last example
and the filtration $(\wh Z_i)$ by $A^{\varphi}_z$-kernels satisfy
$Z_i \subset \wh Z_i$; however, these two filtrations are different, in general.
\end{example}
 
Note that $A^{\varphi}_z|_{\varphi} \equiv A'_{\varphi}$ is zero if and only if
$\varphi$ is antiholomorphic, equivalently, $\varphi^{\perp}$ is holomorphic.
The first non-trivial case of a nilconformal map $\varphi$ is when
 $(A^{\varphi}_z)^2|_{\varphi} \equiv A'_{\varphi^{\perp}} \circ A'_{\varphi}$
is zero; we consider such maps now.

\begin{example} \label{ex:str-conf}
A harmonic map $\varphi:M \to G_*(\cn^n)$ is called \emph{strongly conformal} if\/ $A'_{\varphi^{\perp}} \circ A'_{\varphi} = 0$, equivalently,
$G'(\varphi)$ and $G''(\varphi)$ are orthogonal; such maps are clearly nilconformal.
If\/ $\varphi$ is a strongly conformal harmonic map which is neither holomorphic nor antiholomorphic, then $G'(\varphi)$ and $G''(\varphi)$ are non-zero orthogonal subbundles of $\varphi^{\perp}$, and we have twistor lifts as follows:

(i) In Example \ref{ex:burstall-J2}, case (b) gives the $J_2$-holomorphic lift
$\psi = \bigl(G'(\varphi)^{\perp} \cap \varphi^{\perp}, \varphi, G'(\varphi)\bigr)$ of $\varphi^{\perp}$.

(ii) Similarly, we have a $J_2$-holomorphic lift of
$\varphi^{\perp}$ given by 
$\psi = \bigl(G''(\varphi), \varphi, G''(\varphi)^{\perp} \cap \varphi^{\perp})$.
This lift is dual to (i) in the sense that it is obtained from the lift in (i) by replacing the complex structure
on the domain by its conjugate (cf.\ Example \ref{ex:dual}).

(iii) Examples (i) and (ii) are the extreme cases of the following. 
\emph{Let $\varphi:M \to G_*(\cn^n)$ be a strongly conformal harmonic map which is neither holomorphic nor antiholomorphic.  Suppose that
\be{W-cond}
W \text{ is a holomorphic subbundle of\/ }
\varphi^{\perp} \text{ satisfying\/ } \Ima A'_{\varphi} \subset W \subset \ker A'_{\varphi^{\perp}}.
\ee
Set $V = W^{\perp} \cap \varphi^{\perp}$.  
Then $(V,\varphi,W)$ is a $J_2$-holomorphic lift of\/ $\varphi^{\perp}$, and every
$J_2$-holomorphic lift of\/ $\varphi^{\perp}$ with three legs is given this way}.
To see this, first note that the conditions \eqref{W-cond} are equivalent to
\emph{\be{V-cond}
V \text{ is an antiholomorphic subbundle of\/ } \varphi^{\perp} \text{ satisfying\/ }
\Ima A''_{\varphi} \subset V \subset \ker A''_{\varphi^{\perp}}.
\ee}
Then note that these conditions are equivalent to saying that we have a diagram \eqref{diag:str-conf}, where as usual, the arrows indicate the possible non-zero ($\pa'$-)second fundamental forms.
\begin{equation}
\begin{gathered}\label{diag:str-conf}
\xymatrixrowsep{0.3pc}\xymatrixcolsep{3pc}\xymatrix{
 & V \ar[dl] \\  \varphi \ar[dr] & \\
  & W \ar[uu]
 }
\end{gathered}
\end{equation}
Finally, from Proposition \ref{pr:J1-J2}, we see that this diagram says precisely that
$\psi =(V,\varphi,W)$ is $J_2$-holomorphic.

(iv) On replacing $\varphi$ by $\varphi^{\perp}$, we obtain the following from part (iii).
\emph{Let $\varphi:M \to G_2(\cn^n)$ be a harmonic map which is neither holomorphic nor antiholomorphic.  Suppose that $\varphi^{\perp}$ is strongly conformal.  Then $\varphi$ has a
\emph{unique} $J_2$-holomorphic lift
$\psi = \bigl(G''(\varphi^{\perp}), \varphi^{\perp}, G'(\varphi^{\perp})\bigr)$}.  See
Corollary \ref{co:G2Cn}
for more information, and Example \ref{ex:r2} for the superhorizontal case which links to \cite{eells-wood} and \cite{erdem-wood}.
\end{example}

\begin{example} \label{ex:superconf} 
Consider the isometric minimal immersion of the torus 
$\cn/\langle 2\pi/\sqrt{3}, \ 2\pi\ii \rangle$ into $\CP{2}$ given by the harmonic map
\begin{equation}
\varphi(z) = \bigl[\eu^{z-\ov{z}},\, \eu^{\zeta z - \ov{\zeta z}},\,
    \eu^{\zeta^2 z - \ov{\zeta^2 z}}\,\bigr] \,,
\ee
where $\zeta = \eu^{2\pi\ii/3}$.
 For $i = 0,1,2,\ldots$, set $\varphi_i$ equal to the $i$th $\pa'$-Gauss transform
$G^{(i)}(\varphi)$ (see \S \ref{subsec:Grass}).  Then
$$
\varphi_i(z) = \bigl[\eu^{z-\ov{z}},\, \zeta^i \eu^{\zeta z - \ov{\zeta z}},\,
    \zeta^{2i} \eu^{\zeta^2 z - \ov{\zeta^2 z}}\,\bigr]\,;
$$
in particular, $G^{(3)}(\varphi) = \varphi$ showing that $\varphi$ is \emph{superconformal}.
 \cite{bolton-woodward,bolton-pedit-woodward}.  

 It follows that $\varphi$ is \emph{of finite type}, see \cite[Corollary 2.7]{bolton-woodward}; such a map cannot be of finite uniton number by a result of R.\ Pacheco \cite{pacheco-tori}.
For this $\varphi$, the $A_z^{\varphi}$-filtrations $(z_i)$ of Examples \ref{ex:burstall-J2},
\ref{ex:burstall-dual} and \ref{ex:str-conf} all coincide and are given by $Z_0 = \CC^n$,
$Z_1 = \varphi \oplus \varphi_1$,
$Z_2 = \varphi_1$, $Z_3 = \ul{0}$; the legs of this filtration define
the $J_2$-holomorphic lift
$\psi = (\varphi_2, \varphi, \varphi_1) = \bigl(G''(\varphi), \varphi, G'(\varphi)\bigr)$
of $\varphi^{\perp}$.  
\end{example}

\section{Twistor lifts from extended solutions}  \label{sec:twistor-extended}
\subsection{Extended solutions} \label{subsec:ext-sol}

Let $G$ be a Lie group with identity element $e$ and Lie algebra $\g \cong T_eG$, and let $\g^{\cn}$ denote the complexified Lie algebra $\g \otimes \cn$.  Let $\Omega G$ be the \emph{(based) loop group}
$\Omega G =\{\gamma:S^1 \to G \textrm{ smooth}: \gamma(1)=e\}$.
Recall \cite{uhlenbeck,burstall-guest} that a smooth map $\Phi:M\to\Omega G$ is said
to be an \emph{extended solution} if, on any coordinate domain $(U,z)$,
we have $\Phi^{-1}\Phi_z=(1-\lambda^{-1})A$, for some map $A:U\to\g^{\cn}$. 

For any map $\Phi:M \to \Omega G$ and $\lambda \in S^1$, we define
$\Phi_{\lambda}:M \to G$  by $\Phi_{\lambda}(p) = \Phi(p)(\lambda)$ \ $(p \in M)$.
If $\Phi:M\to\Omega G$ is an extended solution,  the map
$\varphi=\Phi_{-1}:M\to G$ is harmonic
and
$\varphi^{-1}\varphi_z=2A$; on comparing with \eqref{type-decomp} we see that $A = A^{\varphi}_z$.

Conversely, given a harmonic map $\varphi:M\to G$, an extended solution $\Phi:M\to\Omega G$
is said to be \emph{associated} to $\varphi$ if
$
\Phi^{-1}\Phi_z=(1-\lambda^{-1})A^{\varphi}_z\,,
$
equivalently, $\varphi= g \/\Phi_{-1}$
for some $g\in G$.  Extended solutions associated to any given harmonic map always exist locally;
they exist globally if the domain $M$ is simply connected, for example, if $M = S^2$. Further, any two extended solutions  $\Phi$, $\wt{\Phi}$ associated to the same harmonic map differ by a loop on their common domain $D \subset M$, i.e., $\wt\Phi = \gamma \Phi$ for some
$\gamma \in \Omega G$; in particular, they are equal if they agree at some point of $D$.

Let $\H = \H^{(n)}$ denote the Hilbert space $L^2(S^1,\cn^n)$. By expanding into Fourier series, we have
$$\H = \text{ closure of } \
	\spa\{\lambda^i e_j\ :\ i\in\zn,\ j=1,\dots,n\},
$$
where $\{e_1,\dots,e_n\}$ is the standard basis for $\cn^n$.
Thus, elements of $\H$ are of the form $v = \sum_i \lambda^i v_i$ where each
$v_i \in \cn^n$. We define projections $P_i:\H \to \cn^n$ which give the Fourier coefficients:
\begin{equation} \label{fourier-coeffs}
P_i \Bigl(\sum_i \lambda^i v_i \Bigr) = v_i \quad (i \in \zn).
\end{equation}
If $w = \sum_i \lambda^i w_i$ is another element of $\H$, its
$L^2$ inner product with $v$ is given by $\ip{v}{w} = \sum_i v_i \ov w_i$.
The natural action of $\U n$ on $\cn^n$ induces an action of $\Omega\U n$ on $\H$ which is isometric with respect to this $L^2$ inner product.  We consider the closed subspace 
$$
\H_+= \text{ closure of } \, \spa\{\lambda^i e_j\ :\ i \in \nn,\ j=1,\dots,n\},
$$
with orthogonal complement in $\H$ given by
$$
\H_+^\perp = \text{ closure of } \, \spa\{\lambda^{-i} e_j\ :\ i=1,2,\dots, \ j=1,\dots,n\}.
$$
The action of $\Omega\U n$ on $\H$ induces an action of $\Omega\U n$ on subspaces of $\H$; denote by $\Gr = \Gr^{(n)}$ the orbit of $\H_+$ under that action. It is known from \cite{pressley-segal} that $\Gr$ consists of all the closed subspaces $W\subset\H$ which enjoy the following properties:
\begin{itemize}
\item[(i)] $W$ is closed under multiplication by $\lambda$, i.e., $\lambda W\subset W$;
\item[(ii)] the orthogonal projection $W\to\H_+$ is Fredholm;
\item[(iii)] the orthogonal projection $W\to\H_+^{\perp}$ is Hilbert--Schmidt;
\item[(iv)] the images of the orthogonal projections $W\to\H_+^{\perp}$ and $W^\perp\to\H_+$ consist of smooth functions.
\end{itemize}
Furthermore, we have a bijective map  
\begin{equation} \label{Grass-model}
\Omega\U n\ni\Phi \mapsto W = \Phi\H_+\in\Gr\,;
\end{equation}
we call $W$ the \emph{Grassmannian model of\/
$\Phi$}.  The map \eqref{Grass-model} restricts to a bijection from the
\emph{algebraic loop group} $\Omega_{\alg}\U n$ consisting of
those $\gamma \in \Omega\U n$ given by finite Laurent series:
$\gamma = \sum_{i=s}^r \lambda^k T_k$\,, where $r \geq s$ are integers and the $T_k$ are $n \times n$ complex matrices, to the
set of $\lambda$-closed subspaces $W$ of $\H$
satisfying $\lambda^r\H_+ \subset W\subset\lambda^{s}\H_+\,,$
for some integers $r \geq s$.

For $r \in \nn$, let $\Omega_r\U n$ denote the set  of polynomials $\Phi = \sum_{k=0}^r \lambda^k T_k$ in $\Omega_{\alg}\U n$ of degree at most $r$.  Then \eqref{Grass-model} further restricts to a bijection
from $\Omega_r\U n$ to the subset $\Gr_r \subset \Gr$ of those $\lambda$-closed subspaces $W$ of $\H$ satisfying
\be{W-finite}
\lambda^r\H_+ \subset W \subset \H_+\,.
\ee

Now let $\Phi:M \to \Omega\U n$ be a smooth map and set $W =\Phi\H_+:M \to \Gr$.  We call $W$ an \emph{extended solution}
if {\rm (i)} $W$ is holomorphic, i.e. $\pa_{\zbar} (\Gamma(W)) \subset \Gamma(W)$, and
{\rm (ii)} $\Gamma(W)$ is closed under the operator $F = \lambda \pa_z$,
i.e., $F \bigl(\Gamma(W)\bigr) \subset \Gamma(W)$.  (Here, as usual, $\Gamma(\cdot)$ denotes the space of (smooth) sections of a vector bundle.) 
Then \cite{segal},
\emph{$\Phi$ is an extended solution if and only if\/  $W$ is an extended solution}.

A harmonic map $\varphi:M \to \U n$ is said to be of \emph{finite uniton number} if it has a polynomial associated extended solution $\Phi$ (so that $\varphi = g\Phi_{-1}$ for some $g \in \U{n}$) then the minimum degree of such a polynomial $\Phi$ is called the
\emph{(minimal) uniton number of\/ $\varphi$}\,.  All harmonic maps from $S^2$ to $\U n$ are of finite uniton number \cite{uhlenbeck}.
Uhlenbeck further showed  that, if $\varphi:M \to \U n$ has finite uniton number $r$, then
$r \leq n-1$ and $\varphi$ has a unique polynomial associated extended solution $\Phi = \sum_{k=0}^r \lambda^k T_k $ of degree $r$ with
$\Ima T_0$ \emph{full}, i.e., not lying in any proper subspace.  As in \cite{he-shen1,ferreira-simoes-wood,unitons}, we call this the \emph{type one extended solution}; however, this concept does not seem to be useful for the real cases in \S\ref{sec:real}\textit{ff}.

\subsection{Finding $J_2$-holomorphic lifts from extended solutions} \label{subsec:lifts-extended}

When we have an extended solution $\Phi$ for our harmonic map $\varphi:M \to \U{n}$, we can get
$A^{\varphi}_z$-filtrations, and thus twistor lifts, from suitable filtrations of the Grassmannian model of $\Phi$, as we now explain.

\begin{definition} \label{def:F-filtr} Let $W = \Phi\H_+$ be an extended solution.  Let $(Y_i)$ be a sequence of $\lambda$-closed
subbundles of $\CC^n$ with
\be{Y-filt}
W = Y_0 \supset Y_1 \supset \cdots \supset Y_t \supset Y_{t+1} =\lambda W.
\ee
Call $(Y_i)$  an \emph{$F$-filtration (of\/ $W$ of length\/ $t$)} if, for each $i$,
\begin{enumerate}
\item[(i)] $Y_i$ is holomorphic, i.e., $\Gamma(Y_i)$ is closed under $\partial_{\zbar}$, and
\item[(ii)] $F = \lambda \pa_z$ maps sections of $Y_i$ into sections of the smaller subbundle $Y_{i+1}$\,.
\end{enumerate}
These conditions imply that each $Y_i$ is an extended solution.
\end{definition}

Now let $W = \Phi\H_+$ be an extended solution and $\varphi = \Phi_{-1}:M \to \U n$ the corresponding harmonic map.
Consider the bundle morphism $P_0 \circ \Phi^{-1}:W \to \CC^n$ where $\HH$ denotes the trivial bundle $M \times \H$
and $P_0:\HH\to\CC^n$ denotes projection onto the zeroth Fourier coefficient,
as in \eqref{fourier-coeffs}.
If follows from the extended solution equations (see \cite[Proposition 3.9]{unitons}) that
the mapping $P_0 \circ \Phi^{-1}$ intertwines the operators
(i) $\pa_{\zbar}$ with $D_{\zbar}^{\varphi}$ and (ii) $F$ with $-A_z^{\varphi}$,
i.e., we have the commutative diagrams \eqref{diag:F-corr-Az}, where the vertical arrows are surjective maps with kernel $\Gamma(\lambda W)$.
\begin{equation}
\begin{gathered}\label{diag:F-corr-Az}
\xymatrixrowsep{1.5pc}\xymatrixcolsep{3pc}\xymatrix{
\Gamma(W) \ar[r]^{\pa_{\zbar}} \ar[d]_{P_0 \circ \Phi^{-1}} & \Gamma(W) \ar[d]^{P_0\circ\Phi^{-1}}
	\\
\Gamma(\CC^n) \ar[r]_{D_{\zbar}^{\varphi}} & \Gamma(\CC^n)}
\qquad
\xymatrixrowsep{1.5pc}\xymatrixcolsep{3pc}\xymatrix{
\Gamma(W) \ar[r]^F \ar[d]_{P_0 \circ \Phi^{-1}} & \Gamma(W) \ar[d]^{P_0 \circ \Phi^{-1}}
	\\
\Gamma(\CC^n) \ar[r]_{-A_z^{\varphi}} & \Gamma(\CC^n)
}
\end{gathered}
\end{equation}

Now, given a filtration \eqref{Y-filt}, we associate to it a filtration $(Z_i)$ of $\CC^n$ by setting
\be{Z-from-Y}
Z_i = P_0 \circ \Phi^{-1}Y_i \qquad (i=0,1,\ldots,t+1).
\ee

Property (i) above says that \emph{the bundle morphism $P_0 \circ \Phi^{-1}:(W,\pa_{\zbar}) \to (\CC^n, D_{\zbar}^{\varphi})$
is holomorphic}; this enables us to fill out zeros to make each $Z_i$  a subbundle.
Thus we have the commutative diagram \eqref{diag:F-Az-comm}.
\begin{equation}
\begin{gathered}\label{diag:F-Az-comm}
\xymatrixrowsep{1.5pc}\xymatrixcolsep{0pc}\xymatrix{
W = Y_0 \ar[d]_{P_0 \circ \Phi^{-1}} & \supset &  Y_1  \ar[d]^{P_0 \circ \Phi^{-1}}
	& \supset  & \dots & \supset & Y_t  \ar[d]^{P_0 \circ \Phi^{-1}} & \supset & Y_{t+1}
	= \lambda W  \ar[d]^{P_0 \circ \Phi^{-1}} \\
\CC^n = Z_0 & \supset & Z_1 & \supset  & \dots & \supset & Z_t & \supset & Z_{t+1} = \ul{0}
}
\end{gathered}
\end{equation}
Here, each vertical map is a restriction of $P_0 \circ \Phi^{-1}:W \to \CC^n$ and is a
surjective bundle morphism with kernel $\lambda W$;
further, each $Y_i$ is the inverse image of $Z_i$ under $P_0 \circ \Phi^{-1}$.

{}From Property (ii) above, we see that \emph{$(Y_i)$
is an $F$-filtration if and only if\/ $(Z_i)$ is an
$A_z^{\varphi}$-filtration}.
Thus, for an extended solution $W = \Phi\H_+$ and corresponding harmonic map
 $\varphi= \Phi_{-1}$,
\emph{there is a one-to-one correspondence between
$F$-filtrations of\/ $W$ and $A^{\varphi}_z$-filtrations
of\/ $\CC^n$}; in particular, \emph{$F$-filtrations of\/ $W$  exist if and only if $\varphi$
is nilconformal}.

We now give an important example of a $F$-filtration which will give us a canonical twistor lift for a harmonic map of finite uniton number.  In the sequel, let $r \in \nn$.

\begin{example} \label{ex:can-filt}
Let $\Phi$ be an extended solution.  Suppose that this is polynomial of degree $r$
so that $W = \Phi\H_+$ satisfies \eqref{W-finite}. Set
\be{can-filt}
Y_i =  W \cap \lambda^i \HH_+ + \lambda W \quad (i=0,1,\ldots,r+1); 
\ee
where where $\HH_+$ denotes the trivial bundle $M \times \H_+$.
Since $\lambda^{r+1}\HH_+ \subset \lambda W$, we see that $(Y_i)$
is an $F$-filtration of length $r$; we shall call it the
\emph{canonical $F$-filtration for $\Phi$}.  Setting $\varphi = \Phi_{-1}$\,,
we call the associated $A_z^{\varphi}$-filtration $(Z_i)$
obtained from $(Y_i)$ by \eqref{Z-from-Y}
 the \emph{canonical $A_z^{\varphi}$-filtration for $\Phi$}.
See Theorem \ref{th:can-lift} for the resulting twistor lift; we calculate this in terms of unitons in Example \ref{ex:can-filt2}.  
\end{example}

This example shows that \emph{if $\varphi:M \to \U n$ is a harmonic map of finite
uniton number $r$, then it is nilconformal with $(A_z^{\varphi})^{r+1} = 0$}.

To apply the above work to maps to a Grassmannian we now identify the appropriate class of extended solutions.
Let $\nu:\HH \to \HH$ be the involution $\lambda\mapsto-\lambda$. Then,
as in \cite[\S 8]{uhlenbeck} and \cite[\S 3]{segal},
\emph{$W = \Phi\H_+$ is closed under $\nu$ if and only if\/ $\Phi$ is $\nu$-invariant in the sense that}
\be{Phi-sym}
\Phi_{\lambda}\Phi_{-1}=\Phi_{-\lambda} \qquad (\lambda \in S^1)\,;
\ee
this condition implies that the map $\varphi=\Phi_{-1}$ satisfies $\varphi^2 = I$, which means that it has image in a complex Grassmannian $G_*(\cn^n)$.
Conversely, we have the following result.

\begin{lemma} \label{le:Q}
{\rm (i)} Let $\varphi:M\to G_*(\cn^n)$ be a harmonic map from a Riemann surface which has an
associated extended solution.  Then it has a $\nu$-invariant extended solution $\Psi$ with $\Psi_{-1} = \varphi$.

{\rm (ii)} Suppose that $\varphi:M \to G_*(\cn^n)$ is a harmonic map of uniton number $r$.  Then it has a
$\nu$-invariant polynomial extended solution $\Psi$ of degree $r$ or $r+1$ with $\Psi_{-1} = \varphi$.
\end{lemma}

\begin{proof}
(i) Let $\Phi$ be an associated extended solution of $\varphi$.
Fix a base point $z_0\in M$. By replacing $\Phi$ by $\Phi(z_0)^{-1}\Phi$,
we may assume that $\Phi_\lambda(z_0)=I$ for all $\lambda$. Pick a homomorphism $\gamma:S^1\to\U n$ with $\gamma(-1)=\varphi(z_0)$, for example
$\gamma(\lambda) = \pi_{\varphi(z_0)} + \lambda\pi_{\varphi(z_0)}^{\perp}$\,, 
 and define $\Psi=\gamma\Phi$. Then $\Psi$ is an extended solution associated to $\varphi$, and since $\Psi_{-1}(z_0)=\varphi(z_0)$, we have $\Psi_{-1}=\varphi$ everywhere. Now $\Psi_{-\lambda}\Psi^{\:\:-1}_{-1}$ is also an extended solution associated to $\varphi$ and $\Psi_{-\lambda}(z_0)\Psi_{-1}^{\:\:-1}(z_0) = \gamma(-\lambda)\gamma(-1) = \gamma(\lambda)=\Psi_\lambda(z_0)$ for all $\lambda$, which implies that $\Psi$ is $\nu$-invariant. 

(ii) Let $\Phi$ be the type one associated extended solution of $\varphi$ (see \S \ref{subsec:ext-sol}), so that
$\varphi = Q \Phi_{-1}$ for some $Q \in \U n$.
 Uhlenbeck shows \cite[\S 15]{uhlenbeck} (see also \cite[Lemma 4.6]{ferreira-simoes-wood})
that $\varphi$ maps into a Grassmannian if and only if $Q \in G_*(\cn^n)$ and $\Phi_{\lambda} = Q\Phi_{-\lambda}\Phi_{-1}^{\;-1}Q$.
Write $Q = \pi_A - \pi_A^{\perp}$ where
$A$ is a subspace of $\cn^n$; note that if $A = \cn^n$ (resp.\ $A = 0$) then $Q = I$
(resp.\ $Q=-I$). We see that 
$\Psi = (\pi_A + \lambda \pi_A^{\perp})\Phi$ is a $\nu$-invariant polynomial extended solution, of degree $r$ or $r+1$,
with $\Psi_{-1} = \varphi$, as required.
\end{proof}

We remark that the uniton number of a harmonic map $\varphi:M \to G_k(\cn^n)$ of finite uniton number is at most
$\min(2k, 2n-2k, n-1)$ \cite{dong-shen}; in fact \cite[Corollary 5.7]{unitons}, we can find a polynomial extended solution $\Phi$ of degree at most $\min(2k, 2n-2k, n-1)$ with
$\Phi_{-1} = \pm \varphi$. 
For some sharper estimates depending on the rank of $A$, see \cite{ferreira-simoes}.

Next, we consider the effect on the filtrations of $\nu$-invariance.  First note that,
if $W$ is closed under $\nu$, then $W = W_+ \oplus W_-$ where $W_{\pm}$ are the
$\pm 1$-eigenspaces of $\nu$.

\begin{lemma} \label{le:splits}
Let $\Phi$ be a $\nu$-invariant extended solution and set $W = \Phi\H_+$\,. Let $Y$ be a subbundle of\/ $W$
which contains $\lambda W$.  Then
$Y$ is closed under $\nu$ if and only if\/ $Z = P_0 \circ \Phi^{-1}Y$ splits, i.e., is the direct sum of subbundles $Z_+$ and $Z_-$ with $Z_{\pm} \in \pm\varphi$.
In that case, $Z_{\pm} = P_0 \circ \Phi^{-1}Y_{\pm}\,.$
\end{lemma}

\begin{proof}
The $\nu$-invariance condition \eqref{Phi-sym} implies that 
$\Phi_{-1}\circ P_0\circ\Phi^{-1} = P_0\circ\Phi^{-1} \circ\nu$, i.e.,
$P_0\circ\Phi^{-1}$ intertwines $\nu$ with the involution $\Phi_{-1}=\pi_{\varphi}-\pi_{\varphi^{\perp}}$ on $\CC^n$, which establishes the lemma. 
\end{proof}

Thus, in \eqref{diag:F-Az-comm}, each $Z_i$ splits if and only if each $Y_i$ is closed under $\nu$. 

\begin{remark}
For an alternative point of view, given a filtration \eqref{Y-filt}, set
$\wh Y_i = \pi(Y_i)$ \ $(i=0,1,\ldots, t+1)$, where
$\pr:W\to W/\lambda W$ denotes the natural projection. Since $Y_i$ contains $\lambda W$, we have $Y_i = \pi^{-1}(\wh Y_i)$.
The operator $F$ descends to $W/\lambda W$ and becomes tensorial; we call a filtration
$$
W/\lambda W = \wh Y_0 \supset \wh Y_1 \supset \cdots \supset \wt Y_t \supset \wh Y_{t+1} = \ul{0}
$$ an \emph{$F$-filtration (of\/ $W/\lambda W$)} if\/ $F$ maps $\wh Y_i$ into $\wh Y_{i+1}$\,.
Then \emph{$Y_i$ is an $F$-filtration if and only if\/ $\wh Y_i$ is}, and
\emph{the isomorphism $\Phi^{-1}: W/\lambda W \to \CC^n$ gives a one-to-one correspondence between $F$-filtrations of\/
 $W/\lambda W$ and $A_z^{\varphi}$-filtrations of\/ $\CC^n$}.
Since  $\Phi^{-1}\circ \pr = P_0\circ\Phi^{-1}$, we have
$Z_i = P_0\circ\Phi^{-1} Y_i = \Phi^{-1}\wh Y_i$\,. Lastly, $\nu$ descends to $W/\lambda W$, and 
invariance of $Y_i$ under $\nu$ is equivalent to invariance of $\wh Y_i$ under $\nu$.
Hence, it would be natural to work in $W/\lambda W$; however, for convenience, we continue to work in $W$.
\end{remark}

For a harmonic map to a Grassmannian, we can choose an extended solution $\Phi$ which is $\nu$-invariant; we now show that the
canonical filtration for such a $\Phi$ is alternating; see Example \ref{ex:can-filt2} for more information.

\begin{proposition} \label{pr:can-alternate}
Let $\Phi$ be a $\nu$-invariant polynomial extended solution of degree $r$.
Set $W = \Phi\H_+$ and $\varphi = \Phi_{-1}:M \to G_*(\cn^n)$.
Let $(Z_i)$ be the canonical $A_z^{\varphi}$-filtration of Example \ref{ex:can-filt}, and
set $\psi_i = Z_i \ominus Z_{i+1}$ \ $(i=0,1,\ldots, r)$.   Then
$\psi_i \subset (-1)^i \varphi$ so that $\varphi = \sum_j \psi_{2j}$\,.  
\end{proposition}

\begin{proof} Let $x \in \psi_i = Z_i \ominus Z_{i+1}$.  Then $x = P_0\circ\Phi^{-1}(y)$ for some $y \in W \cap \lambda^i \HH_+$\,. Write $y = y_+ + y_-$ where $y_{\pm} \in W_{\pm}$\,.  Then $x = x_+ + x_-$ where
$x_{\pm} = P_0\circ\Phi^{-1}(y_{\pm}) \in \pm\varphi$.

If $i$ is even, $y_- \in Y_{i+1}$; indeed, write $y_- = y_1 + y_2$ where $y_1 = P_i y_-$ (where $P_i$ is as in \eqref{fourier-coeffs}), then
applying $\nu$ gives $-y_- = y_1 + \nu y_2$, so that $y_1 = 0$.
Hence, $x_- \in Z_{i+1}$.  Now $x$ is orthogonal to $Z_{i+1}$\,; it follows that
$0 = \ip{x_+ + x_-}{x_-} = \ip{x_-}{x_-}$ so that $x_- = 0$ and $x = x_+ \in \varphi$.
Similarly if $i$ is odd, $x_+ = 0$ and $x = x_- \in -\varphi = \varphi^{\perp}$.
\end{proof}

\begin{remark} \label{re:Ai}
As in \cite[\S 3.4]{unitons}, write  $A_i = \wh Y_i \ominus \wh Y_{i+1}$; the isometry $\Phi:W/\lambda W \to \CC^n$ maps $A_i$ onto $\psi_i$.  Proposition \ref{pr:can-alternate} is equivalent
to saying that, for each $i$, $A_i$ is in the $(-1)^i$-eigenspace of $\nu:W/\lambda W \to W/\lambda W$. 
\end{remark}

By Lemma \ref{le:J2-Az-equiv}, the moving flag $\psi = (\psi_0,\psi_1, \ldots, \psi_r)$ defined in Proposition \ref{pr:can-alternate}
satisfies the $J_2$-holomorphicity condition \eqref{J2-cond}.  To make $\psi$ a twistor lift to a flag manifold, we need to ensure that each leg $\psi_i$ is non-zero.
As in \cite[\S 3.4]{unitons}, say that a polynomial extended solution $\Phi$ is
 \emph{normalized} if each $A_i$ is non-zero, equivalently, each $\psi_i$ is non-zero.  It is
 shown there that, if $\Phi$ is not normalized, there is a polynomial loop $\gamma$
such that $\wt{\Phi} = \gamma^{-1}\Phi$ is a normalized polynomial extended solution of lesser degree; further, if $\Phi$ is $\nu$-invariant,
then we can choose the loop  to be $\nu$-invariant, so that $\wt{\Phi}$ is $\nu$-invariant and $\wt{\Phi}_{-1} = \pm\Phi_{-1}$\,.   Recalling the definition of the flag manifold $F_{d_0,d_1,\ldots, d_r}$ from
\S \ref{subsec:twistor-Grass}, we have the following result.
 
 \begin{theorem} \label{th:can-lift}
 Let $\Phi$ be a normalized $\nu$-invariant polynomial extended solution; denote its degree by $r$ and set $\varphi = \Phi_{-1}$.
Let the $\psi_i$ be the legs of the canonical filtration as in Proposition \ref{pr:can-alternate} and set $d_i = \rank\psi_i$.   Then
$\psi = (\psi_0, \psi_1, \ldots, \psi_r):M \to F_{d_0,d_1,\ldots, d_r}$ is a $J_2$-holomorphic lift of\/ $\varphi$.
\end{theorem}

We call $\psi$ the \emph{canonical (twistor) lift of $\varphi$ defined by $\Phi$},
see Example \ref{ex:can-filt2} for a calculation of $\psi$ in terms of unitons.
Note that the canonical lift of $\varphi$ depends on the choice of extended solution $\Phi$, however we have the following consequence.

\begin{corollary}
Let $\varphi:M \to G_k(\cn^n)$ be a harmonic map of finite uniton number $r$.  Then there is a $J_2$-holomorphic twistor lift
$\psi:M \to F$ of\/ $\varphi$ or $\varphi^{\perp}$ from $M$ to some flag manifold $F=F_{d_0,d_1,\ldots, d_t}$ with $t \leq \min(r+1, 2k,2n-2k,n-1)$.
\end{corollary}

\begin{proof} By Lemma \ref{le:Q}, there is a $\nu$-invariant polynomial extended
solution $\Phi$ of degree $r$ or $r+1$ with $\Phi_{-1} = \varphi$. 
If $\Phi$ is not normalized, then by \cite[Corollary 5.7]{unitons}, we can
replace it by a normalized
$\nu$-invariant polynomial extended solution $\Psi$ with $\Psi_{-1} = \pm\varphi$ of
lesser degree, and that degree is at most $\min(2k,2n-2k,n-1)$. This gives a twistor lift
as in Theorem \ref{th:can-lift}.
\end{proof}

Lastly, we shall find the $F$-sequence which leads to Burstall's twistor lift --- note that this requires only that $\varphi$ be nilconformal and not necessarily of finite uniton number.

\begin{example} \label{ex:nilconf-Y}
Let $\varphi: M \to \U{n}$ be nilconformal so that $(A^{\varphi}_z)^{t+1} = 0$ for some $t \in \nn$.
Let $\Phi$ be an associated extended solution of $\varphi$ defined on an open subset
of $M$.  As shown in Lemma \ref{le:Q},
we can take this to be $\nu$-invariant with $\Phi_{-1} = \varphi$.
As usual, set $W = \Phi\H_+$\,.  Set $Y_0 = W$ and, for $i=1,2,\ldots$, set $Y_i = F(Y_{i-1}) + \lambda W$
(where we fill out zeros at points where the rank drops) so that
$Y_i = F^i(W) + \lambda W$; it follows that $Y_{t+1} = \lambda W$. 
The associated $A_z^{\varphi}$-filtration defined by \eqref{Z-from-Y}
is the filtration $Z_i = \Ima(A^{\varphi}_z)^i$
of Example \ref{ex:burstall-J2} which, for a harmonic map to a Grassmannian,
leads to Burstall's twistor lifts as explained in that example.
Now, any two associated extended solutions $\Phi$ and $\wt\Phi$ differ by a loop on their common domain: $\wt\Phi = \gamma\Phi$ for some
$\gamma \in \Omega\U{n}$.  Since multiplication by a loop commutes with $F$, these give the same filtration $(Z_i)$ on their common domain, so our construction is well defined on the whole of $M$.

When $\varphi$ has finite uniton number, the twistor lifts arising from this construction are, in general, not the same as the canonical lift discussed above.
See also Example \ref{ex:fact-images}. 
\end{example}

\section{Twistor lifts from unitons} \label{sec:twistor-unitons}
\subsection{Unitons} \label{subsec:unitons}
Let $\varphi:M \to \U n$ be a harmonic map.
Then a subbundle $\alpha$ of $\CC^n$ is said to be a \emph{uniton for $\varphi$} if it is
(i) holomorphic with respect to the Koszul--Malgrange holomorphic structure induced by $\varphi$, i.e., 
$D^{\varphi}_{\zbar}(\sigma)\in\Gamma(\alpha)$ for all $\sigma\in\Gamma(\alpha);$ and
(ii) closed under the endomorphism $A^{\varphi}_z$, i.e., 
$A^{\varphi}_z(\sigma)\in \alpha$ for all $\sigma\in \alpha$. Uhlenbeck showed \cite{uhlenbeck} that,
 if a subbundle $\alpha\subset\CC^n$ is a uniton for a harmonic map $\varphi$,
then (i) $\wt\varphi=\varphi(\pi_{\alpha}-\pi_{\alpha}^{\perp})$ is harmonic,
(ii) $\alpha^{\perp}$ is a uniton for $\wt\varphi$, and (iii) $\varphi= -\wt\varphi(\pi_{\alpha}^{\perp}-\pi_{\alpha})$.

\begin{example} Any holomorphic subbundle of\/ $(\CC^n,D^{\varphi}_{\zbar})$ contained in $\ker A^{\varphi}_z$ is a uniton for $\varphi$; we call such a uniton \emph{basic}. Any holomorphic subbundle of\/ $(\CC^n,D^{\varphi}_{\zbar})$ containing $\Ima A^{\varphi}_z$ is also a uniton, we call such a uniton \emph{antibasic}.  Note that, if\/ $\alpha$ is basic (resp.\ antibasic) for
$\varphi$, then $\alpha^{\perp}$ is antibasic (resp.\ basic) for $\wt\varphi=\varphi(\pi_{\alpha}-\pi_{\alpha}^{\perp})$.
\end{example} 

Now suppose that $\Phi$ is an extended solution associated to $\varphi$ and $\alpha$ is a subbundle of $\CC^n$, then Uhlenbeck showed that \emph{$\alpha$ is a uniton for $\varphi$ if and only if\/  $\wt\Phi=\Phi(\pi_{\alpha}+\lambda\pi_{\alpha}^{\perp})$ is an extended solution (associated to $\wt\varphi=\varphi(\pi_{\alpha}-\pi_{\alpha}^{\perp})$\,)};
therefore, we shall also say that $\alpha$ is a \emph{uniton for\/ $\Phi$}\,.

As before, let $r \in \nn$.
Let $\Phi$ be a polynomial extended solution of degree at most $r$; set $W = \Phi\H_+$\,.
By a \emph{partial uniton factorization of\/ $\Phi$} we mean a
product
\be{Phi-fact}
\Phi= \Phi_0(\pi_{\alpha_1}+\lambda\pi_{\alpha_1}^{\perp})\cdots(\pi_{\alpha_r}+\lambda\pi_{\alpha_r}^{\perp})
\ee
where (i) $\Phi_0$ is an extended solution, and (ii) writing
\be{Phi_i}
\Phi_i= \Phi_0(\pi_{\alpha_1}+\lambda\pi_{\alpha_1}^{\perp})\cdots
(\pi_{\alpha_i} +\lambda\pi_{\alpha_i}^{\perp}) \qquad (i = 1,2,\ldots, r),
\ee
each $\Phi_i$ is an extended solution, and $\alpha_i$ is a uniton for
$\Phi_{i-1}$ equivalently $\alpha_i^{\perp}$ is a uniton for $\Phi_i$.

Note that each $\varphi_i = (\Phi_i)_{-1}$ is harmonic and condition (ii) can be phrased as follows: \emph{$\alpha_i$ is a uniton for $\varphi_{i-1}$}, equivalently,
\emph{$\alpha_i^{\perp}$ is a uniton for $\varphi_i$}.
 
Work of Segal \cite{segal} implies that setting $W_i = \Phi_i\H_+$ defines an equivalence between partial
uniton factorizations \eqref{Phi-fact} and filtrations 
\be{lambda-filt}
W= W_r\subset W_{r-1}\subset\dots\subset W_0 \subset \HH_+
\ee
by extended solutions satisfying
\be{filt-conds}
\lambda W_{i-1}\subset W_i \subset W_{i-1} \qquad (i=1,2,\ldots,r).
\ee
If $\Phi_0 = I$, equivalently, $W_0 = \HH_+$, then \eqref{Phi-fact} is a \emph{uniton factorization} in the sense of Uhlenbeck \cite{uhlenbeck}.
The argument in \cite[\S 2]{unitons} extends immediately to partial factorizations to show that the unitons in \eqref{Phi-fact} are given by 
$\alpha_i=P_0\Phi_{i-1}^{\:\:-1}W_i$\,.

\subsection{$J_2$-holomorphic lifts from unitons}  \label{subsec:lifts-unitons}
Again, let $\Phi$ be a polynomial extended solution of degree at most $r$
and set $W =\Phi\H_+$\,.  Let $\varphi = \Phi_{-1}$ be the corresponding harmonic map. 
Given a (partial) uniton factorization \eqref{Phi-fact},
let $W_i = \Phi_i\H_+$ be the associated filtration and set
\be{Y-from-W}
Y_i = \lambda^i W_{r-i} + \lambda W \quad (i=0,1,\ldots, r)\,, \qquad Y_{r+1} = \lambda W.  
\ee
Then we have a filtration \eqref{Y-filt} of length $r$.  We ask under what conditions it forms an $F$-filtration (see Definition \ref{def:F-filtr}).

\begin{proposition} \label{pr:Y-from-W}
Let $\Phi$ be a polynomial extended solution of degree at most $r$ and \eqref{Phi-fact} a partial uniton factorization.

{\rm (i)} Suppose that
\be{F-basic}
F \bigl(\Gamma(W_i)\bigr) \subset \Gamma(\lambda W_{i-1}) \qquad (i=1,2,\ldots, r);
\ee
then the filtration \eqref{Y-from-W} is an $F$-filtration. 

{\rm (ii)} The condition \eqref{F-basic} holds if and only if, for each $i \in \{1,2,\ldots, r\}$, $\alpha_i$ is a basic uniton for
$\Phi_{i-1}$\,, equivalently,  $\alpha_i^{\perp}$ is an antibasic uniton for $\Phi_i$.

{\rm (iii)} Let $\alpha_1,\ldots,\alpha_r$ be the unitons in \eqref{Phi-fact}.
The composition $\pi_{\alpha_r}^{\perp} \circ \cdots \circ \pi_{\alpha_{r-i+1}}^{\perp}$ is a holomorphic endomorphism from $(\CC^n, D^{\varphi_{r-i}}_{\zbar})$ to $(\CC^n, D^{\varphi}_{\zbar})$, and
the $A_z^{\varphi}$-filtration $(Z_i)$ associated to $(Y_i)$ via \eqref{Z-from-Y} is given  by
\be{Z-from-unitons}
Z_i = \Ima(\pi_{\alpha_r}^{\perp} \circ \cdots \circ \pi_{\alpha_{r-i+1}}^{\perp})\,,
\quad
\text{equivalently},
\quad
Z_i^{\perp} = \ker(\pi_{\alpha_{r-i+1}}^{\perp} \circ \cdots \circ \pi_{\alpha_r}^{\perp})
\ee
$(i = 1,2,\ldots,r)$.
In particular, $Z_1^{\perp} = \alpha_r$ and $Z_2^{\perp} = \alpha_r + \alpha_r^{\perp} \cap \alpha_{r-1}\,.$
\end{proposition}

\begin{proof} (i)
{}From \eqref{F-basic}, we deduce that, for $i \in \{0,1,\ldots, r-1\}$,
$$
F \bigl(\Gamma(Y_i)\bigr) = F \bigl(\Gamma(\lambda^i W_{r-i} + \lambda W) \bigr) \subset 
	\Gamma(\lambda^{i+1} W_{r-i-1} + \lambda W) = \Gamma(Y_{i+1}).
$$
Further, $F \bigl(\Gamma(Y_r)\bigr) \subset \Gamma(\lambda^{r+1}\HH_+ )
	\subset \Gamma(\lambda W)$.

(ii) This follows from the correspondence of the operators $F$ and
$-A^{\varphi}_z$, explained in \S \ref{subsec:ext-sol}, more precisely it is
\cite[Lemma 3.11]{unitons} applied to $\varphi = \varphi_i$\,.

(iii)  Using \eqref{Y-from-W} and noting that $P_0\circ\Phi^{-1}(\lambda W) = 0$ and filling out zeros, we have
\begin{align*}
Z_i = P_0 \circ \Phi^{-1}(Y_i)
&= P_0 (\pi_{\alpha_r} + \lambda^{-1}\pi_{\alpha_r}^{\perp}) \cdots
	(\pi_{\alpha_{r-i+1}} + \lambda^{-1}\pi_{\alpha_{r-i+1}^{\perp}})
		\Phi_{r-i}^{-1}(\lambda^i W_{r-i})
\\		
&= P_0 \lambda^i (\pi_{\alpha_r} + \lambda^{-1}\pi_{\alpha_r}^{\perp}) \cdots
	(\pi_{\alpha_{r-i+1}} + \lambda^{-1}\pi_{\alpha_{r-i+1}}^{\perp}) \HH_+
\\
&= \pi_{\alpha_r}^{\perp} \circ \cdots \circ \pi_{\alpha_{r-i+1}}^{\perp}(\CC^n).
\end{align*}
This gives the first formula of \eqref{Z-from-unitons}; taking the adjoint gives the second one.
\end{proof}

\begin{corollary} \label{co:lift-from-unitons}
Let $W = \Phi\H_+$ be a polynomial extended solution of degree $r$ and let
\eqref{Phi-fact} be a partial uniton factorization of\/ $\Phi$ with corresponding filtration
\eqref{lambda-filt} which satisfies \eqref{F-basic}.
As usual, define $Y_i$ by \eqref{Y-from-W}, set $Z_i = P_0 \circ \Phi^{-1}Y_i$
and write $\psi_i = Z_i \ominus Z_{i+1}$ \ $(i=0,1,\,\ldots, r)$.
\begin{enumerate}
\item[(i)] The $Z_i$ are given in terms of the unitons in \eqref{Phi-fact}
by \eqref{Z-from-unitons}$;$ in particular,
$\psi_0 = \alpha_r$ and $\psi_1 = \alpha_r^{\perp} \cap \alpha_{r-1}$\,.
\item[(ii)] Suppose that $\Phi$ is $\nu$-invariant and that $(Z_i)$ is a strict alternating filtration.  Set $d_i = \rank\psi_i$\,.  Then
$\psi = (\psi_0, \psi_1, \ldots, \psi_r): M \to F_{d_0,d_1,\ldots, d_r}$ is a
$J_2$-holomorphic twistor lift of\/ $\varphi = \Phi_{-1}:M \to G_*(\cn^n)$\,. 
\end{enumerate}
\end{corollary}

If \eqref{F-basic} does not hold, then we cannot expect the filtration \eqref{Y-from-W} to be an $F$-filtration, as shown by the following example.

\begin{example} \label{ex:segal}
Let $\Phi$ be a polynomial extended solution of degree $r$ and set $W = \Phi\H_+$\,.
Set $W_i = W + \lambda^i\H_+$\,.
Then $(W_i)$ defines a filtration \eqref{lambda-filt} satisfying
\eqref{filt-conds}, and so a uniton factorization \eqref{Phi-fact}; we call this
the \emph{Segal factorization of\/ $\Phi$ (or $W$)}, as it appears in \cite{segal}.
Defining $\Phi_i$ by \eqref{Phi_i}, we have $W_i = \Phi_i\H_+$\,.
Each $\alpha_i$ appearing in \eqref{Phi-fact} is a uniton for $\Phi_{i-1}$; we call
the $\alpha_i$ the \emph{Segal unitons of\/ $\Phi$}.

The resulting filtration \eqref{Y-from-W}
is not, in general, an $F$-filtration; indeed, the Segal unitons are not basic in general, in fact, each $\alpha_i$ is antibasic for $\Phi_{i-1}$\,.
\end{example}

However, there are lots of uniton factorizations with basic unitons to which we can apply Proposition \ref{pr:Y-from-W}(i) to give
$F$-filtrations.  We start with the factorization which gives the canonical twistor lift; then we identify the one which leads to Burstall's twistor lift.

\begin{example} \label{ex:can-filt2}
Let $\Phi$ be a polynomial extended solution of degree at most $r$ and set $W = \Phi\H_+$\,.
Set $W_i = \lambda^{r-i}W \cap \HH_+$.  Again, $(W_i)$ defines a filtration \eqref{lambda-filt} satisfying \eqref{filt-conds}, and so a uniton factorization \eqref{Phi-fact}; we call this
the \emph{Uhlenbeck factorization of\/ $\Phi$ (or $W$)}, as it appears in \cite{uhlenbeck}.
Again, defining $\Phi_i$ by \eqref{Phi_i}, we have $W_i = \Phi_i\H_+$ and
each $\alpha_i$ appearing in \eqref{Phi-fact} is a uniton for $\Phi_{i-1}$; we call
the $\alpha_i$ the \emph{Uhlenbeck unitons of\/ $\Phi$}.

This time, the filtration $(W_i)$ clearly satisfies \eqref{F-basic}, equivalently, each Uhlenbeck uniton $\alpha_i$ is basic for $\Phi_{i-1}$.
It is quickly checked that the $F$-filtration $(Y_i)$ associated to $(W_i)$ by \eqref{Y-from-W} is the \emph{canonical $F$-filtration} \eqref{can-filt} which leads to the canonical twistor lift of Theorem \ref{th:can-lift}.
\end{example}

\begin{example} \label{ex:fact-images}
Let $\Phi$ be a polynomial extended solution of degree $r$, and set $W = \Phi\H_+$\,.
For any $i \in \nn$, let $W_{(i)}$
denote the $i$th osculating space spanned by derivatives of local holomorphic sections
of $W$ up to order $i$.  Setting $W_i = W_{(r-i)}$
defines a partial uniton factorization:
\be{pa-Az-images}
W = W_{(0)} \subset W_{(1)} \subset \cdots \subset W_{(r)} \subset \H_+
\ee
which satisfies \eqref{F-basic}.
The proof in \cite[Example 4.7]{unitons} extends immediately to show that the unitons $\alpha_i$ in  
\eqref{Phi-fact} are given by $\alpha_i^{\perp} = \Ima A_z^{\varphi_i}$;
by definition, $\alpha_i^{\perp}$ is antibasic for $\varphi_i$ so $\alpha_i$ is basic
for $\varphi_{i-1}$.
Defining $Y_i$ by \eqref{Y-from-W} gives the $F$-filtration $Y_i = F(Y_{i-1}) + \lambda W
= F^i(W) + \lambda W$ of Example \ref{ex:nilconf-Y}; note that this formula automatically
gives $Y_{r+1} = \lambda W$ since $F^{r+1}(W) \subset \lambda^{r+1}\HH_+ \subset \lambda W$.
The associated filtration $(Z_i)$ defined by \eqref{Z-from-Y} is the filtration
$Z_i = \Ima(A^{\varphi}_z)^i$ of Example \ref{ex:burstall-J2} which leads to Burstall's twistor lift.

Suppose now that there is a $t \in \nn$ such that $(P_0 W)_{(t)} = \CC^n$,
equivalently, $W_{(t)} = \HH_+$; such $t$ exists if and only if\/ $P_0 W$ is full.
Then $\lambda^t\H_+ =  \lambda^t W_{(t)} \subset W$ so that $r \leq t$
and \eqref{pa-Az-images} extends to a \emph{uniton factorization}:
$$
W = W_{(0)} \subset W_{(1)} \subset \cdots \subset W_{(r)} \subset \cdots \subset W_{(t)} = \H_+;
$$
this is the \emph{factorization by $A_z$-images} of \cite{wood-unitary}.  Note, however, that
 the extra terms $W_{(i)}$ \ $(i > r)$ do not lengthen the associated
$F$-filtration since, when $i>r$, we have $\lambda^i W_{(i)} \subset \lambda W$ so that $Z_i =0$.  
\end{example} 

Let $\varphi:M \to \U n$ be a harmonic map.
Let $(Z_i)$ be an $A^{\varphi}_z$-filtration of length $t$.  Then
each subbundle $Z_i$ is a uniton for $\varphi$; further
$Z_1$ is antibasic, and $Z_t$ is basic.  We have the following converse.  In the proof, for any map $f:\CC^n \to \CC^n$
and integer $i >0$ we set $f^{0} =$ identity map and write $f^i$ for the
composition $f^i = f \circ \cdots \circ f$ \ (with $i$ factors); 
further, for any subset $V$ of $\CC^n$, we write
$f^{-i}(V) = (f^i)^{-1}(V) = \{x \in \CC^n : f^i(x) \in V\}$.

\begin{proposition} \label{pr:lift-with-uniton}
Let $\varphi:M \to \U{n}$ be a nilconformal harmonic map, and let $\alpha$ be a uniton for $\varphi$.
Then we can find a strict $A^{\varphi}_z$-filtration $(Z_i)$ with
$Z_k = \alpha$ for some $k$.  
\end{proposition}

\begin{proof} 
Choose $k \in \nn$ such that  $(A^{\varphi}_z)^{-k}(\alpha) = \CC^n$; since $A^{\varphi}_{\zbar}$ is the adjoint of $-A^{\varphi}_z$,
this is equivalent to $(A^{\varphi}_{\zbar})^k(\alpha^{\perp}) = \ul{0}$\,.  Then
set $Z_i = (A^{\varphi}_z)^{i-k}(\alpha)$; equivalently, $Z_i^{\perp} = (A^{\varphi}_{\zbar})^{k-i}(\alpha^{\perp})$.
\end{proof}

Note the duality in these formulae, cf.\  Example \ref{ex:dual}.
Note further that, if $\alpha$ is antibasic, we can take $k=1$ so that $Z_1 = \alpha$; if $\alpha$ is basic, we have $Z_k = \alpha$ and $Z_{k+1} = \ul{0}$.
In the extreme case $\alpha = \CC^n$, the formula gives the filtration $Z_i = \Ima (A^{\varphi}_z)^i$ of Example \ref{ex:burstall-J2}, and 
when $\alpha = \ul{0}$\,, it gives the dual filtration
$Z_i = \ker (A^{\varphi}_z)^{k+1-i}$ of Example \ref{ex:burstall-dual}.  

If $\varphi$ has image in a Grassmannian and $\alpha$ is a uniton for $\varphi$, then, as before, we say
that $\alpha$ \emph{splits (for $\varphi$)} if $\alpha = \alpha \cap \varphi \oplus \alpha \cap \varphi^{\perp}$.
We deduce the following from Proposition \ref{pr:lift-with-uniton}.

\begin{theorem} \label{th:lift-with-uniton}
Let $\varphi:M \to G_*(\cn^n)$ be a nilconformal harmonic map\/  from a surface to a complex Grassmannian, and let $\alpha$ be a uniton for\/ $\varphi$ which splits.
 Then there is a moving flag
$\psi = (\psi_0,\psi_1, \ldots, \psi_t)$ with the uniton given by a sum
$\sum_{j =j_0}^t\psi_j$ of legs $\psi_i$, and
a $J_2$-holomorphic twistor lift $\wt\psi = (\wt\psi_0, \wt\psi_1, \ldots, \wt\psi_s)$
of $\pm\varphi$ with each $\wt\psi$ the sum of some of the legs $\psi_i$.
\end{theorem}

\begin{proof}
Proposition \ref{pr:lift-with-uniton} gives an $A^{\varphi}_z$-filtration $(Z_i)$ with
one of the $Z_i$ equal to $\alpha$.
Since $\alpha$ splits, it is clear that $(Z_i)$ splits.
By Proposition \ref{pr:J2-from-splits}, we obtain a twistor lift $\psi$ as described.
\end{proof}

\subsection{$S^1$-invariant maps and superhorizontal lifts} \label{subsec:superhor} 
We now consider an important special case of the above constructions when the twistor lifts are holomorphic with respect to both the non-integrable almost complex structure $J_2$
and the integrable complex structure $J_1$ of \S \ref{subsec:twistor-Grass}. 

Let $\Phi$ be a polynomial extended solution; denote its degree by $r$.
Then $\Phi$ is called \emph{$S^1$-invariant} if $\Phi_{\lambda}\Phi_{\mu} = \Phi_{\lambda\mu}$.
This clearly implies that $\Phi$ is $\nu$-invariant so that $\varphi=\Phi_{-1}$ is a harmonic map to a Grassmannian.
Uhlenbeck showed \cite[\S 10]{uhlenbeck} that \emph{$\Phi$ is $S^1$-invariant
if and only its Uhlenbeck unitons $\gamma_1, \ldots, \gamma_r$ (Example \ref{ex:can-filt2})
are nested}: $\gamma_i \supset \gamma_{i+1}$.
In that case, the formula \eqref{Z-from-unitons} reduces to $Z_i = \gamma_{r+1-i}^{\perp}$\,.
Equally well, it follows from \cite[Proposition 3.14]{unitons} that
 $\Phi$ is $S^1$-invariant if and only if its Segal unitons $\beta_1,\ldots, \beta_r$
(Example \ref{ex:segal}) are nested:
\be{nested-seq}
\ul{0} =\beta_0 \subset \beta_1 \subset \cdots \subset \beta_r \subset \beta_{r+1} = \CC^n,
\ee
in which case $\beta_i = \gamma_{r+1-i}$\,; it follows that $\Phi$ and $W = \Phi\H_+$ are given by
\be{W-S1-invt}
\Phi = \sum_{i=0}^r \lambda^i \pi_{\psi_i}	\quad \text{and} \quad
W = \sum_{i=0}^{r-1} \lambda^i \beta_{i+1} + \lambda^r\HH_+,
\ee
giving a harmonic map 
$\varphi = \Phi_{-1}  = \sum_{k=0}^{[r/2]}\psi_{2k}$
where $\psi_i = \beta_{i+1} \ominus \beta_i$.   

A nested sequence \eqref{nested-seq} of subbundles of $\CC^n$ is called \emph{superhorizontal} (cf.\ \cite{burstall-guest}) if (i) each 
subbundle is holomorphic with respect to the standard (product) holomorphic structure, i.e., $\partial_{\zbar}$ maps sections of
$\beta_i$ into sections of $\beta_i$,
 and (ii) $\partial_z$ maps sections of $\beta_i$ into
$\beta_{i+1}$.  Then (see, for example, \cite[Proposition 3.14]{unitons}),
\emph{if\/ $\Phi$ is $S^1$-invariant, the sequence $(\beta_i)$ of its Segal unitons is superhorizontal}.

As usual, set $d_i = \rank\psi_i$; the $d_i$ are all non-zero if and only if $\Phi$
 is normalized, in which case the  canonical $J_2$-holomorphic lift
$\psi = (\psi_0,\psi_1, \ldots, \psi_r):M \to F_{d_0,\ldots, d_r}$
of $\varphi = \Phi_{-1}$ defined by $\Phi$
is given by $\psi_i = Z_i \ominus Z_{i+1} = \beta_{i+1} \ominus \beta_i$.
Superhorizontality of the sequence $(\beta_i)$ can be interpreted as saying that the derivative of $\psi$ lies in
the \emph{superhorizontal distribution},
by which we mean the subbundle of the $(1,0)$-horizontal bundle given by $\sum_{i=0}^{r-1}\Hom(\psi_i,\psi_{i+1})$, in which case
$\psi$ is horizontal and both $J_1$- and $J_2$-holomorphic.
Thus, \emph{let $\Phi$ be a normalized extended solution.  Then the canonical lift of\/ $\varphi = \Phi_{-1}$ defined by $\Phi$
is superhorizontal if and only if\/ $\Phi$ is $S^1$-invariant}.
   
\begin{remark} \label{re:isotropic}
(i) When $r=1$, superhorizontality is automatic; when $r=2$, superhorizontality is equivalent to horizontality.

(ii) Given a harmonic map $\varphi:M \to G_*(\cn^n)$, there is a superhorizontal holomorphic lift of $\varphi$ or $\varphi^{\perp}$ if and only if there
is an $S^1$-invariant polynomial extended solution $\Phi$ with $\Phi_{-1} = \pm\varphi$.  Indeed, given such a lift, $\Phi$ is given by \eqref{W-S1-invt}; conversely, after normalizing $\Phi$, the canonical lift is superhorizontal as explained above.

(iii) A harmonic map $\varphi:M \to G_*(\cn^n)$ which has a superhorizontal holomorphic lift is called \emph{isotropic} in \cite{eschenburg-tribuzy} where it is characterized geometrically, see also \cite{dorfmeister-eschenburg}.
%An isotropic harmonic map is always of finite uniton number.
\end{remark}

\begin{example} \label{ex:holo}
\emph{A superhorizontal sequence of length one} is just a single holomorphic subbundle
$\beta_1 \subset \CC^n$.  The corresponding extended solution \eqref{W-S1-invt} is given by
$\Phi = \pi_{\beta_1} + \lambda \pi_{\beta_1}^{\perp}$ and $W = \beta_1 + \lambda\HH_+$.
Set $d_0 = \rank \beta_1$. Then the resulting harmonic map $\varphi = \Phi_{-1}:M \to G_{d_0}(\cn^n)$ is holomorphic
and $-\Phi_{-1} = \varphi^{\perp}:M \to G_{n-d_0}(\cn^n)$ is antiholomorphic; all holomorphic and antiholomorphic maps
$M \to G_*(\cn^n)$ are obtained in this way.
The canonical lift of $\varphi$ defined by $\Phi$ is $\psi = (\psi_0, \psi_1) = (\beta_1, \beta_1^{\perp}):M \to F_{d_0,n-d_0}$.
The projection $\pi_e: F_{d_0,n-d_0} \to G_{d_0}(\cn^n)$ is given by $(\psi_0, \psi_1) \mapsto \psi_0$ and is bijective.
\end{example}

For the next examples, as before, for any $i \in \nn$ and
holomorphic map $f:M \to G_*(\cn^n)$, equivalently, holomorphic subbundle of $\CC^n$, we denote by $f_{(i)}$
 the $i$th osculating space spanned by derivatives of local holomorphic sections
of $f$ up to order $i$; we set $f_{(-1)} = 0$. 

\begin{example} \label{ex:r2}
(i) A \emph{superhorizontal sequence of length $2$} is a nested pair $\beta_1 \subset \beta_2$ of holomorphic subbundles of $\CC^n$ with $\pa_z\bigl(\Gamma(\beta_1)\bigr) \subset \Gamma(\beta_2)$; such a pair is called a \emph{$\pa'$-pair} in \cite{erdem-wood}.
Equivalently, $(\beta_1, \beta_2^{\perp})$ is a \emph{mixed pair} in the sense of
\cite{burstall-wood} (generalized to subbundles of arbitrary rank), i.e., $\beta_1$  is a holomorphic subbundle of $\CC^n$, $\beta_2^{\perp}$ is an antiholomorphic one, and $\pa_z\bigl(\Gamma(\beta_1)\bigr)$ has values perpendicular to
$\beta_2^{\perp}$.
  The corresponding extended solutions $\Phi$ and $W=\Phi\H_+$ of
\eqref{W-S1-invt} are given by
\be{r2}
\Phi = \pi_{\beta_1} + \lambda\pi_{\varphi} + \lambda^2\pi_{\beta_2}^{\perp}
\quad \text{and} \quad
W = \beta_1 + \lambda\beta_2 + \lambda^2\HH_+\,.
\ee 
The resulting harmonic map $\varphi = \Phi_{-1}$ is given by $\varphi = \beta_1 \oplus \beta_2^{\perp}$; it is also called a
\emph{mixed pair}.  Its orthogonal complement is the harmonic map
$\varphi^{\perp} = \beta_2 \ominus \beta_1$; 
this is \emph{strongly isotropic}
\cite[(1.6)]{erdem-wood} in the sense that the Gauss transforms
(see \S \ref{subsec:Grass})
$G^{(i)}(\varphi^{\perp})$ and $G^{(j)}(\varphi^{\perp})$ are orthogonal for all integers $i \neq j$.
All strongly isotropic harmonic maps $M \to G_*(\cn^n)$ are obtained from a $\pa'$-pair $\beta_1 \subset \beta_2$ in this way \cite[\S 4]{erdem-wood}; indeed,
we may take $\beta_1 = \sum_{i <0} G^{(i)}(\varphi^{\perp})$ and
$\beta_2^{\perp} = \sum_{i >0} G^{(i)}(\varphi^{\perp})$.
The canonical lift of $\varphi$ defined by the extended solution \eqref{r2} is given by
$\psi = (\psi_0,\psi_1,\psi_2) = (\beta_1,\varphi^{\perp},\beta_2^{\perp}): M \to F_{d_0,d_1,d_2}$ where $d_i = \rank\psi_i$.
This is, of course, superhorizontal.

Note that a strongly isotropic map is certainly strongly conformal,
and the lift $\psi$ just defined
is of the type described in Example \ref{ex:str-conf}(iii).  For a map
$\varphi:M \to \CP{n-1} = G_1(\cn^n)$, the notion of strong isotropy reduces to
the notion of (complex) isotropy as used in \cite{eells-wood}.

(ii) Let $f:M \to \CP{n-1}$ be a full holomorphic map and let
$i \in \{0,1,2,\ldots, n-1\}$.  
Setting $\beta_1 = f_{(i-1)}$ and $\beta_2 = f_{(i)}$ gives a $\pa'$-pair
$\beta_1 \subset \beta_2$,
and so a full isotropic harmonic map $\varphi^{\perp}:M \to \CP{n-1}$ given by
$\varphi^{\perp} = G^{(i)}(f)$.  All isotropic harmonic maps $M \to \CP{n-1}$ are given
this way and so all harmonic maps from $S^2$ to $\CP{n-1}$, see \cite{eells-wood};  holomorphic and antiholomorphic maps are given by the extreme cases $i=0$ and $i=n-1$, respectively. Excluding those cases, as in part (i),
the canonical lift of $\varphi$ defined by the extended solution \eqref{r2} is
$\psi = (\beta_1,\varphi^{\perp},\beta_2^{\perp})$; this is a superhorizontal holomorphic lift of $\varphi$.
Note that $\varphi^{\perp}$ is strongly conformal, but $\varphi$ is not. In fact
$\varphi^{\perp}:M \to \CP{n-1}$ can have no twistor lift $\psi$ to a flag manifold; indeed,
such a twistor lift $\psi = (\psi_0,\psi_1,\psi_2, \ldots)$ would have to have at least three legs, but then $\varphi^{\perp} = \pi \circ \psi$
would contain $\psi_0 \oplus \psi_2$, which has rank at least two.

(iii) Let $f:M \to \CP{n-1}$ be a full holomorphic map and let $i \in \{2,3, \ldots n-3\}$.
Setting $\beta_1 = f_{(i-2)}$ and $\beta_2 = f_{(i)}$ gives a $\pa'$-pair
$\beta_1 \subset \beta_2$ and so a strongly isotropic harmonic map
$\varphi^{\perp} = G^{(i-1)}(f) \oplus G^{(i)}(f)$; such a harmonic map is called a \emph{Frenet pair} \cite{burstall-wood}.
The canonical  lift of $\varphi$ defined by \eqref{W-S1-invt} is again $\psi = (\beta_1,\varphi^{\perp},\beta_2^{\perp})$.

In contrast to part (ii),  $\varphi$ \emph{is} strongly conformal, and
Example \ref{ex:str-conf}(iv) provides the unique $J_2$-holomorphic lift of
$\varphi^{\perp}$ with three legs:
$\psi = \bigl(G''(\varphi), \varphi, G'(\varphi)\bigr)
= \bigl(G^{(i)}(f), \varphi, G^{(i-1)}(f)\bigr)$.
\end{example}

\section{Twistor lifts of maps of surfaces to real Grassmannians and $\O{2m}/\U{m}$} \label{sec:real}
\subsection{Twistor spaces} \label{subsec:twistor-real}
We now consider the symmetric spaces of the \emph{orthogonal group} $\O{n}$ and
its identity component, the \emph{special orthogonal group} $\SO{n}$.
We think of $\O{n}$ as the totally geodesic submanifold
$\{g \in \U{n} : g = \ov g\}$ of $\U{n}$.   For each $k$, 
the \emph{real Grassmannian} $G_k(\rn^n) = \O{n}\big/ \O{k} \times \O{n-k}$
is a symmetric space; it may be thought of as the  totally geodesic submanifold
$\{ V \in G_k(\cn^n): V = \ov{V}\}$ of $G_k(\cn^n)$, 
and, via the Cartan embedding \eqref{cartan}, as a totally geodesic submanifold
of the orthogonal group $\O{n}$, and so also of $\U{n}$.  
We now identify twistor spaces for the real Grassmannian as subspaces of those for $G_k(\cn^n)$.

Let $n$, $d_0,d_1,\ldots, d_s$ be positive integers with $d_s + 2\sum_{i=0}^{s-1}d_i = n$.
Set $d_i = d_{2s-i}$ for $i = s+1, \ldots, 2s$ so that $\sum_{i=0}^{2s} d_i = n$,
and define a submanifold of the flag manifold $F_{d_0,\ldots, d_{2s}}$ of
\S \ref{subsec:twistor-Grass} by
$$
F^{\rn}_{d_0,\ldots, d_s} = \bigl\{\psi = (\psi_0,\psi_1, \ldots, \psi_{2s}) \in F_{d_0,\ldots, d_{2s}} : \psi_i = \ov\psi_{2s-i} \ \forall i \bigr\}
$$
(here, by $\ov\psi_{2s-i}$ we mean the complex conjugate  $\ov{\psi_{2s-i}}$).  Note that the middle leg $\psi_s$ is real, i.e.,
$\ov\psi_s = \psi_s$\,.  Further note that $F^{\rn}_{d_0,\ldots, d_s}$ is a complex submanifold of  $F_{d_0,\ldots, d_{2s}}$ with respect to the complex structures
$J_1$ and $J_2$.
Hence the twistor fibration \eqref{twistor-proj} restricts to a twistor fibration $\pi_e^{\rn}:F^{\rn}_{d_0,\ldots, d_s} \to G_k(\rn^n)$
where, as before, $k = \sum_{j=0}^s d_{2j}$ and  $\pi_e^{\rn}(\psi) = \sum_{j=0}^s \psi_{2j}$.

On using $\psi_i = \ov\psi_{2s-i}$, these can be written in terms of just $(\psi_0,\psi_1,\ldots, \psi_s)$ as follows.
\be{k}
\left\{ \begin{array}{rl}
k = 2\sum_{k=0}^{s/2-1}d_{2k} + d_s \,, \qquad
&\pi_e^{\rn}(\psi) = \sum_{k=0}^{s/2-1}(\psi_{2k}\oplus\ov\psi_{2k}) \oplus \psi_s \qquad \text{($s$ even)};
\\[1ex]
k = 2\sum_{j=0}^{(s-1)/2}\!\!d_{2j}\,, \qquad
&\pi_e^{\rn}(\psi) = \sum_{j=0}^{(s-1)/2}\!(\psi_{2j}\oplus\ov\psi_{2j}) \qquad \text{($s$ odd)}.
\end{array} \right.
\ee

Note that \emph{if $s$ is even, $n-k$ is even and if $s$ is odd, $k$ is even}; further, from
\eqref{k}, we have $s \leq \min(k-1,n-k)$.

As a homogeneous space,
$F^{\rn}_{d_0,\ldots, d_s} = \O{n}/H$ where $H = \U{d_0} \times \cdots \times \U{d_{s-1}} \times \O{d_s}$.
Write $H = H_1 \times H_2$ where $H_1 = \bigl\{ \prod_{j=0}^{s/2-1}\U{d_{2j}} \bigr\} \times \O{d_s}$ if $s$ is even and
$\Pi_{j=0}^{(s-1)/2}\U{d_{2j}}$ if $s$ is odd.  Then
the projection $\pi_e^{\rn}$ is the homogeneous projection
$\O{n}/H \to \O{n} \bigl/\O{k} \times \O{n-k}$ induced by the inclusion of $H = H_1 \times H_2$  in $\O{k} \times \O{n-k}$ given by the canonical inclusions of $H_1$ in $\O{k}$ and $H_2$ in $\O{n-k}$.

We can also write $F^{\rn}_{d_0,\ldots, d_s} = \SO{n}/\wt H$ where
$\wt H = \U{d_0} \times \cdots \times \U{d_{s-1}} \times \SO{d_s}$.
We have a double covering $\wt G_k(\rn^n) \to G_k(\rn^n)$ by the Grassmannian
$\wt G_k(\rn^n)$ of \emph{oriented}
$k$-dimensional subspaces of $\rn^n$ which forgets orientation; the projection $\pi_e^{\rn}$ lifts to a projection
$\wt\pi_e^{\rn}: F^{\rn}_{d_0,\ldots, d_s}
= \SO{n}/\wt H  \to \SO{n} \bigl/ \SO{k} \times \SO{n-k} = \wt G_k(\rn^n)$, 
providing twistor spaces for $\wt G_k(\rn^n)$.  Composing the double covering $\wt G_k(\rn^n) \to G_k(\rn^n)$
with the inclusion
map $G_k(\rn^n) \hookrightarrow G_k(\cn^n)$ gives a canonical totally geodesic isometric immersion of $\wt G_k(\rn^n)$ into $G_k(\cn^n)$.
 
 \begin{example}
For any $m \in \{1,2,\ldots,\}$, the mapping $(V,Y, \ov V) \mapsto V$ identifies the space
$F_{m,1}^{\rn}=\O{2m+1} \big/\U m\times\O 1\cong\SO{2m+1}/\U m$ with the space of isotropic subspaces of $\cn^{2m+1}$ of dimension $m$; the bundle $\pi_e^{\rn}: F_{m,1}^{\rn} \to \RP{2m}$ with $\pi_e(V,Y, \ov V) = Y$ can be identified with the bundle
$Z \to \RP{2m}$ of almost Hermitian structures (cf.\ Remark \ref{rem:aH-strs}(ii)), and it lifts to a fibre bundle
$\wt\pi_e^{\rn}:F_{m,1}^{\rn} \to S^{2m}$ which is the bundle $Z^+ \to S^{2m}$ of \emph{positive} almost Hermitian structures on $S^{2m}$.   In particular, 
the double covering $\Sp 2\to\SO 5$ maps $\U 1\times\Sp 1$ to $\U 2$,  showing that $F_{2,1}^{\rn}\cong\Sp 2/\U 1\times\Sp 1\cong\CP{3}$, hence the fibration $F_{2,1}^{\rn} \to S^4$ is the classical twistor fibration  $\CP{3} \to S^4$.
\end{example}

The orthogonal group has another symmetric space, $\O{2m} \big/ \U{m}$, the space  of \emph{orthogonal complex structures}; this has identity component, $\SO{2m} \big/ \U{m}$,
\emph{the space of positive orthogonal complex structures}.
By mapping a complex structure to its $(-\ii)$-eigenspace,
$\O{2m} \big/ \U{m}$ may be identified with the totally geodesic submanifold
$S = \{ V \in G_m(\cn^{2m}) : V = \ov{V}^{\perp}\}$.   This in turn may be identified 
via the Cartan embedding, with a totally geodesic submanifold of
$\ii\O{n}$ where $\ii\O{n}$ denotes the totally geodesic submanifold
$\{ \ii g: g \in \O{n}\} = \{g \in \U{n}: g = -\ov{g}\}$ of $\U{n}$.

We obtain twistor spaces for $\O{2m} \big/ \U{m}$ as restrictions of those for
$G_m(\cn^{2m})$ as follows.
Let $m, d_0,d_1, \ldots,d_s$ be positive integers with $m=d_0+\dots+d_s$ and set $d_i=d_{2s+1-i}$ for $i=s+1,\dots,2s+1$. Let $\Z^{\rn}_{d_0, \ldots,d_s}$ be the submanifold of $F_{d_0, \ldots,d_{2s+1}}$ given by 
$$
\Z^{\rn}_{d_0, \ldots,d_s}=\big\{\psi=(\psi_0,\psi_1, \ldots,\psi_{2s+1})\in F_{d_0, \ldots,d_{2s+1}}\ :\ \psi_i=\ov\psi_{2s+1-i}\ \forall i\big\}.
$$
This is a complex submanifold with respect to the almost complex structures $J_1$ and $J_2$ of $F_{d_0, \ldots,d_{2s+1}}$, so the projection \eqref{twistor-proj} restricts to
a twistor fibration $\pi_e^{\rn}:\Z^{\rn}_{d_0, \ldots,d_s}\to\O{2m}/\U{m}$; indeed, writing $\varphi = \pi_e(\psi)$, then for any
$\psi \in \Z^{\rn}_{d_0, \ldots,d_s}$, we see that $\varphi = \sum_i\psi_{2i} = \ov\varphi^{\perp}$, so $\varphi \in S \cong \O{2m} \big/ \U{m}$.

As a homogeneous space, $\Z^{\rn}_{d_0, \ldots,d_s}=\O{2m} \big/\U{d_0}\times\dots\times\U{d_s}$,
 and $\pi_e^{\rn}$ is the homogeneous projection 
$\O{2m} \big/\U{d_0}\times\dots\times\U{d_s} \to \O{2m}/\U{m}$ given by
the canonical inclusion of the product $\U{d_0}\times\dots\times\U{d_s}$ in $\U{m}$.
This restricts to a twistor fibration $\wt\pi_e^{\rn}:\SO{2m} \big/\U{d_0}\times\dots\times\U{d_s} \to \SO{2m}/\U{m}$.

\subsection{Some involutions} \label{subsec:involutions}

We describe some involutions on our various types of filtrations; real cases will then appear as their fixed points.  For a map $\varphi: M \to \U n$, $\ov\varphi$ will denote its complex conjugate; thus $\varphi = \ov\varphi$ (resp.\ $\varphi = -\ov\varphi$) if and only if $\varphi$ is real, i.e., has image in $\O n$ (resp.\ $\varphi$ has image in $\ii\O{n}$, equivalently $\ii\varphi$ is real).

\begin{lemma} \label{le:duality-Az}
  Let $\varphi:M \to \U n$ be a nilconformal harmonic map.

{\rm (i)}  Let $(Z_i)$ be an $A_z^{\varphi}$-filtration of length $t$; denote its legs by $\psi_i = Z_i \ominus Z_{i+1}$. Set
\be{inv-Z}
\wt Z_i= \ov Z_{t+1-i}^{\perp} \quad (i=0,1,\ldots,t+1).
\ee
Then $(\wt Z_i)$ is an $A_z^{\ov\varphi}$-filtration of the same length with legs
$\wt\psi_i = \wt Z_i \ominus \wt Z_{i+1}$ given by $\wt\psi_i = \ov\psi_{t-i}$\,.

Further, if $\varphi:M \to G_*(\cn^n)$ and $(Z_i)$ is split (resp.\ alternating) for $\varphi$, then so is $(\wt Z_i)$.

{\rm (ii)} If\/ $\varphi$ or\/ $\ii\varphi$ is real, then $(Z_i) \mapsto (\wt Z_i)$ defines
an involution on the set of\/ $A_z^{\varphi}$-filtrations.
\end{lemma}

\begin{proof}
Since the adjoint of $A_z^{\varphi}$ is $-A_{\zbar}^{\varphi}$, the condition $A_z^{\varphi}(Z_i) \subset Z_{i+1}$ is equivalent to
$A_{\zbar}^{\varphi}(Z_{i+1}^{\perp}) \subset Z_i^{\perp}$, and this is equivalent to $A_z^{\ov\varphi}(\wt Z_{t-i}) \subset \wt Z_{t+1-i}$\,.
Similarly $D^{\varphi}_{\zbar}\bigl(\Gamma(Z_i)\bigr) \subset \Gamma(Z_i)$ is equivalent to  $D^{\varphi}_{\zbar}\bigl(\Gamma(\wt Z_{t+1-i})\bigr) \subset \Gamma(\wt Z_{t+1-i})$\,;
the rest is clear.
\end{proof}

\begin{example} \label{ex:dual}
(i) If\/ $Z_i = \Ima (A^{\varphi}_z)^i$ as in Example \ref{ex:burstall-J2}, then
$\wt Z_i = \ker (A^{\ov\varphi}_z)^{t+1-i}$.
Replacing $\varphi$ by $\ov\varphi$ gives Example \ref{ex:burstall-dual}.  A similar process gives Example \ref{ex:str-conf}(ii) from
Example \ref{ex:str-conf}(i). 

(ii) An equivalent conclusion to part (i) of the lemma is that $\wh Z_i = Z_{t+1-i}^{\perp} $ defines an $A_z^{\varphi}$-filtration
\emph{with respect to the conjugate complex structure on $M$}.
\end{example}

We have a corresponding involution of $F$-filtrations as follows.
For a map $\Phi:M \to \Omega\U{n}$, set $W = \Phi\H_+$ and $\varphi = \Phi_{-1}$\,.
For a fixed integer $r$, write $W^I = \lambda^{r-1}\ov W^{\perp}$, then $W^I= \wt\Phi\H_+$ where $\wt\Phi = \lambda^r\ov\Phi$ (cf.\ \cite[Remark 2.7]{unitons}).
If $W$ is an extended solution, it is easily checked that $W^I$ is also an extended solution.  Writing $\wt\varphi = \wt\Phi_{-1}$, we have $\wt\varphi = (-1)^r\varphi$.

\begin{definition} \label{def:real}   Let $r \in \zn$.  We call $W$ or $\Phi$
\emph{real of degree $r$} if\/ $W^I = W$, equivalently $\Phi = \lambda^r \ov\Phi$. 
\end{definition}

 If $r=2s$ is even, this says that $\lambda^{-s}\Phi$ has values in $\Omega\O n$, and implies that $\varphi$ is real, i.e., has values in
$\O n$.  When $r$ is odd, it implies that $\ii\varphi$ is real, see \S \ref{subsec:OCS}
for the application of that case. 
 
\begin{lemma}\label{le:inv-Y}
Let $\Phi:M \to \Omega\U n$ be a polynomial extended solution of degree $r$
with $\Phi_{-1} = \varphi$.  Set $W=\Phi\H_+$\,.
Let $(Y_i)$ be an $F$-filtration of\/ $W;$ denote its length by $t$.  Set
\be{inv-Y}
\wt Y_i =  \lambda^r  \ov Y_{t+1-i}^{\perp} \quad (i=0,1,\ldots,t+1).
\ee
Then, with the projection $P_0$ defined as in \eqref{fourier-coeffs},
\begin{enumerate}
\item[(i)] $(\wt Y_i)$ is an $F$-filtration of\/ $W^I;$\\[-2ex]

\item[(ii)] $\wt{\wt Y}_i = Y_i$ for all $i;$

\item[(iii)] set $Z_i = P_0\circ\Phi^{-1}(Y_i)$ and
$\wt Z_i = P_0\circ\wt\Phi^{-1}(\wt Y_i)$. Then $Z_i$ and
$\wt Z_i$ are related by \eqref{inv-Z}$;$

\item[(iv)]  if\/ $\Phi$ is real of degree $r$, then $(Y_i)\mapsto(\wt Y_i)$ defines an involution on the set of\/ $F$-filtrations of\/ $W$. 
\end{enumerate}
\end{lemma}

\begin{proof} 
(i) The condition $F(\Gamma(Y_i)) \subset \Gamma(Y_{i+1})$ is clearly equivalent to
$F(\Gamma(\ov Y_{i+1}^{\perp})) \subset \Gamma(\ov Y_i^{\perp})$.  Also $\wt Y_0 = \lambda^r \ov Y_{t+1}^{\perp}
= \lambda^{r-1}\ov W^{\perp} = W^I$ and $\wt Y_{t+1} = \lambda^r \ov Y_0^{\perp} = \lambda^r \ov W^{\perp} = \lambda W^I$.

(ii) This is quickly checked.

(iii) Noting that $Z_i = P_0 \circ \Phi^{-1}Y_i$ is equivalent to $\Phi^{-1}Y_i+\lambda\HH_+=Z_i+\lambda\HH_+$, we have
$$
\wt Z_i+\lambda\HH_+=\wt\Phi^{-1}\wt Y_i+\lambda\HH_+ = \ov\Phi^{-1}\ov Y_{t+1-i}^{\perp}+\lambda\HH_+ = \ov{\Phi^{-1}Y_{t+1-i}}^{\,\perp}+\lambda\HH_+= \ov Z_{t+1-i}^{\perp}+\lambda\HH_+\,,
$$
giving the result.

(iv) Immediate from (i). 
\end{proof}

Now let $\Phi$ be a polynomial extended solutions of degree at most $r$.
Then the involution $W \mapsto W^I$ above
gives another polynomial extended solution $\wt\Phi =  \lambda^r\ov\Phi$ of degree
at most $r$, see \cite[Example 3.8]{unitons}.   We now see what this involution does to the canonical filtration.

\begin{proposition}\label{pr:realW-Y}
Let $\Phi$ a polynomial extended solution of degree at most $r$.

{\rm (i)} Let $(Y_i)$ be the canonical $F$-filtration of\/ $W$, and define $(\wt Y_i)$ by \eqref{inv-Y} with $t=r;$ then $(\wt Y_i)$ is
the canonical $F$-filtration of\/ $W^I$.

{\rm (ii)} If\/ $\Phi$ is real of degree $r$, then
{\rm (a)} the canonical $F$-filtration $(Y_i)$  is \emph{real}, i.e., fixed under the involution \eqref{inv-Y}$;$
{\rm (b)} the canonical $A^{\varphi}_z$-filtration $Z_i = P_0 \circ \Phi^{-1}(Y_i)$ is
\emph{real}, i.e., fixed under the involution \eqref{inv-Z}$;$ 
{\rm (c)} the legs $\psi_i = Z_i \ominus Z_{i+1}$ satisfy $\psi_i = \ov\psi_{r-i}$ \ $(i=0,1,\ldots, r)$.
\end{proposition}

\begin{proof}
(i) We have $Y_i = W \cap \lambda^i\HH_ + +  \lambda W$ so that
\begin{align*}
\wt Y_i &= \lambda^r \ov Y_{r-i+1}^{\perp} 
	= \lambda^r \bigl\{ (\ov W^{\perp} +
	\lambda^{i-r}\HH_+) \cap (\lambda^{-1}\ov W^{\perp}) \bigr\} \\
	&=(\lambda W^I + \lambda^i\HH_+) \cap W^I = W^I \cap \lambda^i\HH_ + + \lambda W^I.
\end{align*}

(ii) Immediate from (i).
\end{proof}

\subsection{$J_2$-holomorphic lifts for maps to real Grassmannians} \label{subsec:real-Grass}
To apply our work to harmonic maps into real Grassmannians, we need the following existence result for extended solutions.

\begin{proposition} \label{pr:real-ext-sol}
Let $\Phi: M \to \Omega \U n$ be a $\nu$-invariant polynomial extended solution which is real of some
even degree $r = 2s$.  Then
{\rm (i)} $\varphi = (-1)^s \Phi_{-1}: M \to G_k(\rn^n)$ is a harmonic map of finite uniton number with $n-k$ even$;$
{\rm (ii)} all such harmonic maps $\varphi$ are given this way; in fact, we may take $\Phi$ to be normalized of
degree at most $2\min(k-1,n-k)$, if\/ $s$ even, and $2\min(k,n-k-1)$, if\/ $s$ is odd.
\end{proposition}

\begin{proof}
(i) This is a consequence of the formula for $\varphi$ in Proposition \ref{pr:can-alternate}; the parity of $n-k$ following from the symmetry of the legs as in
Proposition \ref{pr:realW-Y}(ii).

(ii) By \cite[Lemma 6.6]{unitons}, there is a $\nu$-invariant extended solution $\Psi:M \to \Omega \O{n}$ of the form
$\Psi = \sum_{\ell = -s}^s \lambda^{\ell} T_{\ell}$ with $T_{-\ell} = \ov{T_{\ell}}$ \ $\forall \ell$\,, $T_s \neq 0$ and $\Psi_{-1} = \varphi$,
Setting $\Phi = \lambda^s\Psi$ gives a $\nu$-invariant real polynomial extended solution
of degree $2s$ with $\Phi_{-1} = (-1)^s\varphi$.  That we may take $\Phi$ normalized
with the given bounds on the degree follows from \cite[Proposition 6.23]{unitons}.
\end{proof}

\begin{remark} \label{re:odd-even}
The statement (ii) is false without the factor $(-1)^s$, see Example \ref{ex:RPn}(ii) below.
Also, if $n-k$ is odd, we may embed $G_k(\rn^n)$ in $G_k(\rn^{n+1})$, then harmonic maps from $M$ to $G_k(\rn^n)$ are obtained as non-full harmonic maps into $G_k(\rn^{n+1})$.
\end{remark}

\begin{theorem} \label{th:real-twistor-lifts}
Let $\Phi: M \to \Omega \U n$ be a $\nu$-invariant polynomial extended solution
which is normalized and real of even degree $r = 2s$.
Let $\varphi = \Phi_{-1}:M \to G_k(\rn^n)$ be the resulting harmonic map.  Then $\varphi$ has a $J_2$-holomorphic
lift $\psi:M \to F^{\rn}_{d_0,d_1,\ldots,d_s}$ for some $(d_0,d_1,\ldots,d_s)$
satisfying \eqref{k}, namely the canonical twistor lift defined by $\Phi$ (see Theorem \ref{th:can-lift}).
\qed \end{theorem}

On applying Proposition \ref{pr:real-ext-sol}, we obtain the following corollary.

\begin{corollary} \label{co:real-twistor-lifts}
Let $\varphi:M \to G_k(\rn^n)$ be a harmonic map of finite uniton number.  Then either
$\varphi$ or $-\varphi$ ($=\varphi^{\perp}$) has a $J_2$-holomorphic twistor lift
$\psi: M \to F^{\rn}_{d_0,d_1,\ldots,d_s}$ for some $(d_0,d_1,\ldots,d_s)$ satisfying \eqref{k},
namely the canonical twistor lift defined by a normalized extended solution $\Phi$  of $\pm\varphi$.
\qed \end{corollary}

\begin{example} \label{ex:RPn}
(i) Recall from Example \ref{ex:r2} that a $\pa'$-pair  (i.e., superhorizontal sequence of length $2$) $\ul{0} \subset \beta_1 \subset \beta_2 \subset \CC^n$
gives rise to two harmonic maps: the mixed pair $\varphi = \beta_1 \oplus \beta_2^{\perp}$ and the strongly isotropic map
$\varphi^{\perp} = \beta_2 \ominus \beta_1$.  The harmonic maps $\varphi$ and $\varphi^{\perp}$
 are real, i.e., have image in $G_*(\rn^n)$, if and only if\/
$\beta_2^{\perp} = \ov{\beta}_1$, in which case
$(\beta_1,\beta_2^{\perp}) = (\beta_1, \ov\beta_1)$, and the resulting harmonic map
$\varphi = \beta_1 \oplus \ov\beta_1$, is called a
\emph{real mixed pair} \cite{bahy-wood-G2} .
In this case, the canonical lift of $\varphi$ defined by the extended solution \eqref{r2}
is the superhorizontal holomorphic map $\psi = (\beta_1, \varphi^{\perp}, \ov\beta_1):M \to F^{\rn}_{d_0,d_1}$, where $d_0 = \rank\beta_1$ and $d_1 = n-2d_0$.

In the case that $\beta_1$ has rank one, we have $\varphi = \beta_1 \oplus \ov\beta_1:M \to G_2(\rn^n)$.  We may identify the twistor space $F^{\rn}_{1,n-2}$ of $G_2(\rn^n)$
with the quadric
$Q_{n-2} = \{L = [L_1,\ldots, L_n] \in \CP{n-1}: \sum_1^n L_i^{\;2} = 0\}$ via the map
$(\psi_0,\psi_1,\psi_2) \mapsto \psi_0$, then $\pi_e^{\rn}:Q_{n-2} \to G_2(\rn^n)$ is the double cover $L \mapsto L \oplus \ov{L}$.   The canonical lift of $\varphi$ is
the superhorizontal holomorphic map $\beta_1 \cong (\beta_1, \varphi^{\perp},\ov\beta_1):M \to Q_{n-2} \cong F^{\rn}_{1,n-2}$\,. See \cite{bahy-wood-G2} for more information on harmonic maps from a surface to $G_2(\rn^n)$.

(ii)  Let $f:M \to \CP{n-1}$ be a full holomorphic map which is \emph{totally isotropic} \cite{eells-wood} in the sense that
$G^{(n-1)}(f) = \ov f$.  Then $n-1$ is even, say $2m$, and $f_{(m-1)}$ is a maximal isotropic subbundle of $\CC^n$.
Setting $\beta_1 = f_{(m-1)}$ and $\beta_2 = f_{(m)}$, we have
$\beta_2^{\perp} = \ov\beta_1$, so we obtain a real mixed pair
 $\varphi = \beta_1 \oplus \ov\beta_1:M \to G_{2m}(\rn^{2m+1})$ with
$\varphi^{\perp}:M \to \RP{2m}$ a full harmonic map whose composition with the canonical inclusion of $\RP{2m}$ in $\CP{2m}$ is isotropic.  E.\ Calabi and
S.-S.\ Chern showed (see \cite{eells-wood}) that all such isotropic harmonic maps from a surface to a real projective space $\RP{n-1}$, in particular, all harmonic maps from the
$2$-sphere, are given this way; all
 harmonic maps from $S^2$ to a sphere $S^{n-1}$ can be obtained as double covers of those maps.     The canonical lift of $\varphi$ defined by \eqref{r2} is the superhorizontal holomorphic map
$\psi = (\beta_1, \varphi^{\perp}, \ov\beta_1): M \to F_{m,1}^{\rn} = \O{2m+1}\big/ (\U{m} \times \O{1}) = \SO{2m+1}/\U{m}$. 
 
As before, $\varphi^{\perp}$ has no twistor lift, as it would have to be the sum of at least two legs. 

(iii) Example \ref{ex:r2} part (iii) does not specialize to give real maps; indeed \emph{there are no real Frenet pairs}
\cite[Prop.\ 5.10]{bahy-wood-G2}.

(iv) Generalizing part (i), let $\ul{0} =\beta_0 \subset \beta_1 \subset \cdots \subset \beta_r \subset \beta_{r+1} = \CC^n$ be a nested sequence of subbundles which is  superhorizontal (see \S \ref{subsec:superhor}).   Recall that such a sequence defines an
$S^1$-invariant extended solution $\Phi$ which is polynomial of degree $r$, and a harmonic map $\varphi = \Phi_{-1}$ given by \eqref{W-S1-invt}.
Say that a nested sequence $(\beta_i)$ is \emph{real (of degree $r$)} if\/
$\beta_i^{\perp} = \ov\beta_{r+1-i}$ for all $i$;
on setting $\psi_i = \beta_{i+1}\ominus\beta_{i}$ \ $(i=0,1,\ldots, r)$, this is equivalent to $\psi_i = \ov\psi_{r-i}$\,.  
\emph{A superhorizontal sequence $(\beta_i)$ is real if and only if the
corresponding extended solution \eqref{W-S1-invt} is real}.
Now suppose that $(\beta_i)$ is real of even degree $r=2s$.  Then $\varphi = \sum_j \psi_{2j}$ is a map into $G_*(\rn^n)$ and
the canonical $J_2$-holomorphic twistor lift of $\varphi:M \to G_*(\rn^n)$ defined by $\Phi$ is
$\psi = (\psi_0,\psi_1,\ldots, \psi_r):M \to F^{\rn}_{d_0,\ldots, d_r}$.
As in \S \ref{subsec:superhor}, it is superhorizontal. 
\end{example}

We now see how to obtain twistor lifts for real nilconformal maps by a method which extends that which gave the
Burstall lift of Example \ref{ex:burstall-J2}.
It is easy to check (in fact, it is a special case of Lemma \ref{le:duality-Az}(ii))
that, \emph{if\/ $\varphi$ or\/ $\ii\varphi$ is real and $\alpha$ is a uniton for $\varphi$, so is $\ov\alpha^{\perp}$}.
Note that $\alpha$ is isotropic if and only if $\ov\alpha^{\perp} \subset \alpha$, and
\emph{maximally} isotropic exactly when $\ov\alpha^{\perp} = \alpha$.

\begin{proposition} \label{pr:lift-with-uniton-real}
Let $\varphi:M \to \U{n}$ be a nilconformal harmonic map\/  from a surface which has image in $\O{n}$ or $\ii\O{n}$, and let $\alpha$ be an isotropic uniton for $\varphi$ (possibly identically zero).
Then
\begin{enumerate}
\item[(i)] there is a real strict $A^{\varphi}_z$-filtration $(Z_i)$ with
$Z_i= \alpha$ for some $i$.

\item[(ii)] If\/ $\varphi$ maps into a real Grassmannian or into $\O{2m}\big/\U{m}$, and $\alpha$ splits, then we can find a filtration $(Z_i)$ as described in part {\rm (i)} which splits.
\end{enumerate}
\end{proposition}

\begin{proof}
(i) Let $\alpha = Z_0 \supset Z_1 \supset \cdots \supset Z_s \supset Z_{s+1} = \ul{0}$
be a strict \emph{partial $A^{\varphi}_z$-filtration}, i.e., a strict filtration satisfying
conditions (i) and (ii) of Definition \ref{def:Az-filt}; for example $Z_i = (A_z^{\varphi})^i(\alpha)$.
As in Lemma \ref{le:duality-Az}(i),
$\CC^n = \ov Z_{s+1}^{\perp} \supset \ov Z_s^{\perp} \supset \cdots \supset \ov Z_1^{\perp} \supset \ov Z_0^{\perp} = \ov\alpha^{\perp}$
is also a strict partial $A^{\varphi}_z$-filtration.    We can put them together to give a filtration
\be{long-filt}
\CC^n = \ov Z_{s+1}^{\perp} \supset \ov Z_s^{\perp} \supset \cdots \supset \ov Z_1^{\perp}\supset \ov Z_0^{\perp} = \ov\alpha^{\perp} \supset \alpha= Z_0 \supset Z_1 \supset \cdots \supset Z_s \supset Z_{s+1} = \ul{0},
\ee
which is a real strict $A^{\varphi}_z$-filtration except that $\ov\alpha^{\perp}$ might equal $\alpha$, or
$A^{\varphi}_z(\ov\alpha^{\perp})$ might not lie in $\alpha$.

So let $t = t(\alpha) \in \{-1,0,1,\ldots\}$ be the least integer
such that $(A^{\varphi}_z)^{t+1}(\ov{\alpha}^{\perp}) \subset \alpha$; this exists by nilconformality.

(a) If $t=-1$, i.e., $\ov\alpha^{\perp} = \alpha$, remove $\ov\alpha^{\perp}$
from \eqref{long-filt} leaving
a real strict $A^{\varphi}_z$-filtration $(Z_i)$ of length $2s+1$ with middle subbundle
equal to $\alpha$.   

(b) If $t=0$, then $A^{\varphi}_z(\ov{\alpha}^{\perp}) \subset \alpha$ and \eqref{long-filt} is a real strict $A^{\varphi}_z$-filtration of length $2s+2$.

(c) Otherwise, we have $t \geq 1$; set 
$\alpha_1 = (A^{\varphi}_z)^t(\ov{\alpha}^{\perp})  + \alpha$.   Then $\alpha_1$ is a uniton which contains $\alpha$.  Further,
 $\alpha_1$ is isotropic, indeed, for the standard complex symmetric $\cn$-bilinear inner product $\ip{\cdot}{\cdot}_{\cn}$
on $\cn^n$, since $A^{\varphi}_z$ is symmetric,
$$
\ip{\alpha_1}{\alpha_1}_{\cn} = \ip{(A^{\varphi}_z)^t(\ov\alpha^{\perp})} {(A^{\varphi}_z)^t(\ov\alpha^{\perp})}_{\cn} =
\ip{\ov{\alpha}^{\perp}}{(A^{\varphi}_z)^{2t}(\ov{\alpha}^{\perp})}_{\cn}\,;
$$
this is zero since $2t \geq t+1$.

Thus we obtain a filtration:
$\ov{\alpha}^{\perp} \supset \ov\alpha_1^{\perp} \supset \alpha_1 \supset \alpha$ with
 $A^{\varphi}_z(\alpha_1) \subset \alpha$ and
$A^{\varphi}_z(\ov\alpha^{\perp}) \subset \ov\alpha_1^{\perp}$.
Further $(A^{\varphi}_z)^t(\ov\alpha_1^{\perp})  \subset (A^{\varphi}_z)^t(\ov{\alpha}^{\perp})  \subset \alpha_1$, so that $t(\alpha_1) \leq t(\alpha)-1$.

By repeating this construction at most $t$ times we obtain a partial
$A^{\varphi}_z$-filtration:
$$
\ov{\alpha}^{\perp} \supset \ov\alpha_1^{\perp} \supset \ov\alpha_2^{\perp}
\supset \cdots \supset \ov\alpha_j^{\perp}
\supset \alpha_j \supset \cdots \supset \alpha_2 \supset \alpha_1 \supset \alpha\,.
$$
Gluing this into the middle of
\eqref{long-filt} gives a real $A^{\varphi}_z$-filtration.
If $\ov\alpha_j^{\perp} \neq \alpha_j$, this is a real strict $A^{\varphi}_z$-filtration of even length.
If $\ov\alpha_j^{\perp} = \alpha_j$, remove $\ov\alpha_j^{\perp}$, leaving a real strict $A^{\varphi}_z$-filtration of odd length.

(ii) This is clear from the construction.  
\end{proof}

\begin{example}
(i) If\/ $\alpha$ is a \emph{basic} isotropic uniton, then it can be taken to be the last
leg of the filtration, so that $\ov\alpha$ is the first.

(ii) If\/ $t \in \{0,1,\ldots\}$ is the least integer such that $(A^{\varphi}_z)^{t+1} = 0$. Then, for any $(t+1)/2 \leq s \leq t$, $\Ima(A^{\varphi}_z)^s$ is a non-zero isotropic uniton with the last one, $\Ima(A^{\varphi}_z)^t$, basic. 

(iii) If\/ $n = 2m+1$ is odd and
$\alpha$ is an isotropic uniton of rank $m$, then $t(\alpha) = 1$, i.e., $A^{\varphi}_z(\ov\alpha^{\perp}) \subset \alpha$ and we have case (b) above,
so that we obtain a strict
$A^{\varphi}_z$-filtration with $\ov\alpha^{\perp} \supset \alpha$ in the middle.
Indeed, if we had $t(\alpha) >1$, then $A^{\varphi}_z$ would factor to a non-zero map
on the rank one bundle $\ov\alpha^{\perp}/\alpha$ which is not possible by nilconformality.
\end{example}

Our next result concerns the Grassmannian $\wt G_k(\rn^n)$ of oriented
$k$-dimensional subspaces of $\rn^n$ discussed in \S \ref{subsec:twistor-real}.  Recall that we have a double covering
$\wt G_k(\rn^n) \to G_k(\rn^n)$ which forgets the orientation: we call a smooth map
$\varphi:M \to \wt G_k(\rn^n)$ nilconformal if its composition with this double covering is nilconformal. 

\begin{proposition} \label{th:lift-with-uniton-real}
Let $\varphi:M \to \wt G_k(\rn^n)$ be a harmonic map from a surface
 with $k$ or $n-k$ even. 
Then, $\varphi$ or $\varphi^{\perp}$ has a $J_2$-holomorphic lift\/ $\psi:M \to F^{\rn}_{d_0,d_1,\ldots,d_s}$ for some $(d_0,d_1,\ldots,d_s)$ satisfying \eqref{k}
if and only if\/ $\varphi$ is nilconformal.
\qed \end{proposition}

\begin{proof}
Suppose that $\varphi$ or $\varphi^{\perp}$ has a $J_2$-holomorphic lift as stated.  Then $\varphi$ is nilconformal by Corollary \ref{co:converse}.

Conversely, suppose that $\varphi$ is nilconformal.  By replacing $\varphi$ by
$\varphi^{\perp}$, if necessary, we can assume that $n-k$ is even.

(a) \emph{We find a maximal isotropic holomorphic subbundle $W$ of\/ $\varphi^{\perp}$ which is closed under $(A^{\varphi}_z)^2$}.
This is done by a modification of the argument in the proof of Proposition \ref{pr:lift-with-uniton-real}, as follows.
  
Let $\beta$ be an isotropic holomorphic subbundle of $\varphi^{\perp}$ which is closed under
$(A^{\varphi}_z)^2$.  Note that this implies that $\ov\beta^{\perp} \cap \varphi^{\perp}$ is also closed under
$(A^{\varphi}_z)^2$.

\indent\indent (i) We extend $\beta$ to an isotropic holomorphic subbundle of
$\varphi^{\perp}$, again called $\beta$, which satisfies
\be{beta-cond}
(A^{\varphi}_z)^2(\ov\beta^{\perp} \!\cap \varphi^{\perp}) \subset \beta.
\ee
To do this, let $u = u(\beta) \in \{-1,0,1,\ldots\}$ be the least integer such that $(A^{\varphi}_z)^{2u+2}(\ov\beta^{\perp} \!\cap \varphi^{\perp}) \subset \beta$;
this exists by nilconformality.
If $u < 1$, $\beta$ already satisfies \eqref{beta-cond}.
Otherwise $u \geq1$ and we set
$\beta_1 = (A^{\varphi}_z)^{2u}(\ov\beta^{\perp}\!\cap \varphi^{\perp}) +\beta$.  Then it is easy to check that $\beta_1$ is is isotropic, closed under $(A_z^{\varphi})^2$ and has $u(\beta_1 ) < u(\beta)$.    By repeating this construction at most $u$ times
starting with $\beta = 0$ we obtain
an \emph{isotropic holomorphic subbundle $\beta$ of\/ $\varphi^{\perp}$ which
satisfies \eqref{beta-cond}}.

\indent\indent (ii) We extend this $\beta$ to a \emph{maximal isotropic} holomorphic subbundle $W$ of $\varphi^{\perp}$ as 
 in \cite[Theorem 2.5]{burstall-rawnsley} as follows.
First note that the bundle $\varphi^{\perp} \to M$ is oriented since it is the pull-back of the oriented tautological bundle
$L^{\perp} \to \wt G_{n-k}(\rn^n)$.
Then, starting with $X=\beta$, we can successively extend $X$ increasing its rank by one until we obtain an isotropic holomorphic subbundle $X$ with $\rank \ov X^{\perp} - \rank X = 2$. Then, by orientability of the bundle
$\varphi^{\perp} \to M$, there is precisely one positive maximally isotropic subbundle $W$ of $\varphi^{\perp}$ containing $X$.
Since $\ov\beta^{\perp}\!\cap \varphi^{\perp} \supset W \supset \beta$, by \eqref{beta-cond}
we have $(A^{\varphi}_z)^2(W) \subset W$, as desired.

(b) \emph{Let $W$ be a maximal isotropic holomorphic subbundle of $\varphi^{\perp}$ which is closed under $(A^{\varphi}_z)^2$}.
Set $\alpha = W + A^{\varphi}_z(W)$.  Then $\alpha$ is an isotropic uniton.
Note that $\ov\alpha^{\perp} = W + A$ where $A = \ov{A^{\varphi}_z(W)}^{\perp} \cap \varphi$ so that $\ov\alpha^{\perp} \ominus \alpha \subset \varphi$.
{}From $\ip{A^{\varphi}_z(A)}{W}_{\cn} = \ip{A^{\varphi}_z(W)}{A}_{\cn} = 0$, we see that
$A^{\varphi}_z(A) \subset W$ whence
$A^{\varphi}_z (\ov\alpha^{\perp}) \subset \alpha$.  Let $s$ be the least positive integer such that
\be{W-order}
(A^{\varphi}_z)^s(W) = 0, \quad \text{equivalently,} \quad (A^{\varphi}_z)^s(\alpha) = 0.
\ee

Set $Z_i = (A^{\varphi}_z)^{i-s-1}(\alpha)$ \ $(i = s+1,s+2,\ldots, 2s+1)$ and
$Z_i = \ov Z_{2s+1-i}^{\perp}$ \ $(i = 0,1,\ldots,s)$.  Then $(Z_i)$ is an alternating
real $A^{\varphi}_z$-filtration of length $2s$ with $Z_s = \ov\alpha^{\perp}$,
$Z_{s+1} = \alpha$ and $Z_{2s}  = (A^{\varphi}_z)^{s-1}(W) \in \wt\varphi$
where $\wt\varphi = (-1)^s\varphi$.
Setting $\psi_i = Z_i \ominus Z_{i+1}$ defines a moving flag $(\psi_0,\psi_1,\ldots, \psi_{2s})$ which satisfies the $J_2$-holomorphicity condition \eqref{J2-cond}.
Now all $\psi_i$ are non-zero with the possible exception of
$\psi_s = \ov\alpha^{\perp} \ominus \alpha$.  If this is zero, remove it and combine the legs
$\psi_{s-1}$ and $\psi_{s+1}$ as in Operation 3 of Lemma \ref{le:remove-zero}, thus reducing
$s$ by one.  Thus we obtain a $J_2$-holomorphic lift  
$\psi:M \to F^{\rn}_{d_0,\ldots, d_s}$ of $\wt\varphi$ for some $(d_0,\ldots,d_s)$.
\end{proof}

\begin{remark} \label{re:unoriented}
{\rm (i)} Unlike Corollary \ref{co:real-twistor-lifts}, this corollary applies to nilconformal harmonic maps whether they have finite
uniton number or not; however, it is not as explicit, as the proof, in general involves the choice of a holomorphic subbundle
which is maximal isotropic.

{\rm (ii)} Let $\varphi:M \to G_k(\rn^n)$ be a harmonic map from a surface to a
real Grassmannian $G_k(\rn^n)$ of \emph{unoriented} subspaces, with $k$ or $n-k$ even. 
Then \emph{$\varphi$ or $\varphi^{\perp}$ has a $J_2$-holomorphic lift\/ $\psi:M \to F^{\rn}_{d_0,d_1,\ldots,d_s}$, for some $(d_0,d_1,\ldots,d_s)$ satisfying \eqref{k},
if and only if\/ $\varphi$ is nilconformal with corresponding subbundle $\varphi$ orientable, equivalently, the first Steifel--Whitney class $w_1(L)$ of the tautological bundle
$L \to G_k(\rn^n)$ satisfies $\varphi^* w_1(L) = 0$}.
Indeed, under those conditions $\varphi$ lifts to a map from $M$ to $\wt G_k(\rn^n)$, and the theorem applies.
\end{remark}
 
\begin{example}  
(i) The condition  \eqref{W-order} implies that
\be{W-cond-s}
\Ima(A^{\varphi}_z)^s \cap \varphi^{\perp} \subset W \subset \ker(A^{\varphi}_z)^s \cap \varphi^{\perp}.
\ee
which in turn implies that $(A^{\varphi}_z)^{2s}(\wt\varphi^{\perp}) = 0$.
Conversely, if\/ $\varphi$ satisfies this last condition, then we can choose $W$ to satisfy \eqref{W-cond-s}: just do the construction of part (i) of the proof above starting with $\beta = \Ima(A^{\varphi}_z)^s \cap \varphi^{\perp}$.

(ii) Putting $s=1$, we deduce the following. \emph{Let
$\varphi:M \to \wt G_k(\rn^n)$ be a non-constant strongly conformal harmonic map
with $n-k$ even.  Then {\rm (i)} there are maximal isotropic holomorphic subbundles $W$ of\/
$\varphi^\perp$ which satisfy \eqref{W-cond}$;$  {\rm (ii)} for such a $W$,
we have a $J_2$-holomorphic twistor lift\/
$\psi = (\ov W,\varphi,W):M \to F^{\rn}_{m,k}$ of\/ $\varphi^{\perp}$, where $m = (n-k)/2$},
cf.\ Example \ref{ex:str-conf}(iii).

(iii) For $n-k=2$, as in Example \ref{ex:str-conf}(iv), there is only one choice of $W$ in (ii) and, reversing the roles of $\varphi$ and $\varphi^{\perp}$, we obtain the following.
\emph{Let $\varphi:M \to G_2(\rn^n)$ a non-constant harmonic map  with $\varphi^{\perp}$ strongly conformal.  Then $\varphi$ has a unique $J_2$-holomorphic lift
$\psi = (\ov W,  \varphi^{\perp}, W):M \to F^{\rn}_{1,n-2}$}; see also
Corollary \ref{co:G2Cn}.
\end{example}
 
\subsection{$J_2$-holomorphic lifts for maps to the space of orthogonal complex structures}
\label{subsec:OCS}

The analogues of the results of \S \ref{subsec:real-Grass} are as follows; the first following from the results of \S 6.3 and Corollary 6.23(iii) of \cite{unitons}.

\begin{proposition} 
Let $\Phi: M \to \Omega \U{n}$ be a $\nu$-invariant polynomial extended solution
which is real of some
odd degree $r = 2s+1$.  Then
{\rm (i)} $n$ is even, i.e., $n =2m$ for some $m \in \nn$;
{\rm (ii)} $\varphi = \Phi_{-1}: M \to \O{2m}/\U{m}$ is a harmonic map of finite uniton
number$;$
{\rm (iii)} all such harmonic maps $\varphi$ are given this way up to sign; in fact, we may take $\Phi$ to be normalized of
degree at most $2m-3$.
\qed \end{proposition}

\begin{theorem} \label{th:twistor-lifts-OCS}
Let $\Phi: M \to \Omega \U{2m}$ be a $\nu$-invariant polynomial extended solution
which is normalized and real of odd degree $r = 2s+1$.
Let $\varphi = \Phi_{-1}:M \to \O{2m}/\U{m}$ be the resulting harmonic map.  Then $\varphi$ has a $J_2$-holomorphic
lift $\psi:M \to \Z^{\rn}_{d_0,\ldots, d_s}$  for some $(d_0,d_1,\ldots,d_s)$
with $s \leq 2m-3$ and $\sum_{i=0}^s d_i = m$, namely the canonical twistor lift defined by $\Phi$ (see Theorem \ref{th:can-lift}).
\qed \end{theorem}

\begin{corollary} \label{co:twistor-lifts-OCS}
Let $\varphi:M \to \O{2m}\big/\U{m}$ be a harmonic map of finite uniton number.
Then either
$\varphi$ or $-\varphi$ has a $J_2$-holomorphic twistor lift
$\psi: M \to \Z^{\rn}_{d_0,d_1,\ldots,d_s}$ for some $(d_0,d_1,\ldots,d_s)$
with $s \leq 2m-3$ and $\sum_{i=0}^s d_i = m$, namely the canonical twistor lift defined by a normalized extended solution $\Phi$ with $\Phi_{-1} = \pm\varphi$.
\qed \end{corollary}

\begin{example}  \label{ex:superhor-real}
As in Example \ref{ex:RPn}, let $\ul{0} =\beta_0 \subset \beta_1 \subset \cdots \subset \beta_r \subset \beta_{r+1} = \CC^n$ be a real superhorizontal sequence.
Set $\psi_i = \beta_{i+1} \ominus \beta_i$.
If\/ $r$ is odd, say $r=2s+1$, then $\varphi = \Phi_{-1} = \sum_j \psi_{2j}$ is a map into $\O{2m}/\U{m}$.
The map $\psi = (\psi_0,\psi_1,\ldots, \psi_r):M \to\Z^{\rn}_{d_0,\ldots, d_s}$ is thus the canonical twistor lift of $\varphi$ defined by $\Phi$; again,
as in \S \ref{subsec:superhor}, it is superhorizontal.

The construction of real superhorizontal sequences is discussed in \cite[\S 6.4]{unitons}.
\end{example}

As for maps into real Grassmannians, we can actually find lifts for harmonic maps which are not of finite uniton number provided they are nilconformal, as follows.

\begin{proposition} \label{th:lift-OCS}
Let $\varphi:M \to \O{2m}\big/\U{m}$ be a harmonic map.
Then $\varphi$ or $-\varphi$ has a $J_2$-holomorphic twistor lift $\psi:M \to \Z^{\rn}_{d_0, \ldots,d_s}$ for some
$s$ and $d_i$ if and only if\/ $\varphi$ is nilconformal.
\end{proposition} 

\begin{proof}
As before, if there is a twistor lift, then there is an
 $A^{\varphi}_z$-filtration so that $\varphi$ is nilconformal.

Conversely, as in Proposition \ref{pr:lift-with-uniton-real} we can construct a real
$A^{\varphi}_z$-filtration which splits.   Lemma \ref{le:combine}(i) or (ii) then gives a
moving flag which satisfies the $J_2$-holomorphicity condition; reality of the filtration implies that this flag is real.   We can then apply Operations 1 and 2, and Operation 3 \emph{symmetrically} (i.e., if a zero leg $\psi_i$ is removed, so is its conjugate
$\psi_{s-i}$), to remove zero legs whilst preserving reality; once that is done, we are left with a twistor lift $\psi$ as stated.
\end{proof}

\section{Harmonic maps from surfaces to quaternionic spaces}

\subsection{Twistor lifts of maps to quaternionic Grassmannians and $\Sp{m}/\U{m}$} \label{sec:sympl}

The results of the previous section for the orthogonal group hold for the symplectic group $\Sp m$, with a few modifications. We give here some definitions, and refer to
\cite{pacheco-sympl} and \cite[\S 6.8]{unitons} for more results on harmonic maps into $\Sp m$. 

To define the relevant twistor spaces, let $J$ be the conjugate linear endomorphism of $\cn^{2m}\cong\hn^m$ corresponding to left multiplication by the quaternion $j$. Let $d_0,d_1,\dots,d_s$ be positive integers with $d_s+2\sum_{i=0}^{s-1}d_i = m$, and set $d_i=d_{2s-i}$ for $i=s+1,\dots,2s$. Define a submanifold $F^J_{d_0, \ldots,d_s}\subset F_{d_0, \ldots,d_{2s}}$ by 
$$
F^J_{d_0, \ldots,d_s}=\bigl\{\psi=(\psi_0,\psi_1, \ldots,\psi_{2s})\in F_{d_0, \ldots,d_{2s}} : \psi_i=J\psi_{2s-i}\ \forall i \bigr\}.
$$
Note that the middle leg $\psi_s$ is \emph{quaternionic}, i.e., $J\psi_s=\psi_s$. 

Similarly, let $d_0,d_1,\dots,d_s$ be positive integers with $d_0+\dots+d_s = m$, set $d_i=d_{2s+1-i}$ for $i=s,\dots,2s+1$,
and define a submanifold $\Z^J_{d_0, \ldots,d_s}\subset F_{d_0, \ldots,d_{2s+1}}$ by 
$$
\Z^J_{d_0, \ldots,d_s}=\bigl\{\psi=(\psi_0,\psi_1, \ldots,\psi_{2s+1})\in F_{d_0, \ldots,d_{2s+1}} : \psi_i=J\psi_{2s+1-i}\ \forall i \bigr\}.
$$
As in the previous section, the projection \eqref{twistor-proj} restricts to
homogeneous projections $\pi_e^J$ from $F^J_{d_0, \ldots,d_s}$ to the quaternionic
 Grassmannian $G_k(\hn^m) = \Sp{m}\big/\Sp{k} \times \Sp{m-k}$ where
$k = \sum_{i=0}^s d_{2i}$,
and from $\Z^J_{d_0, \ldots,d_s}$ to the space $\Sp m/\U m$ of quaternionic complex structures on $\cn^{2m}$. 

A map $\varphi:M\to\U{2m}$ takes values in the subgroup $\Sp m$ if and only if $J\varphi=\varphi J$.
Let $r$ be an integer; then an extended solution $\Phi$ is said to be
\emph{symplectic (of degree $r$)} if $J\Phi J^{-1}=\lambda^{-r}\Phi$.
Set $W=\Phi\H_+$; then $\Phi$ is symplectic (of degree $r$) if and only if
 $JW^{\perp}=\lambda^{1-r}W$, in which case $W$ is also said to be symplectic
(of degree $r$). On setting $\varphi=\Phi_{-1}$, it follows that $\varphi$
(if $r$ is even) or $\ii\varphi$ (if $r$ is odd) takes values in $\Sp m$.
If $\Phi$ is $\nu$-invariant, then $\varphi$ takes values in an inner symmetric space
of $\Sp m$, more specifically, a quaternionic Grassmannian (if $r$ is even)
or $\Sp m/\U m$ (if $r$ is odd). 

Given a polynomial, $\nu$-invariant, symplectic extended solution $\Phi$, we obtain from the canonical filtration of $W=\Phi\H_+$ a twistor lift $\psi$ of $\varphi=\Phi_{-1}$ with values in either $F^J_{d_0, \ldots,d_s}$ or $\Z^J_{d_0, \ldots,d_s}$ according as $r=2s$
or $r=2s+1$. This is proved in the same way as was done for the orthogonal group in the previous section. 

We obtain similar theorems to those of Sections \ref{subsec:real-Grass} and \ref{subsec:OCS}; we leave the reader to write these down.

\begin{example} \label{ex:superhor}
A superhorizontal sequence $\ul 0=\beta_0\subset\beta_1\subset\cdots\subset\beta_r\subset\beta_{r+1}=\CC^{2m}$ is said to be \emph{symplectic} if $J\beta_i^{\perp}=\beta_{r-i}$ for all $i$. Writing, as before, $\psi_i=\beta_{i+1}\ominus\beta_i$ and
$d_i = \rank \psi_i$, then $\psi=(\psi_0,\psi_1, \ldots,\psi_r)$ satisfies the superhorizontality condition \eqref{J2-cond}; so, if the $d_i$ are all non-zero, $\psi$ is a superhorizontal holomorphic map with values in $F^J_{d_0, \ldots,d_s}$ (if\/ $r=2s$) or $\Z^J_{d_0, \ldots,d_s}$ (if\/ $r =2s+1$). 

Set $\Phi=\sum_{i=0}^r\lambda^i\pi_{\psi_{i}}$; equivalently,
$W = \Phi\H_+ =\sum_{i=0}^{r-1}\lambda^i\beta_i+\lambda^r\HH_+$\,. Then
$\Phi$ is an $S^1$-invariant extended solution, symplectic of degree $r$, and $\Phi_{-1}=\pi_e^J \circ \psi$. Conversely, any symplectic
$S^1$-invariant extended solution is given this way. 

Recall that a full holomorphic map $h:M\to\CP{2m-1}$ is said to be
\emph{totally $J$-isotropic} if\/ $G^{(2m-1)}(h)=Jh$ \cite{bahy-wood-HPn}.
Then setting $\beta_i=h_{(i-1)}$ defines a superhorizontal symplectic sequence
of length $2m-1$;
the superhorizontal holomorphic map $\psi$ takes values in $\Z^J_{1,\dots,1}$ (with $m$ $1$s)
and is a $J_2$-holomorphic twistor lift of a harmonic map into $\Sp m/\U m$. 
\end{example}

\begin{example}\label{ex:cpn-quat} (i) Recall from Example \ref{ex:r2} that a
$\pa'$-pair  (i.e., superhorizontal sequence of length $2$)
$\ul{0} \subset \beta_1 \subset \beta_2 \subset \CC^n$
gives rise to two harmonic maps: the mixed pair $\varphi = \beta_1 \oplus \beta_2^{\perp}$ and the strongly isotropic map
$\varphi^{\perp} = \beta_2 \ominus \beta_1$.
Suppose that $n = 2m$.  Then the harmonic maps $\varphi$ and $\varphi^{\perp}$ are quaternionic, i.e., have image in a quaternionic Grassmannian
 $G_*(\hn^{m})$, if and only if\/ $\beta_2^{\perp} = J \beta_1$, in which case
 $(\beta_1,\beta_2^{\perp}) = (\beta_1, J\beta_1) $ is called a
\emph{quaternionic mixed pair} \cite{bahy-wood-HPn} and
$\varphi = \beta_1 \oplus J\beta_1$.
In this case, the canonical lift of $\varphi$ defined by the extended solution \eqref{r2}
is the superhorizontal holomorphic map $\psi = (\beta_1, \varphi^{\perp}, J \beta_1):M \to F^J_{d_0,d_1}$ where
$d_0 = \rank\beta_1$, $d_1 = 2m-2d_0$.
 
(ii) In the case that $\beta_1$ has rank one, we have
$\varphi = \beta_1 \oplus J\beta_1:M \to \HP{m-1}$.  We may identify  
$F^J_{1,2m-2}$ with $\CP{2m-1}$ via the map $(\psi_0,\psi_1,\psi_2)\mapsto\psi_0$. With this identification,
 $\pi_e^{\rn}:\CP{2m-1}\to\HP{m-1}$ is the standard Riemannian fibration which maps
$L\in\CP{2m-1}$ to $L\oplus JL\in\HP{m-1}$, and the canonical lift of $\varphi$ gives the superhorizontal holomorphic map
$\beta_1 \cong(\beta_1,\varphi^{\perp},J\beta_1):M\to\CP{2m-1}\cong F^J_{1,2m-2}$. 

(iii) Let $h:M\to\CP{2m-1}$ be a full totally $J$-isotropic map. Then
$JG^{(m-1)}(h)=G^{(m)}(h)$, and the harmonic map
$\varphi^{\perp} = G^{(m-1)}(h)\oplus G^{(m)}(h):M\to\HP{m-1}$ is called a
\emph{quaternionic Frenet pair} \cite{bahy-wood-HPn}. As in Example \ref{ex:r2}(iii),
set $\beta_1=h_{(m-2)}$ and $\beta_2=h_{(m)}$. Then $J\beta_2^{\perp} = \beta_1$ and  the canonical lift of $\varphi$ defined by the extended solution
\eqref{r2} is the superhorizontal holomorphic map
 $(\beta_1,\varphi^{\perp},J\beta_1):M \to F^J_{m-1,2}$.
Since $\varphi$ is strongly conformal,  $\varphi^{\perp}$ also has a
(unique) twistor lift as in Example \ref{ex:str-conf}(iv), namely the $J_2$-holomorphic map
 $G^{(m)}(h)\cong \bigl(G^{(m)}(h),\varphi,G^{(m-1)}(h)\bigr)
 = \bigl(G''(\varphi),\varphi,G'(\varphi)\bigr):M \to \CP{2m-1}\cong F^J_{1,2m-2}$\,.

In contrast to harmonic maps from the 2-sphere to real and complex projective spaces, harmonic maps from the 2-sphere to quaternionic projective spaces are harder to describe: see \cite{bahy-wood-HPn} for a method of reduction to Frenet and mixed pairs, and see
\cite{pacheco-sympl,unitons} for uniton factorizations; however, there is one important class that we can completely describe, we turn to that class now. 
\end{example}

\subsection{Inclusive harmonic maps into quaternionic K\"ahler manifolds}
\label{subsec:inclusive}

Recall \cite{salamon-quat} that a \emph{quaternionic K\"ahler manifold $N^{4n}$} is a real oriented $4n$-dimensional Riemannian
manifold whose holonomy belongs to the subgroup $\Sp{n}\Sp{1}$ of $\SO{4n}$.
 Such a manifold has a natural
$\CP{1}$-bundle $Q \to N$ whose fibre at a point $q \in N$ consists of
all the almost Hermitian structures (i.e.\ orthogonal complex structures) on $T_qN$ which are `compatible' with the $\Sp{n}\Sp1$ structure; we will call $Q$ the
\emph{twistor space of $N$ in the quaternionic sense}.   Any oriented Riemannian $4$-manifold satisfies the above definition with
$Q \to N$ equal to the bundle $Z^+ \to N$ of positive almost Hermitian structures,
but, as this dimension is exceptional, most authors insist that $n \geq 2$ in the definition
of `quaternionic K\"ahler'.

We call a subspace of $TN$ \emph{quaternionic} if it is closed under $Q$.
A weakly conformal map $\varphi:M \to N$ from a Riemann surface to a quaternionic K\"ahler manifold
is called \emph{inclusive} \cite{salamon, eells-salamon} if, for each $p \in M$, $\d\varphi(T_pM)$ is contained in a
$4$-real-dimensional quaternionic subspace $S_p$ of $T_{\varphi(p)}N$. 
This is equivalent to saying that, for each $p \in M$, there is a $q_p \in Q_{\varphi(p)}$ with respect to which $\varphi$ is holomorphic, i.e., its differential intertwines the complex structure on $T_pM$ with $q_p$. 
 
If $\d\varphi(p)$ is non-zero, $\d\varphi(\pa/\pa z)$ spans an isotropic subspace;
$S_p$ and $q_p$ are determined uniquely by that.
 If $\varphi$ is harmonic, then $\d\varphi(\pa/\pa z)$ is holomorphic with respect to the Koszul--Malgrange structure on $\varphi^{-1}T^cN$ (see, for example,
\cite[Chapter 2]{burstall-rawnsley}); as usual, we can fill out zeros to extend the span of $\d\varphi(\pa/\pa z)$, and so $S$ and $q$, smoothly across the zeros of $\d\varphi$. 
Thus a (weakly conformal) inclusive harmonic  map has a twistor lift $\psi:M \to Q$;
Eells and Salamon \cite{eells-salamon} showed that this lift is $J_2$-holomorphic,
establishing that \emph{there is a one-to-one correspondence between inclusive weakly conformal
harmonic maps $\varphi:M \to N$ and $J_2$-holomorphic maps $\psi:M \to Q$
which project to $\varphi$}. We identify this correspondence for the
three quaternionic K\"ahler manifolds: the Grassmannians $\wt G_4(\rn^n)$, $G_2(\cn^n)$ and quaternionic projective space $\HP{m-1}$.

(i) First, we consider the real Grassmannian
$N = \wt G_4(\rn^n) = \SO{n}\big/  \SO{4} \times \SO{n-4}$ of oriented
$4$-dimensional subspaces of $\rn^n$.
Taking the orthogonal complement of a subspace identifies this with $\wt G_{n-4}(\rn^n)$; then, for each $Y \in N$, $T_YN$ can be identified with the space
$\Hom_{\rn}(Y^{\perp},Y)$ of real linear maps.  Set
$Q_Y^+$ (resp.\ $Q_Y^-$) equal to the set of almost complex structures on $T_YN$ given by postcomposition of an element of $\Hom(Y^{\perp},Y)$ with a positive (resp.\ negative) almost Hermitian structure on $Y$.  Then the bundle $Q^+ \to N$ is the twistor space of $N$ in the quaternionic sense; to see $Q^- \to N$ as a quaternionic K\"ahler structure, we must put the other orientation on $N$ or proceed as follows.
 
We may identify $Q^+_Y$ (resp.\ $Q_Y^-$) with the space of maximal positive (resp.\ negative) isotropic subspaces of $Y$ by associating to $q \in Q$
its $(0,1)$-space $V$; thus the bundle $Q^+ \to N$ can be identified with $\pi_e^{\rn}:F^{\rn}_{2, n-4} \to N$. Let $A:\wt G_{n-4}(\rn^n) \to \wt G_{n-4}(\rn^n)$ be the map
 which sends each subspace to the same subspace with the opposite orientation: thus for $n=5$,
 $A:\wt G_1(\rn^5)= S^4 \to S^4$ is the antipodal map $A(x) = -x$; then $Q^- \to N$ can be identified with the bundle $A \circ \pi_e^{\rn}:F^{\rn}_{2, n-4} \to N$.

Let $i:\wt G_{n-4}(\rn^n) \to G_{n-4}(\rn^n) \hookrightarrow G_{n-4}(\cn^n)$ be the
canonical  immersion (see \S \ref{subsec:twistor-real}).
 We have the following result.

\begin{proposition} \label{pr:G4Rn} Let $\varphi:M \to \wt G_4(\rn^n)$ be a non-constant weakly conformal harmonic.  Then either $\varphi$ or $A\circ\varphi$ is inclusive if and only if\/ $i \circ \varphi^{\perp}$ is strongly conformal.    In this case $\varphi$ or $A \circ\varphi$
has a $J_2$-holomorphic lift $\psi: M \to Q^{\pm} = F^{\rn}_{2, n-4}$.
\end{proposition}

\begin{proof}
By definition, $i \circ \varphi^{\perp}$ strongly conformal means that,
at each point $p \in M$, the Gauss transforms
$G'(i \circ \varphi^{\perp})$ and $G''(i \circ \varphi^{\perp})$ are orthogonal.
Now, under the inclusion map $i$, $\d\varphi_p(\pa/\pa z)$ and $\d\varphi_p(\pa/\pa\zbar)$ map to
$A'_{i \circ \varphi^{\perp}}$ and $A''_{i \circ \varphi^{\perp}}$, respectively,
so that strong conformality of $i \circ \varphi^{\perp}$ is equivalent to
the image of $\d\varphi_p(\pa/\pa z)$ being an isotropic subspace of $\varphi(p) \times \cn$ of dimension one or two.  By filling out zeros we obtain the isotropic image subbundle
$\Ima\d\varphi(\pa/\pa z)$ of rank one or two.

If $\Ima\d\varphi(\pa/\pa z)$ is of rank one, then there are precisely two isotropic subbundles of $\varphi(p) \times \cn$ of rank two containing it, giving two almost Hermitian structures $q$ at each point, one positive and one negative.  The positive one gives a lift of $\varphi$
into $Q^+$; the negative one gives a lift into $Q^-$, equivalently,
of $A \circ \varphi$ into $Q^+$.

If $\Ima\d\varphi(\pa/\pa z)$ is of rank two, then it defines a positive or negative almost Hermitian structure at each point, giving a lift into either $Q^+$ or $Q^-$; these are $J_2$-holomorphic, as before.
\end{proof}

(ii) We next consider the complex Grassmannian $G_2(\cn^n)$.  On identifying $\cn^n$ with $\rn^{2n}$,  $G_2(\cn^n)$
can be considered as the totally geodesic submanifold of $G_4(\rn^{2n})$ given by
$\{Y \in G_4(\rn^{2n}) : Y \text{ is complex}\}$.  We give $G_2(\cn^n)$ the conjugate of its canonical complex structure, i.e., that inherited from $G_{n-2}(\cn^n)$ by the 
identification $G_2(\cn^n) \to G_{n-2}(\cn^n)$ given by $Y \mapsto Y^{\perp}$.
Then the complexified tangent space to 
$G_2(\cn^n)$ at any $Y$ can be identified with
$\Hom_{\cn}(Y^{\perp},Y) \oplus \Hom_{\cn}(Y,Y^{\perp})$, where the summands are the
$(1,0)$- and $(0,1)$-tangent spaces for this complex structure on
$G_2(\cn^n)$.  
The manifold $G_2(\cn^n)$ has a quaternionic K\"ahler structure with quaternionic twistor space
 $\pi_e:F_{1,n-2,1} \to G_2(\cn^n)$.   The almost Hermitian structure $q_Y$ on $T_Y G_2(\cn^n)$ 
 corresponding to $(V,Y^{\perp}, W) \in F_{1,n-2,1}$ is that determined by the subspace $V$ of $Y$ (or its orthogonal complement $W$ in $Y$), explicitly, the $(1,0)$-space of $q_Y$ is
$\Hom_{\cn}(Y^{\perp},V) \oplus \Hom_{\cn}(W,Y^{\perp})$ (cf.\ \cite[p.\ 125(iii)]{rawnsley}).    We have an embedding
of $F_{1,n-2,1}$ in $F^{\rn}_{2, 2n-4}$, covering the inclusion of $G_2(\cn^n)$ in
$G_4(\rn^{2n})$, given by  $(V,Y^{\perp},W) \mapsto (V + \ov{W},Y^{\perp},\ov{V}+W)$.

Now, a map $\varphi^{\perp}:M \to G_{n-2}(\cn^n)$ is strongly conformal if and only if its composition with the inclusion mappings
$G_{n-2}(\cn^n) \hookrightarrow G_{2n-4}(\rn^{2n}) \hookrightarrow  G_{2n-4}(\cn^{2n})$ is strongly conformal, so we obtain the following: the first part is due to Rawnsley
\cite[\S 5C]{burstall-wood}; uniqueness comes from Example \ref{ex:str-conf}.

\begin{corollary} \label{co:G2Cn}
 Let $\varphi:M \to G_2(\cn^{n})$ be non-constant and weakly conformal.  Then $\varphi$ is inclusive if and only if\/
$\varphi^{\perp}$ is strongly conformal.    In that case $\varphi$ has a unique $J_2$-holomorphic lift
$\psi: M \to Q = F_{1,n-2,1}$ given by
$$
\psi = \bigl(G''(\varphi^{\perp}), \varphi^{\perp}, G'(\varphi^{\perp})\bigr).
$$
On including $G_2(\cn^n)$ in $G_4(\rn^{2n})$ and $F_{1,n-2,1}$ in $F^{\rn}_{2, 2n-4}$, this coincides with the lift given by Proposition \ref{pr:G4Rn}.
\end{corollary}

(iii) Lastly, we consider the quaternionic projective space $\HP{m-1}$. On identifying $\hn^m$ with $\cn^{2m}$, the space $\HP{m-1}$
can be thought of as the totally geodesic submanifold of $G_2(\cn^{2m})$
given by $\{Y \in G_2(\cn^{2m}) : Y \text{ is quaternionic}\}$.  The quaternionic 
 K\"ahler structure on $G_2(\cn^{2m})$ has quaternionic twistor space 
$Q = \pi_e^J: F^J_{1, 2m-2} \to \HP{m-1}$; 
as in Example \ref{ex:cpn-quat}(ii), this is the standard fibration
$\CP{2m-1} \to \HP{m-1}$. 
We have an embedding
of $F^J_{1, 2m-2}$ in $F_{1,2m-2,1}$, covering the inclusion of $\HP{m-1}$ in
$G_2(\cn^{2m})$, given by $F^J_{1, 2m-2} \cong \CP{2m-1} \ni V  \mapsto \bigl(V, (V+JV)^{\perp},JV \bigr) \in F_{1,2m-2,1}$. 

Using $G''(\varphi) = JG'(\varphi)$,
we deduce the following from Corollary \ref{co:G2Cn}, cf.\ \cite[Proposition 5.7]{burstall-wood}.

\begin{corollary}
Let $\varphi:M \to \HP{m-1}$ be non-constant and weakly conformal.
Then $\varphi$ is inclusive if and only if it is \emph{reducible}, i.e., 
its Gauss transform $G'(\varphi)$ has rank one.
In this case, $G'(\varphi^{\perp})$ also has rank one and
$\varphi$ has a unique $J_2$-holomorphic lift to $F^J_{1, 2m-2}$
given by $\psi = (J G'(\varphi^{\perp}), \varphi^{\perp}, G'(\varphi^{\perp}) )$.
On including $F^J_{1, 2m-2}$ in $F_{1,2m-2,1}$, this agrees with the lift given in Corollary \ref{co:G2Cn}.
\end{corollary}

\section{Explicit formulae for twistor lifts} \label{sec:explicit}
\subsection{Formulae from the Grassmannian model} \label{subsec:formulae-Grass-model} 

Corollary \ref{co:lift-from-unitons} describes how twistor lifts which are $J_2$-holomorphic
can be obtained from (partial) uniton factorizations with basic unitons.
We show how to find explicit formulae for these lifts. 

First of all, we recall from \cite{unitons} how to find explicit formulae for all
polynomial extended solutions, and thus for all harmonic maps of finite uniton number from a surface to $\U{n}$, in terms of arbitrary holomorphic data.

We need the following construction of M.~A.\ Guest \cite{guest-update}:
Let $r \in \nn$ and let $X$ be an arbitrary holomorphic subbundle of
$(\CC^{rn},\pa_{\zbar})$.
Set $W$ equal to the following subbundle of $\HH_+$:
\begin{equation} \label{W}
W = X + \lambda X_{(1)} + \lambda^2 X_{(2)} + \cdots + \lambda^{r-1} X_{(r-1)}
	+ \lambda^r \HH_+\,, 
\end{equation}
where $X_{(i)}$ denotes the $i$th osculating space of $X$ as in Example \ref{ex:fact-images}.
Then $W$ is an extended solution satisfying \eqref{W-finite}, and all such $W$ are given this way (we can take $X = W$); we shall say that $X$
\emph{generates} $W$. 
We shall describe subbundles of a trivial holomorphic bundle $(\CC^N, \pa_{\zbar})$ ($N \in \nn$) by giving \emph{meromorphic spanning sets} for them, as in \cite[\S 4.1]{unitons}.
As in that paper, by the \emph{order $o(L)$} of a meromorphic section $L$ of $\HH_+$ we mean the least integer $i$ such that
$P_iL \neq 0$ (where $P_iL$ is the $i$th Fourier coefficient of $L$ as in \eqref{fourier-coeffs}); equivalently, $L = \lambda^{o(L)} \wh{L}$ for some $\wh{L} = \sum_{\ell \geq 0} \lambda^{\ell}\wh{L}_{\ell}$ with $\wh{L}_{0}$ non-zero.
  Let $\{L_j\}$ be a meromorphic spanning set for $X$.  Then a meromorphic spanning set
for $W \mod \lambda W$ is $\{\lambda^k (L_j)^{\! (k)} : 0 \leq k \leq r\}$.

Now choose a uniton or partial uniton factorization \eqref{lambda-filt} of $W$ satisfying
\eqref{F-basic}.   Thus, writing $W = \Phi\H_+$, the extended solution $\Phi$ has (partial) uniton factorization \eqref{Phi-fact} with unitons $\alpha_i$ which are basic for $\Phi_{i-1}$.
The corresponding $F$-filtration $(Y_i)$ is given by \eqref{Y-from-W} and the $A^{\varphi}_z$-filtration corresponding
to that is given by $Z_i = P_0 \circ \Phi^{-1}Y_i$. 
If $\{H^i_j\}$ is a meromorphic spanning set for $Y_i$ $\mod \lambda W$, then
$\{P_0 \circ \Phi^{-1}H^i_j\}$ is a meromorphic spanning set for $Z_i$; then a basis for the legs $\psi_i = Z_i \ominus Z_{i+1}$ can be found from that set by the Gram--Schmidt process.
If all the legs are non-zero, this gives the $J_2$-holomorphic twistor lift of $\varphi = \Phi_{-1}$ described in Corollary
\ref{co:lift-from-unitons}.

To calculate this explicitly, first, we need to find a meromorphic spanning set $\{H^i_j\}$
for each
$Y_i$ from a meromorphic spanning set for $W$; in the examples below, this is done
by finding a meromorphic spanning set for $W$ adapted to the filtration $(Y_i)$. 
 
Second, \emph{let $S^{r}_{s}$ denote the sum of all $r$-fold products
of the form $\Pi_r \cdots \Pi_1$  where exactly $s$ of the $\Pi_j$ are $\pi_{\alpha_j}^{\perp}$ and the other $r-s$ are $\pi_{\alpha_j}$}.
Then given a meromorphic spanning set $\{H^i_j\}$ for $Y_i$ $\mod \lambda W$,
the meromorphic spanning set $\{P_0 \circ \Phi^{-1}H^i_j\}$ for $Z_i$ is given by
$P_0 \circ \Phi^{-1}H^i_j = \sum_{s=k}^r S^r_s P_s(H^i_j)$. 
We now see how this works for our two main examples; other examples can be done similarly.

\begin{example}
We find explicit formulae for the canonical twistor lift of the harmonic map
$\varphi = \Phi_{-1}$ defined by a normalized extended solution $W = \Phi\H_+$\,.
Let $X$ generate $W$ and let $\{L_j\}$ be a meromorphic spanning set for $X$.
The filtration $(Y_i)$ which gives the canonical twistor lift is given by \eqref{can-filt}, and a meromorphic spanning set for $Y_i \mod \lambda W$  is $\{\lambda^k (L_j)^{\! (k)} : i \leq o(L_j)+k \leq r \}$.

Let $\alpha_1,\ldots, \alpha_r$ be the Uhlenbeck unitons of $\Phi$, see Example \ref{ex:can-filt2}; then a
meromorphic spanning set for the corresponding $A^{\varphi}_z$-filtration $Z_i = P_0 \circ \Phi^{-1}Y_i$ is given by
\begin{equation*}
Z_i 	= \spa\bigl\{ \sum_{s=k}^r S^r_s P_{s-k}(L_j)^{\! (k)}\
	:\ i \leq o(L_j)+k \leq r \bigr\}.
\end{equation*} 
Applying the Gram--Schmidt process gives explicit formulae for the canonical twistor lift $\psi$ of $\varphi = \Phi_{-1}$ defined by $\Phi$.
\end{example}

\begin{example}
Suppose that $\varphi:M \to G_*(\cn^n)$ is harmonic map of finite uniton number.
We find explicit formulae for Burstall's twistor lift (Example \ref{ex:burstall-J2})
of $\varphi$.
As in Lemma \ref{le:Q} 
we can find a $\nu$-invariant polynomial extended solution $\Phi$ with $\Phi_{-1} = \varphi$.
Set $W = \Phi\H_+$ as usual.
Let $X$ generate $W$ and let $\{L_j\}$ be a meromorphic spanning set for $X$.
 The relevant filtration $(Y_i)$ is described in
Example \ref{ex:nilconf-Y};
a meromorphic spanning set for $Y_i \mod \lambda Y_i$ is given by
$\{\lambda^{k} (L_j)^{\! (k)} : i \leq k \leq r\}$.

Let $\alpha_1,\ldots, \alpha_r$ be the unitons of $\Phi$ in the uniton factorization by $A_z$-images, see Example \ref{ex:fact-images}. 
Then a meromorphic spanning set for the corresponding $A^{\varphi}_z$-filtration $Z_i = P_0 \circ \Phi^{-1}Y_i$ is given by 
\begin{equation*}
Z_i 	= \spa\bigl\{ \sum_{s=k}^r S^r_s P_{s-k}(L_j)^{\! (k)}\ :\ i \leq k \leq r \bigr\}.
\end{equation*}
Applying the Gram--Schmidt process gives explicit formulae for the twistor lift $\psi$
of $\pm\varphi$ defined in Example \ref{ex:burstall-J2}.
\end{example}

Note that, (i) in both cases, the harmonic map $\varphi$ is given (a) as the product of unitons by the formulae in \cite[\S 4.1]{unitons},
(b) as the sum $\sum\psi_{2j}$ of the even-numbered legs of $\psi$;
(ii) we can generate \emph{all} harmonic maps $M \to G_*(\cn^n)$ of finite uniton number, and the above twistor lifts of them, by \emph{freely} choosing meromorphic functions $L_j: M \to \CC^n$ and computing the lifts as above, giving \emph{completely explicit formulae for all harmonic maps of finite uniton number from a surface to a complex Grassmannian and their twistor lifts}.

\subsection{Examples} \label{subsec:exs}

In the following examples, as before, for any $i \in \nn$ and
holomorphic map $f:M \to G_*(\cn^n)$, equivalently, holomorphic subbundle of $\CC^n$, we denote by $f_{(i)}$
 the $i$th osculating space spanned by derivatives of local holomorphic sections of $f$ up to order $i$. 

\begin{example} \label{ex:r3n4}
Let $\varphi$ be a harmonic map from a Riemann surface to $G_2(\cn^4)$ of (minimal) uniton number $3$.  We shall find
a $J_2$-holomorphic twistor lift of $\pm\varphi$.
By \cite[Corollary 5.7]{unitons}, there is a polynomial extended solution $\Phi$ of degree $3$
with $\Phi_{-1} = \pm\varphi$.  On replacing $\varphi$ by its orthogonal complement if necessary, we may
assume that  $\Phi_{-1} = \varphi$.
Set $W = \Phi\H_+$\,.  This is closed under $\nu$ and so has the form
\be{W-G2C4}
W=\spa\{H_0+\lambda^2 H_2\} + \lambda\delta_2  +  \lambda^2 \delta_3 +  \lambda^3\HH_+
\ee
where $H_0, H_2:M \to \cn^4$ are meromorphic maps (equivalently, meromorphic sections of the trivial bundle $\CC^4$), the
$\delta_i$ are subbundles of $\CC^4$, and
setting $\delta_1 = \spa\{H_0\}$, we have $(\delta_1)_{(1)} \subset \delta_2$ and $(\delta_2)_{(1)} \subset \delta_3$.

Now, none of the $\delta_i$ is constant, otherwise, $(\pi_{\delta_i} + \lambda^{-1}\pi_{\delta_i})\Phi$ would be a polynomial extended solution of degree $2$ associated to $\varphi$, contradicting the definition of uniton number.  It follows that 
$\delta_1$ is full,
$\delta_i = (H_0)_{(i-1)}$ and $\delta_i$ has rank $i$ \ $(i=1,2,3)$.

All the filtrations of Examples
\ref{ex:burstall-J2}, \ref{ex:burstall-dual} and \ref{ex:str-conf} have length $3$ and agree; further,
$\rank Z_i = 4-i$ \ $(i=0,1,2,3,4)$, so that the legs $\psi_i = Z_i \ominus Z_{i+1}$
are of rank one.  With notation as in those examples, since $Z_1 \neq \CC^4$, either
$U_1 \neq \varphi$ or $V_1 \neq \varphi^{\perp}$; assume without loss of generality the former.  Then we obtain the diagram \eqref{diag:G24},
with $Z_i = \sum_{j=i+1}^4 \psi_j$, $\varphi = \psi_0 \oplus \psi_2$ and $\varphi^{\perp} = \psi_1 \oplus \psi_3$.
As before the arrows show the possible non-zero second fundamental forms $A'_{\psi_i,\,\psi_j}$. 
The second diagram is identical to the first, but is drawn to be in keeping with our earlier diagrams, with `across' arrows going down.

\begin{equation}
\begin{gathered}\label{diag:G24}
\xymatrixrowsep{1.8pc}\xymatrix{
\psi_0  \ar[rd] \ar[r] &\psi_1 \ar[ld] \\
\psi_2 \ar[r]  \ar[u] &\psi_3 \ar[u]}
\qquad \qquad \qquad
\xymatrixrowsep{0.3pc}\xymatrixcolsep{1.5pc}\xymatrix{ 
\psi_0 \ar[rd] \ar[rddd]  & \\
&  \psi_1 \ar[ld] \\
\psi_2 \ar[rd] \ar[uu] &\\
& \psi_3 \ar[uu]}
\end{gathered}
\end{equation}

\medskip

As in Example \ref{ex:segal}, one factorization of $\Phi$ is provided by its Segal unitons, which we shall denote by
$\beta_i$.  According to
\cite[Example 4.9(i)]{ferreira-simoes-wood}, these are given by 
$\beta_1 = h$, $\beta_2 = h_{(1)}$ and
\be{beta_3}
\beta_3 = \spa\{H_0 + \pi_{h_{(1)}}^{\perp}H_2\} \oplus G^{(1)}(h) \oplus G^{(2)}(h)
=\spa\{H_0 + \pi_{h_{(2)}}^{\perp}H_2\} \oplus G^{(1)}(h) \oplus G^{(2)}(h).
\ee
where $G^{(i)}(h)$ denotes the $i$th $\pa'$-Gauss transform of $h$
 (see \S \ref{subsec:Grass}).  Hence
\be{phi-r3n4} 
\varphi = \spa\{H_0 + \pi_{h_{(1)}}^{\perp}H_2\} \oplus G^{(2)}(h)=\spa\{H_0 + \pi_{h_{(2)}}^{\perp}H_2\} \oplus G^{(2)}(h).
\ee
 
Another factorization of $\Phi$ is provided by the Uhlenbeck unitons, see
Example \ref{ex:can-filt2}; we shall denote these by $\gamma_i$.
By the formulae in \cite[Example 4.6]{unitons}, they are given by
$\gamma_1 = h_{(2)}$, $\gamma_2 = {h_{(1)}}$, and 
$\gamma_3 = \spa\{H_0 + \pi_{h_{(2)}}^{\perp}H_2 \}$.

{}From Corollary \ref{co:lift-from-unitons}, we have $\psi_0 = \gamma_3$ and
$\psi_1= \gamma_3^{\perp} \cap \gamma_2 = G^{(1)}(h)$.  It follows from \eqref{phi-r3n4} that
$\psi_2 = G^{(2)}(h)$.
{}From \eqref{beta_3}, we have $\beta_3=\psi_0 \oplus \psi_1 \oplus \psi_2$ so that
$\psi_3=\beta_3^\perp$.   Thus we obtain the $J_2$-holomorphic twistor lift
$\psi = (\gamma_3\,,\, G^{(1)}(h) \,,\, G^{(2)}(h) \,,\, \beta_3^\perp):M \to F_{1,1,1,1}$ of $\varphi$.

{}From the diagram we see that $\varphi^{\perp} = \psi_1 \oplus \psi_3$  is the sum of the harmonic map $\psi_1  = G^{(1)}(h)$ and the antiholomorphic subbundle $\psi_3$ of $\{\psi_1 \oplus G^{(1)}(\psi_1)\}^{\perp}  =
\psi_0 \oplus \psi_3$, in accordance with J. Ramanathan's description
\cite{ramanathan}.
\end{example}

We finish with two examples: the first one real and the second one symplectic.

\begin{example}  Let $H_0,H_1,H_2,H_3:M\to\cn^n$ be meromorphic maps, set $\delta_1=\spa\{H_0,H_2\}$,
and consider the $\nu$-invariant extended solution $W = \Phi\H_+$ given by
$$
W=\spa\{H_0+\lambda^2H_1,H_2+\lambda^2H_3\}+\lambda\delta_2+\lambda^2\delta_3+\lambda^3\HH_+\,,
$$
where $\ul{0}\subset\delta_1\subset\delta_2\subset\delta_3\subset\CC^n$ is a superhorizontal sequence.
By the formulae in \cite[Example 4.6]{unitons}, we calculate the Uhlenbeck unitons as
\be{uhl1}
\gamma_1=\delta_3\,, \quad \gamma_2=\delta_2\,,  \quad
\gamma_3=\spa\{H_0+\pi_{\delta_3}^{\perp}H_1, H_2+\pi_{\delta_3}^{\perp}H_3\}\,,
\ee
and the corresponding
harmonic map $\varphi = \Phi_{-1}:M \to G_*(\cn^n)$ is given by 
$
\varphi=\gamma_3 \oplus (\delta_2^\perp\cap\delta_3).
$
{}From Proposition \ref{pr:Y-from-W} we calculate the canonical lift
$(\psi_0,\psi_1,\psi_2,\psi_3):M \to\  F$ as follows: 
$\psi_0=Z_1^{\perp}=\gamma_3$ and $\psi_1=\gamma_3^{\perp}\cap\gamma_2=\delta_1^{\perp}\cap\delta_2$; since $\varphi =\psi_0 \oplus \psi_2$, this gives $\psi_2=\delta_2^\perp\cap\delta_3$. Finally, $\psi_3=Z_3=\pi_{\gamma_3}^{\perp}(\delta_3^{\perp})$. 

Now, as in \cite[\S 6.6]{unitons}, we can choose the data $H_i$, $\delta_2$, $\delta_3$ such that $W$ is real of degree $3$.  Then,
by Proposition \ref{pr:realW-Y} we have $\ov\psi_0=\psi_3$ and $\ov\psi_1=\psi_2$; thus
$n$ is even, say $n = 2m$, and  $\varphi=\psi_0\oplus\psi_2 = \psi_0 \oplus \ov\psi_1$ defines a harmonic map from $M$ to $\O{2m}/\U m$. The canonical twistor lift of $\varphi$ defined by $\Phi$ is the $J_2$-holomorphic map
$
(\psi_0,\psi_1,\ov\psi_1,\ov\psi_0):M\to\ \O{2m} \big/ \U 2\times\U{m-2}
$
where $\psi_0 = \gamma_3$ and $\psi_1 = \gamma_3^{\perp}\cap\delta_2$ with $\gamma_3$ is given by \eqref{uhl1}.
\end{example}

\begin{example} Let $H_0,H_2,H_3:M\to\cn^6$ be meromorphic maps with $\spa\{H_0\}:M \to \CP{5}$ full, and consider the $\nu$-invariant extended solution $W = \Phi\H_+$ given by
$$
W=\spa\{H_0+\lambda^2H_2\}+
\lambda\spa\bigl\{(H_0)_{(1)}, H_3\bigr\}
+\lambda^2 \spa\bigl\{(H_0)_{(2)}, (H_3)_{(1)}\bigr\} 	+\lambda^3\HH_+.
$$
Set 
$$
\delta_1=\spa\{H_0\},\quad \delta_2=(H_0)_{(1)}+\spa\{H_3\},\quad \delta_3=(H_0)_{(2)}+(H_3)_{(1)}.
$$
A simple calculation shows that the Uhlenbeck unitons of $W$ are given by
$\gamma_1=\delta_3$, $\gamma_2=\delta_2$ and
$\gamma_3=\spa\{H_0+\pi_{\delta_3}^{\perp}H_3\}$. The corresponding harmonic map
$\varphi = \Phi_{-1}$ of minimal uniton number $3$ is given by 
$$
\varphi=\gamma_3\oplus (\delta_2^\perp\cap\delta_3):M\to G_3(\cn^6).
$$
{}From Proposition \ref{pr:Y-from-W}, the legs of the canonical $A^\varphi_z$-filtration
defined by $\Phi$ are given by $\psi_0=\gamma_3$, 
$\psi_1=\delta_2\ominus\delta_1$, $\psi_2=\delta_3\ominus\delta_2$ and $\psi_3=\pi_{\gamma_3}^{\perp}\delta_3^\perp$.
The canonical twistor lift of $\varphi$ defined by $\Phi$ is the  $J_2$-holomorphic map
$\psi=(\psi_0,\psi_1,\psi_2,\psi_3):M\to F_{1,2,2,1}$.

As in \cite[Example 6.31]{unitons} we see that $W$ is symplectic of degree $3$ if $H_0$ is totally $J$-isotropic and $H_3$
is a section of $(H_0)_{(3)}$. In this case, $\varphi$ is a harmonic map from $M$ to $\Sp3/\U 3$, and $J\psi_0=\psi_3$, $J\psi_1=\psi_2$.
Then the canonical twistor lift of $\varphi: M \to \Sp3/\U 3$ defined by $\Phi$
is the $J_2$-holomorphic map  $\psi:M \to Z^J_{1,2} =  \Sp3/\U1\times\U2$ given explicitly by
$\psi = (\psi_0, \psi_1, J\psi_1, J\psi_0)$ where $\psi_0 = \spa\{H_0+\pi_{\delta_3}^{\perp}H_3\}$ and
$\psi_1 = G'(\delta_1) + \spa\{H_3\}$.

\end{example}

\end{document}